\documentclass[3p,times]{amsart}

\usepackage{amssymb}
\usepackage{amsthm}

\usepackage{enumerate} 

\usepackage{amsmath}
\usepackage{pdfpages}

\newtheorem{theorem}{Theorem}[section]
\newtheorem{lemma}[theorem]{Lemma}
\newtheorem{corollary}[theorem]{Corollary}
\newtheorem{proposition}[theorem]{Proposition}

\newtheorem{main-theorem}[theorem]{Theorem}
\newtheorem*{main-theorem*}{Main Theorem}
\newtheorem*{triangulation-theorem*}{Triangulation Theorem}
\newtheorem{corollary*}{Corollary}

\newtheorem*{problem*}{Problem}

\theoremstyle{definition}

\newtheorem*{question*}{Question}
\newtheorem*{notation*}{Notation}
\newtheorem{notation}[theorem]{Notation}
\newtheorem{remark}[theorem]{Remark}
\newtheorem{definition}[theorem]{Definition}

\newtheorem{example}[theorem]{Example}

\renewcommand{\mod}{\operatorname{mod}}

\newcommand{\soc}{\operatorname{soc}}

\newcommand{\umod}{\operatorname{\underline{mod}}}

\newcommand{\Hom}{\operatorname{Hom}}
\newcommand{\Ker}{\operatorname{Ker}}
\newcommand{\T}{\operatorname{T}}
\newcommand{\Tr}{\operatorname{Tr}}

\newcommand{\op}{\operatorname{op}}

\newcommand{\id}{\operatorname{id}}
\newcommand{\CM}{\operatorname{CM}}
\newcommand{\uEnd}{\operatorname{\underline{End}}}

\newcommand{\intt}{\operatorname{int}}

\renewcommand{\Im}{\operatorname{Im}}

\newcommand{\bN}{\mathbb{N}}
\newcommand{\bP}{\mathbb{P}}

\newcommand{\bR}{\mathbb{R}}
\newcommand{\bS}{\mathbb{S}}
\newcommand{\bT}{\mathbb{T}}

\newcommand{\cF}{\mathcal{F}}

\newcommand{\cO}{\mathcal{O}}

\usepackage{mathrsfs}
\newcommand{\sC}{\mathscr{C}}

\newcommand{\overbar}[1]{\mkern 5mu\overline{\mkern-5mu#1\mkern-5mu}\mkern 5mu}

\usepackage{etex} 

\usepackage{dsfont}         
\usepackage[h]{esvect}

\usepackage[matrix,arrow,curve,frame,arc]{xy}

\usepackage{pgf}
\usepackage{tikz}
\usepackage{xcolor}

\newcommand{\ba}{\bar{\alpha}}
\newcommand{\vf}{\varphi}

\usetikzlibrary{arrows,shapes,automata,backgrounds,petri,calc}

\newcommand{\tikzAngleOfLine}{\tikz@AngleOfLine}
\def\tikz@AngleOfLine(#1)(#2)#3{%
\pgfmathanglebetweenpoints{%
\pgfpointanchor{#1}{center}}{%
\pgfpointanchor{#2}{center}}
\pgfmathsetmacro{#3}{\pgfmathresult}%
}

\begin{document}

\title{Algebras of generalized quaternion type}

{\def\thefootnote{}
\footnote{The research was supported by the research grant
DEC-2011/02/A/ST1/00216 of the National Science Center Poland.}
}

\author[K. Edrmann]{Karin Erdmann}
\address[Karin Erdmann]{Mathematical Institute,
   Oxford University,
   ROQ, Oxford OX2 6GG,
   United Kingdom}
\email{erdmann@maths.ox.ac.uk}

\author[A. Skowro\'nski]{Andrzej Skowro\'nski}
\address[Andrzej Skowro\'nski]{Faculty of Mathematics and Computer Science,
   Nicolaus Copernicus University,
   Chopina~12/18,
   87-100 Toru\'n,
   Poland}
\email{skowron@mat.uni.torun.pl}

\begin{abstract}
We introduce and study the algebras of generalized quaternion type,
being natural generalizations of algebras which occurred in the study
of blocks of group algebras with generalized quaternion defect groups.
We prove that all these algebras, with $2$-regular Gabriel quivers,
are periodic algebras of period $4$ and
very specific deformations of the weighted surface algebras
of triangulated surfaces with arbitrarily oriented triangles.
The main result of the paper forms an important step towards
the Morita equivalence classification of all periodic symmetric
tame algebras of non-polynomial growth.
Applying the main result, we establish existence of wild
periodic algebras of period 4, with arbitrary large number
(at least 4) of pairwise non-isomorphic simple modules.
These wild periodic algebras arise as stable endomorphism
rings of cluster tilting Cohen-Macaulay modules over
one-dimensional hypersurface singularities.

\bigskip



\end{abstract}


\maketitle

\addtocounter{section}{-1}
\section{Introduction and the main result}\label{sec:intro}

Throughout this paper, $K$ will denote a fixed algebraically closed field.
By an algebra we mean an associative finite-dimensional $K$-algebra
with an identity.
For an algebra $A$, we denote by $\mod A$ the category of
finite-dimensional right $A$-modules and by $D$ the standard
duality $\Hom_K(-,K)$ on $\mod A$.
An algebra $A$ is called \emph{self-injective}
if $A_A$ is injective in $\mod A$, or equivalently,
the projective modules in $\mod A$ are injective.
A prominent class of self-injective algebras is formed
by the \emph{symmetric algebras} $A$ for which there exists
an associative, non-degenerate symmetric $K$-bilinear form
$(-,-): A \times A \to K$.
Classical examples of symmetric algebras are provided
by the blocks of group algebras of finite groups and
the Hecke algebras of finite Coxeter groups.
In fact, any algebra $A$ is the quotient algebra
of its trivial extension algebra $\T(A) = A \ltimes D(A)$,
which is a symmetric algebra.
Two self-injective algebras $A$ and $\Lambda$ are said
to be \emph{socle equivalent} if the quotient algebras
$A/\soc(A)$ and $\Lambda/\soc(\Lambda)$ are isomorphic.

From the remarkable Tame and Wild Theorem of Drozd
(see \cite{CB,Dr})
the class of algebras over $K$ may be divided
into two disjoint classes.
The first class consists of the \emph{tame algebras}
for which the indecomposable modules occur in each dimension $d$
in a finite number of discrete and a finite number of one-parameter
families.
The second class is formed by the \emph{wild algebras}
whose representation theory comprises
the representation theories of all
algebras over $K$.
Accordingly, we may realistically hope to classify
the indecomposable finite-dimensional modules only
for the tame algebras.
Among the tame algebras we may distinguish the
\emph{representation-finite algebras}, having
only finitely many isomorphism classes of indecomposable
modules, for which the representation theory is rather
well understood.
On the other hand, the representation theory of arbitrary
tame algebras is still only emerging.
The most accessible ones
amongst the tame algebras are
\emph{algebras of polynomial growth} \cite{Sk0} for which the number
of one-parameter families
of indecomposable modules in each dimension $d$ is bounded by $d^m$,
for some positive integer $m$ (depending only on the algebra).

Let $A$ be an algebra.
Given a module $M$ in  $\mod A$, its \emph{syzygy}
is defined to be the kernel $\Omega_A(M)$ of a minimal
projective cover of $M$ in $\mod A$.
The syzygy operator $\Omega_A$ is a very important tool
to construct modules in $\mod A$ and relate them.
For $A$ self-injective, it induces an equivalence
of the stable module category $\umod A$,
and its inverse is the shift of a triangulated structure
on $\umod A$ \cite{Ha1}.
A module $M$ in $\mod A$ is said to be \emph{periodic}
if $\Omega_A^n(M) \cong M$ for some $n \geq 1$, and if so
the minimal such $n$ is called the \emph{period} of $M$.
The action of $\Omega_A$ on $\mod A$ can effect
the algebra structure of $A$.
For example, if all simple modules in $\mod A$ are periodic,
then $A$ is a self-injective algebra.
Sometimes one can even recover the algebra $A$ and its
module category from the action of $\Omega_A$.
For example, the self-injective Nakayama algebras
are precisely the algebras $A$ for which $\Omega_A^2$ permutes
the isomorphism classes of simple modules in $\mod A$.
An algebra $A$ is defined to be \emph{periodic} if it is periodic
viewed as a module over the enveloping algebra
$A^e = A^{\op} \otimes_K A$, or equivalently,
as an $A$-$A$-bimodule.
It is known that if $A$ is a periodic algebra of period $n$ then
for any indecomposable non-projective module $M$
in $\mod A$  the syzygy $\Omega_A^n(M)$ is isomorphic to $M$.

Finding or possibly classifying periodic algebras
is an important problem. It is very interesting because of
connections
with group theory, topology, singularity theory
and cluster algebras.
Periodic algebras have periodic Hochschild cohomology.
Periodicity of an algebra, and its period,
are invariant under derived equivalences
\cite{Ric2} (see also \cite{ES2}).
Therefore, to study periodic algebras
we may assume that the algebras are basic
and indecomposable.

We are concerned with the classification of all periodic tame
symmetric algebras.
In \cite{Du1} Dugas proved that every representation-finite
self-injective algebra, without simple blocks, is a periodic algebra.
We note that, by general theory (see \cite[Section~3]{Sk}),
a basic, indecomposable, non-simple, symmetric algebra $A$
is representation-finite if and only if $A$ is socle equivalent
to an algebra $\T(B)^G$ of invariants of the trivial extension
algebra $\T(B)$ of a tilted algebra $B$ of Dynkin type
with respect to free action of a finite cyclic group $G$.
The representation-infinite, indecomposable,
periodic algebras of polynomial growth were classified
by Bia\l kowski, Erdmann and Skowro\'nski in \cite{BES2}
(see also \cite{Sk1,Sk}).
In particular, it follows from \cite{BES2}
that every basic, indecomposable,
representation-infinite periodic symmetric tame algebra
of polynomial growth is socle
equivalent to an algebra $\T(B)^G$
of invariants of the trivial extension algebra $\T(B)$
of a tubular algebra $B$ of tubular type
$(2,2,2,2)$, $(3,3,3)$, $(2,4,4)$, $(2,3,6)$
(introduced by Ringel \cite{R}) with respect to free
action of a finite cyclic group $G$.

We also mention that there are natural classes
of periodic wild algebras.
For example, with the exception of few small
cases (see \cite{ES1}),
the preprojective algebras of Dynkin type \cite{AR,Bu,ESn},
or more generally, the deformed preprojective algebras
of generalized Dynkin type \cite{BES1,ES2},
have this property.

It would be interesting to classify
the Morita equivalence classes of
all indecomposable
periodic symmetric tame algebras of non-polynomial growth.
Such algebras appeared naturally in the study of blocks
of group algebras with generalized quaternion defect groups.
Motivated by known properties of  these blocks,
Erdmann introduced and investigated in
\cite{E1,E2,E3,E4} the
\emph{algebras of quaternion type},
being the indecomposable, representation-infinite tame
symmetric algebras $A$ with non-singular Cartan matrix
for which every indecomposable non-projective
module in $\mod A$ is periodic of period dividing $4$.
In particular, Erdmann proved that every algebra $A$
of quaternion type has at most $3$ non-isomorphic
simple modules and its basic algebra is isomorphic
to an algebra belonging to 12 families of symmetric
representation-infinite algebras defined by quivers
and relations.
Subsequently it has been proved in \cite{Ho} (see also
\cite{NeS} for the polynomial growth cases) that all these
algebras are tame, and are in fact periodic of period $4$
(see \cite{BES2,ES1}).
In particular it shows that a  finite group $G$
is periodic with respect to the group cohomology
$H^*(G,\mathbb{Z})$ if and only if all blocks
with non-trivial defect groups of its group algebras
$K G$  over an arbitrary algebraically closed
field $K$ are periodic algebras
(see \cite{ES3}).
By the famous result of Swan \cite{Sw}
 periodic groups can be characterized as the
finite groups acting freely on finite CW-complexes
homotopically equivalent to spheres
(see \cite[Section~4]{ES2} for more details).
Some of the algebras
of quaternion type occur as  endomorphism
algebras of cluster tilting objects in the stable
categories of maximal Cohen-Macaulay modules over
odd-dimensional isolated hypersurface singularities
(see \cite[Section~7]{BIKR}).

In this paper,
we introduce and investigate algebras
of generalized quaternion type.
Let $A$ be an algebra.
We say that $A$ is of \emph{generalized quaternion type}
if it satisfies the following conditions:
\begin{enumerate}[(1)]
 \item
  $A$ is symmetric, indecomposable, representation-infinite
  and tame.
 \item
  All simple modules in $\mod A$ are periodic of period $4$.
\end{enumerate}

We note that every algebra $A$ of generalized quaternion type
with the Grothendieck group $K_0(A)$ of rank at most two is
an algebra of quaternion type, and hence its basic algebra
is known (see \cite{E4}).
An algebra $A$ is called $2$-regular if its Gabriel
quiver is $2$-regular (see Section~\ref{sec:boundquiver}).

In \cite{ES4} and \cite{ES5} we introduced and investigated
weighted surface algebras of triangulated surfaces with
arbitrarily oriented triangles, and their deformations
(see Sections \ref{sec:wsalg} and \ref{sec:tetralg} for details).

The following theorem is the main result of the paper.

\begin{main-theorem*}
Let $A$ be a basic, indecomposable, $2$-regular algebra
over an algebraically closed field $K$,
with the Grothendieck group $K_0(A)$ of rank at least $3$.
Then the following statements are equivalent:
\begin{enumerate}[(i)]
 \item
  $A$ is of generalized quaternion type.
 \item
  $A$ is a symmetric tame periodic algebra of period $4$.
 \item
  $A$ is socle equivalent to a weighted surface algebra,
  different from a singular tetrahedral algebra,
  or is isomorphic to a non-singular higher tetrahedral algebra.
\end{enumerate}
Moreover, if $K$ is of characteristic different from $2$,
we may replace in \emph{(iii)} ``socle equivalent'' by ``isomorphic''.
\end{main-theorem*}

We note that the implication (ii) $\Rightarrow$ (i) follows from
general theory, and the implication (iii) $\Rightarrow$ (ii)
is a consequence of results established in \cite{ES4} and \cite{ES5}.
Hence the main aim of the paper is to provide a highly
non-trivial proof of the implication  (i) $\Rightarrow$ (iii).

We note that the $2$-regularity assumption
imposed on the algebra $A$ is essential for the
validity of the above theorem.
Namely, there are symmetric tame periodic algebras
of period $4$ which are derived equivalent but not
isomorphic to algebras occurring in (iii).
For example, the trivial extension algebra
$\T(C)$ of a canonical tubular algebra $C$
of type $(2,2,2,2)$ is derived equivalent to
a non-singular tetrahedral algebra but not socle
equivalent to a weighted surface algebra
(see \cite{ES4,HR}).
We also stress that the class of weighted surface
algebras occurring in (iii) contains as a subclass
the class of Jacobian algebras of quivers with
potentials associated to orientable surfaces with punctures
and empty boundary, whose quivers have no loops nor $2$-cycles
(see \cite{DWZ1,DWZ2,FST,GLS,LF,Lad2,VD} for details and relevant
results).

We would like to mention that all algebras occurring in (iii),
except non-singular tetrahedral algebras, are of non-polynomial
growth.
Moreover, those algebras with the Grothendieck group of rank
at least $4$ have singular Cartan matrix.
Further, the orbit closures
of algebras in (iii) in the affine varieties of associative $K$-algebra
structures contain a wide class of symmetric tame algebras,
being natural generalizations of algebras of dihedral and
semidihedral type, which occurred in the study of blocks
of group algebras with dihedral and semidihedral defect groups.
We refer to
\cite{ES6} and \cite{ES7}
for a classification of algebras of generalized dihedral type
and a characterization of Brauer graph algebras, using
biserial weighted surface algebras.

A prominent role in the proof of implication
(i) $\Rightarrow$ (iii)
of the Main Theorem is played by the following
consequence of
Theorems \ref{th:2.1}, \ref{th:5.1} and \ref{th:6.1}.

\begin{triangulation-theorem*}
Let $A$ be a basic, indecomposable, $2$-regular algebra
of generalized quaternion type
with the Grothendieck group $K_0(A)$ of rank at least $3$.
Then there is a bound quiver presentation $A = K Q / I$ of $A$
and a permutation $f$ of arrows of $Q$ such that
the following statements hold.
\begin{enumerate}[(i)]
 \item
  $(Q,f)$ is the triangulation quiver
  $(Q(S,\vec{T}),f)$ associated to a directed
  triangulated surface $(S,\vec{T})$.
 \item
  For each arrow $\alpha$ of $Q$,
  $\alpha f(\alpha)$ occurs in a minimal relation of $I$.
\end{enumerate}
\end{triangulation-theorem*}

An important combinatorial and homological invariant
of the module category $\mod A$ of an algebra $A$
is its Auslander-Reiten quiver $\Gamma_A$.
Recall that $\Gamma_A$ is the translation quiver whose
vertices are the isomorphism classes of indecomposable
modules in $\mod A$, the arrows correspond
to irreducible homomorphisms, and the translation
is the Auslander-Reiten translation $\tau_A = D \Tr$.
For $A$ self-injective, we denote by $\Gamma_A^s$
the stable Auslander-Reiten quiver of $A$, obtained
from $\Gamma_A$ by removing the isomorphism classes
of projective modules and the arrows attached to them.
By a stable tube we mean a translation quiver $\Gamma$
of the form $\mathbb{Z} \mathbb{A}_{\infty}/(\tau^r)$,
for some $r \geq 1$, and call $r$ the rank of $\Gamma$.
By general theory, an infinite connected component $\sC$
of $\Gamma_A^s$ is a stable tube if and only if $\sC$
consists of $\tau_A$-periodic modules 
(see \cite[Section~IX.4]{SY2}).
We note also that, for a symmetric algebra $A$, we have
$\tau_A = \Omega_A^2$ (see \cite[Corollary~IV.8.6]{SY}).

We have the following direct consequence of the Main Theorem.

\begin{corollary*}
\label{cor:main1}
Let $A$ be a $2$-regular algebra of generalized quaternion type.
Then the stable Auslander-Reiten quiver $\Gamma_A^s$
of $A$ consists of stable tubes of ranks $1$ and $2$.
\end{corollary*}

The next corollary is a consequence of results
proved in \cite{BIKR,Du2} and the Triangulation Theorem,
and solves the open problem concerning the existence
of wild periodic algebras of period $4$.

\begin{corollary*}
\label{cor:main2}
Let $n \geq 4$ be a natural number
and $K$ an algebraically closed field.
Then there exist a one-dimensional hypersurface singularity $R$
over $K$
and a cluster tilting module $T$ in the category $\CM(R)$
of maximal Cohen-Macaulay modules over $R$
such that for the stable endomorphism algebra $B = \uEnd_R(T)$
the following statements hold.
\begin{enumerate}[(i)]
  \item
    $B$ is a basic, indecomposable, finite-dimensional symmetric
    $2$-regular $K$-algebra, with the Grothendieck group $K_0(B)$
    of rank $n$.
  \item
    $B$ is a wild algebra.
  \item
    $B$ is a periodic algebra of period $4$.
  \item
    The stable Auslander-Reiten quiver $\Gamma_B^s$ of $B$
   consists of stable tubes of ranks $1$ and $2$.
\end{enumerate}
\end{corollary*}

This paper is organized as follows.
In Sections
\ref{sec:boundquiver},
\ref{sec:wsalg},
\ref{sec:tetralg}
we present necessary background and results
needed for the proof of the main theorem.
Section~\ref{sec:pre} is devoted to preliminary results
on minimal relations of $2$-regular algebras of generalized
quaternion type.
In Section~\ref{sec:markov} we prove the implication
(i) $\Rightarrow$ (iii) for algebras whose quiver is the
Markov quiver.
In Section~\ref{sec:triangulation} we prove that the
quiver of a $2$-regular algebra of generalized
quaternion type admits a triangulation, and hence comes
from a directed triangulated surface.
Section~\ref{sec:almosttetrahedral} is devoted to
a local presentation of algebras with almost tetrahedral quiver,
while in Section~\ref{sec:tetrahedralquiver} we prove
the main theorem for algebras with the tetrahedral quiver.
Section~\ref{sec:localpresentation} is devoted to
local presentation of $2$-regular algebras of
generalized quaternion type, which are not treated in
Sections \ref{sec:markov} and \ref{sec:tetrahedralquiver}.
In Section~\ref{sec:proofth} we
complete  the proof of the implication
(i) $\Rightarrow$ (iii) of the main theorem.
The final Section~\ref{sec:proofcor} is devoted to
the proof of Corollary~ \ref{cor:main2}.

For general background on the relevant representation theory
we refer to the books
\cite{ASS,E4,SS2,SY,SY2}.
We assume throughout that $Q$ has at least three vertices, 
algebras of generalized quaternion type with at most two simple modules
are of quaternion type and can be found in \cite{E4}.

\section{Bound quiver algebras}\label{sec:boundquiver}

A quiver is a quadruple $Q = (Q_0,Q_1,s,t)$ consisting
of a finite set $Q_0$ of vertices,
a finite set $Q_1$ of arrows,
and two maps $s,t : Q_1 \to Q_0$
which associate to each arrow $\alpha \in Q_1$
its source $s(\alpha) \in Q_0$ and its target $t(\alpha) \in Q_0$.
We denote by $K Q$ the path algebra of $Q$ over $K$ whose
underlying $K$-vector space has as its basis the set of all paths
in $Q$ of length $\geq 0$, and by $R_Q$ the arrows ideal
of $K Q$ generated by all paths in $Q$ of length $\geq 1$.
An ideal $I$ of $K Q$ is said to be admissible
if there exists $m \geq 2$ such that
$R_Q^m \subseteq I \subseteq R_Q^2$.
If $I$ is an admissible ideal of $K Q$, then
the quotient algebra $K Q/I$ is called
a bound quiver algebra, and is a finite-dimensional
basic $K$-algebra.
The algebra $K Q/I$ is indecomposable if and only if
$Q$ is connected.
Moreover, $R_Q/I$ is the radical of $K Q/I$.
Every basic finite dimensional
$K$-algebra $A$ has a bound quiver presentation
$A \cong K Q/I$, where $Q = Q_A$ is the Gabriel quiver
of $A$ and $I$ is an admissible ideal of $K Q$.
For a bound quiver algebra $A = KQ/I$, we denote by $e_i$,
$i \in Q_0$, the associated complete set of pairwise
orthogonal primitive idempotents of $A$, and by
$S_i = e_i A/e_i J(A)$ (respectively, $P_i = e_i A$),
$i \in Q_0$, the associated complete family of pairwise
non-isomorphic simple modules (respectively, indecomposable
projective modules) in $\mod A$.
Here, $J(A)$ is the radical of $A$, which we will usually
write as $J$.
A relation in the path algebra $K Q$ of a quiver $Q$
is a $K$-linear combination
\[
  \varrho = \sum_{i=0}^m \lambda_i w_i,
\]
where all $\lambda_i$ are non-zero elements of $K$
and $w_i$ are pairwise different paths in $Q$
of length at least two having the same source
and the same target.
It is known that every admissible ideal $I$
of a path algebra $K Q$ is generated by a finite set
$\varrho_1,\dots,\varrho_n$ of relations.
Moreover, we may choose such relations
$\varrho_1,\dots,\varrho_n$ to be minimal relations
(not linear combinations of relations from $I$).
For a bound quiver algebra $A = K Q / I$ we will identify
an arrow $\alpha \in Q_1$ with its coset $\alpha + I$ in $A$.
Then, for a set $\varrho_1,\dots,\varrho_n$ of (minimal) relations
generating $I$, we have in $A$ the (minimal) equalities
$\varrho_1 = 0, \dots, \varrho_n = 0$,
called minimal relations of $A$.

A quiver $Q = (Q_0,Q_1,s,t)$ is said to be $2$-regular
if for every vertex $i \in Q_0$
there are exactly two arrows with source $i$
and exactly two arrows with target $i$.
For a $2$-regular quiver $Q$ and an arrow $\alpha \in Q_1$,
we denote by $\bar{\alpha}$ the arrow of $Q$ with
$\bar{\alpha} \neq {\alpha}$ and $s(\bar{\alpha}) = s({\alpha})$,
and by $\alpha^*$ the arrow of $Q$ with
$\alpha^* \neq \alpha$ and $t(\alpha^*) = t(\alpha)$.
A bound quiver algebra $A = K Q / I$ with $Q$
a $2$-regular quiver is said to be a $2$-regular algebra.

\section{Weighted surface algebras}\label{sec:wsalg}

In this  paper, by a \emph{surface}
we mean a connected, compact, $2$-dimensional real
manifold $S$, orientable or non-orientable,
with or without boundary.
It is well known that every surface $S$ admits
an additional structure of a finite
$2$-dimensional triangular cell complex,
and hence a triangulation by the deep Triangulation Theorem
(see for example \cite[Section~2.3]{Ca}).

For a natural number $n$, we denote by $D^n$ the unit disk
in the $n$-dimensional Euclidean space $\bR^n$,
formed by all points of distance $\leq 1$ from the origin.
The boundary $\partial D^n$ of $D^n$ is the unit sphere
$S^{n-1}$ in $\bR^n$, the points of distance $1$
from the origin.
Further, by an $n$-cell we mean a topological space
homeomorphic to the open disk $\intt D^n = D^n \setminus \partial D^n$.
In particular,
$D^0$ and $e^0$ consist of a single point,
and $S^0 = \partial D^1$ consists of two points.

We refer to \cite[Appendix]{H} for some basic topological
facts about cell complexes.

Let $S$ be a surface.
In this  paper, by a
\emph{finite $2$-dimensional triangular cell complex structure}
on $S$ we mean a finite family of continous maps
$\varphi_i^n : D_i^n \to S$, with $n \in \{0,1,2\}$
and $D_i^n = D^n$,
satisfying the following conditions:
\begin{enumerate}[(1)]
 \item
  Each $\varphi_i^n$ restricts to a homeomorphism from the interior\
  $\intt D_i^n$ to the $n$-cell $e_i^n = \varphi_i^n(\intt D_i^n)$ of $S$,
  and these cells are all disjoint and their union is $S$.
 \item
  For each $2$-dimensional cell $e_i^2$, the set $\varphi_i^2(\partial D_i^2)$
  is contained in the union of $k$ $1$-cells and $k$ $0$-cells,
  with $k \in \{2,3\}$.
\end{enumerate}
Then the closures  $\varphi_i^2(D_i^2)$ of all $2$-cells $e_i^2$
are called \emph{triangles} of $S$,
and the closures  $\varphi_i^1(D_i^1)$ of all $1$-cells $e_i^1$
are called \emph{edges} of $S$.
The collection $T$ of all triangles  $\varphi_i^2(D_i^2)$
is said to be a \emph{triangulation} of $S$.

We assume that such a triangulation $T$ of $S$
has at least three
pairwise
different edges,
or equivalently, there are at least three
pairwise
different
$1$-cells in the corresponding triangular cell complex structure on $S$.
Then $T$ is a finite collection $T_1,\dots,T_n$ of triangles
of the form
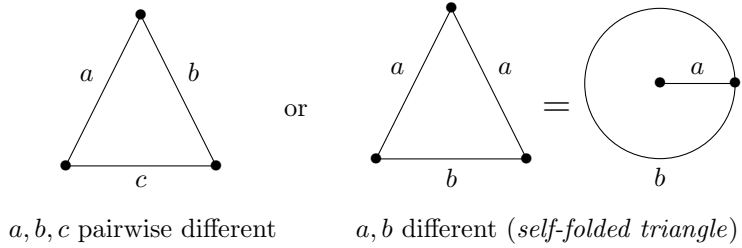
\begin{gather*}
\qquad
\begin{tikzpicture}[auto]
\coordinate (a) at (0,2);
\coordinate (b) at (-1,0);
\coordinate (c) at (1,0);
\draw (a) to node {$b$} (c)
(c) to node {$c$} (b);
\draw (b) to node {$a$} (a);
\node (a) at (0,2) {$\bullet$};
\node (b) at (-1,0) {$\bullet$};
\node (c) at (1,0) {$\bullet$};
\end{tikzpicture}
\qquad
\raisebox{7ex}{\mbox{or}}
\qquad
\begin{tikzpicture}[auto]
\coordinate (a) at (0,2);
\coordinate (b) at (-1,0);
\coordinate (c) at (1,0);
\draw (c) to node {$b$} (b)
(b) to node {$a$} (a);
\draw (a) to node {$a$} (c);
\node (a) at (0,2) {$\bullet$};
\node (b) at (-1,0) {$\bullet$};
\node (c) at (1,0) {$\bullet$};
\end{tikzpicture}
\raisebox{7ex}{\LARGE =}
\ \,
\begin{tikzpicture}[auto]
\coordinate (c) at (0,0);
\coordinate (a) at (1,0);
\coordinate (b) at (0,-1);
\draw (c) to node {$a$} (a);
\draw (b) arc (-90:270:1) node [below] {$b$};
\node (a) at (1,0) {$\bullet$};
\node (c) at (0,0) {$\bullet$};
\end{tikzpicture}
\\
\mbox{$a,b,c$ pairwise different}
\qquad
\quad
\mbox{$a,b$ different (\emph{self-folded triangle})}
\end{gather*}
such that every edge of such a triangle in $T$ is either
the edge of exactly two triangles, or is the self-folded
edge, or else lies on the boundary.
A  given surface $S$ admits many
finite $2$-dimensional triangular cell structures, and hence
triangulations.
We refer to \cite{Ca,KC,Ki} for
general background on surfaces and
constructions of surfaces from plane models.

Let $S$ be a surface and
$T$ a triangulation of $S$.
To each triangle $\Delta$ in $T$ we may associate an orientation
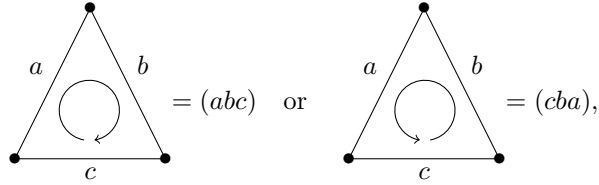
\[
\begin{tikzpicture}[auto]
\coordinate (a) at (0,2);
\coordinate (b) at (-1,0);
\coordinate (c) at (1,0);
\coordinate (d) at (-.08,.25);
\draw (a) to node {$b$} (c)
(c) to node {$c$} (b)
(b) to node {$a$} (a);
\draw[->] (d) arc (260:-80:.4);
\node (a) at (0,2) {$\bullet$};
\node (b) at (-1,0) {$\bullet$};
\node (c) at (1,0) {$\bullet$};
\end{tikzpicture}
\raisebox{7ex}{\!\!$=(abc)$}
\raisebox{7ex}{\quad or \ \ }
\begin{tikzpicture}[auto]
\coordinate (a) at (0,2);
\coordinate (b) at (-1,0);
\coordinate (c) at (1,0);
\coordinate (d) at (.08,.25);
\draw (a) to node {$b$} (c)
(c) to node {$c$} (b)
(b) to node {$a$} (a);
\draw[->] (d) arc (-80:260:.4);
\node (a) at (0,2) {$\bullet$};
\node (b) at (-1,0) {$\bullet$};
\node (c) at (1,0) {$\bullet$};
\end{tikzpicture}
\raisebox{7ex}{\!\!$=(cba)$,}
\]
if $\Delta$ has pairwise different edges $a,b,c$, and
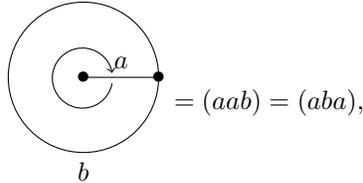
\[
\begin{tikzpicture}[auto]
\coordinate (c) at (0,0);
\coordinate (a) at (1,0);
\coordinate (b) at (0,-1);
\coordinate (d) at (.38,-.08);
\draw (c) to node {$a$} (a);
\draw (b) arc (-90:270:1) node [below] {$b$};
\node (a) at (1,0) {$\bullet$};
\node (c) at (0,0) {$\bullet$};
\draw[->] (d) arc (-10:-350:.4);
\end{tikzpicture}
\raisebox{7ex}{$=(aab)=(aba)$,}
\]
if $\Delta$ is self-folded, with the self-folded edge $a$,
and the other edge $b$.
Fix an orientation of each triangle $\Delta$ of $T$,
and denote this choice by $\vec{T}$.
Then
the pair $(S,\vec{T})$ is said to be a
\emph{directed triangulated surface}.
To each directed triangulated surface $(S,\vec{T})$
we associate the quiver $Q(S,\vec{T})$ whose vertices
are the edges of $T$ and the arrows are defined as
follows:
\begin{enumerate}[(1)]
 \item
  for any oriented triangle $\Delta = (a b c)$ in $\vec{T}$
  with pairwise different edges $a,b,c$, we have the cycle
  
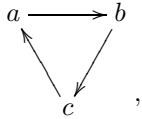
\[
    \xymatrix@C=.8pc{a \ar[rr] && b \ar[ld] \\ & c \ar[lu]}
    \raisebox{-7ex}{,}
  \]
 \item
  for any self-folded triangle $\Delta = (a a b)$ in $\vec{T}$,
  we have the quiver
  
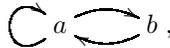
\[
    \xymatrix{ a \ar@(dl,ul)[] \ar@/^1.5ex/[r] & b \ar@/^1.5ex/[l]} ,
  \]
 \item
  for any boundary edge $a$ in ${T}$,
  we have the loop
  \[
    {\xymatrix{ a \ar@(dl,ul)[]}} .
  \]
\end{enumerate}
Then $Q = Q(S,\vec{T})$ is a triangulation quiver in
the following sense (introduced independently by Ladkani
in \cite{La3}).

A \emph{triangulation quiver} is a pair $(Q,f)$,
where $Q = (Q_0,Q_1,s,t)$ is a finite connected quiver
and $f : Q_1 \to Q_1$ is a permutation
on the set $Q_1$ of arrows of $Q$ satisfying
the following conditions:
\begin{enumerate}[(a)]
 \item[(a)]
  every vertex $i \in Q_0$ is the source and target of exactly two
  arrows in $Q_1$,
 \item[(b)]
  for each arrow $\alpha \in Q_1$, we have $s(f(\alpha)) = t(\alpha)$,
 \item[(c)]
  $f^3$ is the identity on $Q_1$.
\end{enumerate}

For the quiver $Q = Q(S,\vec{T})$ of a
directed triangulated surface $(S,\vec{T})$,
the permutation $f$ on its set of arrows
is defined as follows:
\begin{enumerate}[(1)]
 \item
  \raisebox{3ex}%
  {\xymatrix@C=.8pc{a \ar[rr]^{\alpha} && b \ar[ld]^{\beta} \\ & c \ar[lu]^{\gamma}}}%
  \quad
  $f(\alpha) = \beta$,
  $f(\beta) = \gamma$,
  $f(\gamma) = \alpha$,

  for an oriented triangle $\Delta = (a b c)$ in $\vec{T}$,
  with pairwise different edges $a,b,c$,
 \item
  \raisebox{0ex}%
  {\xymatrix{ a \ar@(dl,ul)[]^{\alpha} \ar@/^1.5ex/[r]^{\beta} & b \ar@/^1.5ex/[l]^{\gamma}}}
  \quad
  $f(\alpha) = \beta$,
  $f(\beta) = \gamma$,
  $f(\gamma) = \alpha$,

  for a self-folded triangle $\Delta = (a a b)$ in $\vec{T}$,\vspace{1mm}
 \item
  \raisebox{0ex}%
  {\xymatrix{ a \ar@(dl,ul)[]^{\alpha}}}
  \quad
  $f(\alpha) = \alpha$,\vspace{1mm}

  for a boundary edge $a$ of ${T}$.
\end{enumerate}

We note that, for every
triangulation quiver $(Q,f)$,
$Q$ is a $2$-regular quiver.
\emph{We will  consider only the trangulation
quivers with at least three vertices.}

We would like to mention that different directed triangulated
surfaces (even of different genus) may lead to the same
triangulation quiver (see \cite[Example~4.3]{ES4}).

The following theorem and  its consequence
have been established in \cite[Section~4]{ES4}.

\begin{theorem}
\label{th:2.1}
Let $(Q,f)$ be a triangulation quiver.
Then there exists a directed triangulated surface
$(S,\vec{T})$ such that $(Q,f) = (Q(S,\vec{T}),f)$.
\end{theorem}

\begin{corollary}
\label{cor:2.2}
Let $(Q,f)$ be a triangulation quiver.
Then $Q$ contain a loop $\alpha$ with $f(\alpha) = \alpha$
is and only if
$(Q,f) = (Q(S,\vec{T}),f)$
for a directed triangulated surface $(S,\vec{T})$ with $S$
having non-empty boundary.
\end{corollary}

Let $(Q,f)$ be a triangulation quiver.
Then we have two permutations
$f : Q_1 \to Q_1$
and
$g : Q_1 \to Q_1$
on the set $Q_1$ of arrows of $Q$ such that $f^3$
is the identity on $Q_1$ and $g = \bar{f}$,
where $\bar{ } : Q_1 \to Q_1$
is the involution which assigns to an arrow
$\alpha \in Q_1$ the arrow $\bar{\alpha}$
with $s({\alpha}) = s(\bar{\alpha})$
and ${\alpha} \neq \bar{\alpha}$.
For each arrow $\alpha \in Q_1$, we denote by
$\cO(\alpha)$ the $g$-orbit of $\alpha$
in $Q_1$, and set
$n_{\alpha} = n_{\cO(\alpha)} = |\cO(\alpha)|$.
We denote by $\cO(g)$ the set
of all $g$-orbits in $Q_1$.
A function
\[
  m_{\bullet} : \cO(g) \to \bN^* = \bN \setminus \{0\}
\]
is said to be a \emph{weight function} of $(Q,f)$,
and a function
\[
  c_{\bullet} : \cO(g) \to K^* = K \setminus \{0\}
\]
is said to be a \emph{parameter function} of $(Q,f)$.
We write briefly $m_{\alpha} = m_{\cO(\alpha)}$
and $c_{\alpha} = c_{\cO(\alpha)}$ for $\alpha \in Q_1$.
\emph{In this paper, we will assume that $m_{\alpha} n_{\alpha} \geq 3$
for any arrow $\alpha \in Q_1$.}

\medskip

For any arrow $\alpha \in Q_1$, we consider the path
\begin{align*}
  A_{\alpha} &= \Big( \alpha g(\alpha) \dots g^{n_{\alpha}-1}(\alpha)\Big)^{m_{\alpha}-1}
             \alpha g(\alpha) \dots g^{n_{\alpha}-2}(\alpha) , \mbox{ if } n_{\alpha} \geq 2, \\
  A_{\alpha} &= \alpha^{m_{\alpha}-1} , \mbox{ if } n_{\alpha} = 1 ,
\end{align*}
in $Q$ of length $m_{\alpha} n_{\alpha} - 1$ from
$s(\alpha)$ to $t(g^{n_{\alpha}-2}(\alpha))$.

\medskip

Let $(Q,f)$ be a triangulation quiver with weight and parameter functions
$m_{\bullet}$ and $c_{\bullet}$.
Following  \cite{ES4},
we define  the bound quiver algebra
\[
  \Lambda(Q,f,m_{\bullet},c_{\bullet})
   = K Q / I (Q,f,m_{\bullet},c_{\bullet}),
\]
where $I (Q,f,m_{\bullet},c_{\bullet})$
is the admissible ideal in the path algebra $KQ$ of $Q$ over $K$
generated by:
\begin{enumerate}[(1)]
 \item
  ${\alpha} f({\alpha}) - c_{\bar{\alpha}} A_{\bar{\alpha}}$,
  for all arrows $\alpha \in Q_1$,
 \item
  $\beta f(\beta) g(f(\beta))$,
  for all arrows $\beta \in Q_1$.
\end{enumerate}
Then $\Lambda(Q,f,m_{\bullet},c_{\bullet})$ is called a
\emph{weighted triangulation algebra} of $(Q,f)$.

We note that
$Q$ is the quiver of the algebra $\Lambda(Q,f,m_{\bullet},c_{\bullet})$,
and
the ideal $I(Q,f,m_{\bullet},c_{\bullet})$
is an admissible ideal of $K Q$, by the
assumption that $m_{\alpha} n_{\alpha} \geq 3$
for all arrows $\alpha \in Q_1$.

We have the following proposition proved in \cite[Proposition~5.8]{ES4}.

\begin{proposition}
\label{prop:2.3}
Let $(Q,f)$ be a triangulation quiver,
$m_{\bullet}$ and $c_{\bullet}$
weight and parameter functions of $(Q,f)$.
Then $\Lambda(Q,f,m_{\bullet},c_{\bullet})$
is a finite-dimensional tame symmetric algebra
of dimension $\sum_{\cO \in \cO(g)} m_{\cO} n_{\cO}^2$.
\end{proposition}

Let $(S,\vec{T})$ be a directed triangulated surface,
$(Q(S,\vec{T}),f)$ the associated triangulation quiver,
and let $m_{\bullet}$ and $c_{\bullet}$
be weight and parameter functions of $(Q(S,\vec{T}),f)$.
Then the weighted triangulation algebra
$\Lambda(Q(S,\vec{T}),f,m_{\bullet},c_{\bullet})$
is called a \emph{weighted surface algebra}
and denoted by
$\Lambda(S,\vec{T},m_{\bullet},c_{\bullet})$
\cite{ES4}.
We refer to the next section for definition of
tetrahedral algebras.

We have the following theorem proved in \cite[Theorem~1.2]{ES4}.

\begin{theorem}
\label{th:2.4}
Let $\Lambda = \Lambda(S,\vec{T},m_{\bullet},c_{\bullet})$
be a weighted surface algebra over an algebraically
closed field $K$.
Then the following statements are equivalent:
\begin{enumerate}[(i)]
 \item
  All simple modules in $\mod \Lambda$ are periodic of period $4$.
 \item
  $\Lambda$ is a periodic algebra of period $4$.
 \item
  $\Lambda$ is not a singular tetrahedral algebra.
\end{enumerate}
\end{theorem}
We need also socle deformations of weighted surface algebras.

Let $(Q,f)$ be a triangulation quiver.
A vertex $i \in Q_0$ is said to be a \emph{border vertex}
of $(Q,f)$ if there is a loop $\alpha$ at $i$ with $f(\alpha) = \alpha$.
If so, then
$\bar{\alpha} = g(\alpha)$,
$\alpha = f^2(\alpha) = g^{n_{\bar{\alpha}} - 1} (\bar{\alpha})$,
and
$f^2(\bar{\alpha}) = g^{-1} (\alpha)$.
In particular, we have
$n_{\alpha} = n_{\bar{\alpha}} \geq 3$,
because $|Q_0| \geq 3$.
Hence the loop $\alpha$ is uniquely determined by the vertex $i$,
and we call it a \emph{border loop} of $(Q,f)$.
We also note that
the following equalities hold:
$
  \alpha A_{\bar{\alpha}}
    = B_{\alpha}
    = A_{\alpha} f^2(\bar{\alpha})$
and
$
  \bar{\alpha} A_{g(\bar{\alpha})}
   = B_{\bar{\alpha}}
    = A_{\bar{\alpha}} \alpha
$.
We denote by $\partial(Q,f)$
the set of all border vertices of $(Q,f)$,
and call it the \emph{border} of $(Q,f)$.
Observe that, if $(S,\vec{T})$
is a directed triangulated surface with
$(Q(S,\vec{T}), f) = (Q, f)$,
then the border vertices of $(Q,f)$
correspond bijectively to the boundary edges
of the triangulation $T$ of $S$.
Hence, the border $\partial(Q,f)$ of $(Q,f)$
is non-empty if and only if the boundary
$\partial S$ of $S$ is not empty.
A function
\[
  b_{\bullet} : \partial(Q,f) \to K
\]
is said to be a \emph{border function} of $(Q,f)$.
Assume that $\partial(Q,f)$ is not empty.
Then, for
a weight function
$m_{\bullet} : \cO(g) \to \bN^*$,
a parameter function
$c_{\bullet} : \cO(g) \to K^*$,
and
a border function
$b_{\bullet} : \partial(Q,f) \to K$,
we may consider the bound quiver algebra
\[
  \Lambda(Q,f,m_{\bullet},c_{\bullet},b_{\bullet})
   = K Q / I (Q,f,m_{\bullet},c_{\bullet},b_{\bullet}),
\]
where $I (Q,f,m_{\bullet},c_{\bullet},b_{\bullet})$
is the admissible ideal in the path algebra $KQ$ of $Q$ over $K$
generated by the elements:
\begin{enumerate}[(1)]
 \item
  ${\alpha} f({\alpha}) - c_{\bar{\alpha}} A_{\bar{\alpha}}$,
  for all arrows $\alpha \in Q_1$ which are not border loops,
 \item
  $\alpha^2 - c_{\bar{\alpha}} A_{\bar{\alpha}} - b_{s(\alpha)} B_{\bar{\alpha}}$,
  for all border loops $\alpha \in Q_1$,
 \item
  $\beta f(\beta) g(f(\beta))$,
  for all arrows $\beta \in Q_1$.
\end{enumerate}
Then $\Lambda(Q,f,m_{\bullet},c_{\bullet},b_{\bullet})$ is
said to be a
\emph{socle deformed weighted triangulation algebra} \cite{ES4}.
We note that if $b_{\bullet}$ is the zero border function
($b_i = 0$ for all $i \in \partial(Q,f)$) then
$\Lambda(Q,f,m_{\bullet},c_{\bullet},b_{\bullet})
  =
    \Lambda(Q,f,m_{\bullet},c_{\bullet})$.

We have the following proposition proved in \cite[Proposition~8.1]{ES4}.

\begin{proposition}
\label{prop:2.5}
Let $(Q,f)$ be a triangulation quiver
with $\partial(Q,f)$ not empty,
and let
$m_{\bullet}$,
$c_{\bullet}$,
$b_{\bullet}$
be
weight,
parameter,
border functions of $(Q,f)$.
Then
$\Lambda(Q,f,m_{\bullet},c_{\bullet},b_{\bullet})$
is a finite-dimensional tame symmetric algebra which
is socle equivalent to\linebreak $\Lambda(Q,f,m_{\bullet},c_{\bullet})$.
\end{proposition}

Let $(S,\vec{T})$ be a directed triangulated surface
with non-empty boundary, and let
$m_{\bullet}$,
$c_{\bullet}$,
$b_{\bullet}$
be
weight,
parameter,
border functions of $(Q(S,\vec{T}),f)$.
Then the socle deformed weighted triangulation algebra
$\Lambda(Q(S,\vec{T}),f,m_{\bullet},c_{\bullet},b_{\bullet})$
is called a \emph{socle deformed weighted surface algebra}
and denoted by
$\Lambda(S,\vec{T},m_{\bullet},c_{\bullet},b_{\bullet})$
\cite{ES4}.

We have the following consequence of
\cite[Proposition~9.1 and Theorem~9.4]{ES4}.

\begin{theorem}
\label{th:2.6}
Let
$\bar{\Lambda} = \Lambda(S,\vec{T},m_{\bullet},c_{\bullet},b_{\bullet})$
be a socle deformed weighted surface algebra.
Then the following statements hold.
\begin{enumerate}[(i)]
 \item
  All simple modules in $\mod \bar{\Lambda}$ are periodic modules
  of period $4$.
 \item
  $\bar{\Lambda}$ is a periodic algebra of period $4$.
\end{enumerate}
\end{theorem}

Moreover, we have the following theorem established in
\cite[Theorem~1.4]{ES4}.

\begin{theorem}
\label{th:2.7}
Let $A$ be a basic, indecomposable, symmetric algebra
over an algebraically closed field $K$.
Assume that $A$ is socle equivalent but not isomorphic to
a weighted surface algebra
$\Lambda(S,\vec{T},m_{\bullet},c_{\bullet})$.
Then the following statements hold.
\begin{enumerate}[(i)]
 \item
  The surface $S$ has non-empty boundary.
 \item
  $K$ is of characteristic $2$.
 \item
  $A$ is isomorphic to a socle deformed weighted surface algebra
  $\Lambda(S,\vec{T},m_{\bullet},c_{\bullet},b_{\bullet})$.
\end{enumerate}
\end{theorem}

We end this section with two examples.

\begin{example}
\label{ex:2.8}
Let $S = T$ be the triangle
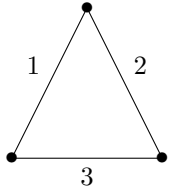
\[
\begin{tikzpicture}[auto]
\coordinate (a) at (0,2);
\coordinate (b) at (-1,0);
\coordinate (c) at (1,0);
\draw (a) to node {2} (c)
(c) to node {3} (b);
\draw (b) to node {1} (a);
\node (a) at (0,2) {$\bullet$};
\node (b) at (-1,0) {$\bullet$};
\node (c) at (1,0) {$\bullet$};
\end{tikzpicture}
\]
with the three pairwise different edges,
forming the boundary of $S$, and consider the orientation
$\vec{T}$ of $T$
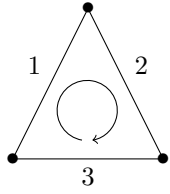
\[
\begin{tikzpicture}[auto]
\coordinate (a) at (0,2);
\coordinate (b) at (-1,0);
\coordinate (c) at (1,0);
\coordinate (d) at (-.08,.25);
\draw (a) to node {2} (c)
(c) to node {3} (b)
(b) to node {1} (a);
\draw[->] (d) arc (260:-80:.4);
\node (a) at (0,2) {$\bullet$};
\node (b) at (-1,0) {$\bullet$};
\node (c) at (1,0) {$\bullet$};
\end{tikzpicture}
\]
of $T$.
Then the triangulation quiver
$(Q(S,\vec{T}),f)$ is the quiver
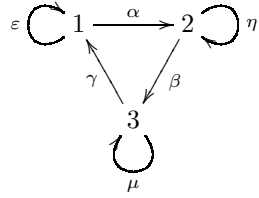
\[
  \xymatrix@C=.8pc{
     1 \ar@(dl,ul)[]^{\varepsilon} \ar[rr]^{\alpha} &&
     2 \ar@(ur,dr)[]^{\eta} \ar[ld]^{\beta} \\
     & 3 \ar@(dr,dl)[]^{\mu} \ar[lu]^{\gamma}}
\]
with the $f$-orbits
$(\alpha\, \beta\, \gamma)$,
$(\varepsilon)$,
$(\eta)$,
$(\mu)$.
Observe that we have only one $g$-orbit
$(\alpha\, \eta\, \beta\, \mu\, \gamma\, \varepsilon)$
of arrows in $Q(S,\vec{T})$.
Hence a weight function
$m_{\bullet} : \cO(g) \to \bN^*$
and a parameter function
$c_{\bullet} : \cO(g) \to K^*$
are given by a positive integer $m$ and
a parameter $c \in K^*$.
The associated weighted surface algebra
$\Lambda = \Lambda(Q(S,\vec{T}),f,m_{\bullet},c_{\bullet})$
is given by the above quiver and the relations
\begin{align*}
 \alpha\beta &= c(\varepsilon\alpha\eta\beta\mu\gamma)^{m-1} \varepsilon\alpha\eta\beta\mu ,
 &
 \varepsilon^2 &= c(\alpha\eta\beta\mu\gamma\varepsilon)^{m-1} \alpha\eta\beta\mu\gamma,
 &
 \alpha\beta\mu &= 0,
 &
 \varepsilon^2 \alpha &= 0,
\\
 \beta\gamma &= c(\eta\beta\mu\gamma\varepsilon\alpha)^{m-1} \eta\beta\mu\gamma\varepsilon ,
 &
 \eta^2 &= c(\beta\mu\gamma\varepsilon\alpha\eta)^{m-1} \beta\mu\gamma\varepsilon\alpha,
 &
 \beta\gamma\varepsilon &= 0,
 &
 \eta^2 \beta &= 0,
\\
 \gamma\alpha &= c(\mu\gamma\varepsilon\alpha\eta\beta)^{m-1} \mu\gamma\varepsilon\alpha\eta,
 &
 \mu^2 &= c(\gamma\varepsilon\alpha\eta\beta\mu)^{m-1} \gamma\varepsilon\alpha\eta\beta,
 &
 \gamma\alpha\eta &= 0,
 &
 \mu^2 \gamma &= 0.
\end{align*}
Observe that the border
$\partial(Q(S,\vec{T}),f)$
of $(Q(S,\vec{T}),f)$
is the set $Q_0 = \{1,2,3\}$ of vertices of $Q$,
and $\varepsilon$, $\eta$, $\mu$
are the border loops.
Take now a border function
$b_{\bullet} : \partial(Q(S,\vec{T}),f) \to K$.
Then the associated socle  deformed
weighted surface algebra
$\bar{\Lambda} = \Lambda(Q(S,\vec{T}),f,m_{\bullet},c_{\bullet},b_{\bullet})$
is given by the above quiver and the relations
\begin{gather*}
\alpha\beta = c(\varepsilon\alpha\eta\beta\mu\gamma)^{m-1} \varepsilon\alpha\eta\beta\mu ,
\ \,\quad
\varepsilon^2 = c(\alpha\eta\beta\mu\gamma\varepsilon)^{m-1} \alpha\eta\beta\mu\gamma
+ b_1 (\alpha\eta\beta\mu\gamma\varepsilon)^m,
\\
\beta\gamma = c(\eta\beta\mu\gamma\varepsilon\alpha)^{m-1} \eta\beta\mu\gamma\varepsilon ,
\ \,\quad
\eta^2 =c(\beta\mu\gamma\varepsilon\alpha\eta)^{m-1} \beta\mu\gamma\varepsilon\alpha
+ b_2 (\beta\mu\gamma\varepsilon\alpha\eta)^m ,
\\
\gamma\alpha = c(\mu\gamma\varepsilon\alpha\eta\beta)^{m-1} \mu\gamma\varepsilon\alpha\eta,
\ \,\quad
\mu^2 = c(\gamma\varepsilon\alpha\eta\beta\mu)^{m-1} \gamma\varepsilon\alpha\eta\beta
+ b_3 (\gamma\varepsilon\alpha\eta\beta\mu)^m ,
\\
\alpha\beta\mu = 0,
\ \ \quad
\varepsilon^2 \alpha = 0,
\ \ \quad
\beta\gamma\varepsilon = 0,
\ \ \quad
\eta^2 \beta = 0,
\ \ \quad
\gamma\alpha\eta = 0,
\ \ \quad
\mu^2 \gamma = 0.
\end{gather*}

Assume that $m=1$, $c=1$,  $K$ has characteristic $2$ and $b_{\bullet}$
is non-zero. Then it is shown in \cite[Example~8.4]{ES4}
that the algebras $\Lambda$ and $\bar{\Lambda}$
are not isomorphic.
\end{example}

\begin{example}
Let $S$ be the connected sum $\bT \# \bP$
of the torus $\bT$ and the projective plane $\bP$,
and $T$ the following triangulation of $S$
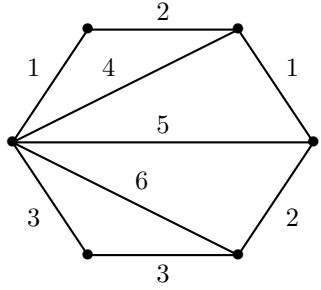
\[
\begin{tikzpicture}
[scale=1,auto]
\node (1) at (-2.0,0) {$\bullet$};
\node (2) at (-1,1.5) {$\bullet$};
\node (3) at (1,1.5) {$\bullet$};
\node (4) at (2.0,0) {$\bullet$};
\node (5) at (1,-1.5) {$\bullet$};
\node (6) at (-1,-1.5) {$\bullet$};
\coordinate (1) at (-2.0,0);
\coordinate (2) at (-1,1.5) ;
\coordinate (3) at (1,1.5) ;
\coordinate (4) at (2.0,0) ;
\coordinate (5) at (1,-1.5) ;
\coordinate (6) at (-1,-1.5) ;
\draw[thick]
(1) edge node {1} (2)
(1) edge node {4} (3)
(1) edge node {5} (4)
(1) edge node {6} (5)
(2) edge node {2} (3)
(3) edge node {1} (4)
(4) edge node {2} (5)
(5) edge node {3} (6)
(6) edge node {3} (1) ;
\end{tikzpicture}
\]
where  the edges $1,2$ correspond to $\bT$, and the edge $3$
corresponds to $\bP$.
We note that $S$ is non-orientable and without boundary.
Consider the following orientation $\vec{T}$ of triangles in $T$
\[
  (1\ 2\ 4),\ (1\ 4\ 5),\ (5\ 2\ 6),\ (3\ 3\ 6) .
\]
Then the associated triangulation quiver $(Q(S,\vec{T}),f)$
is of the form
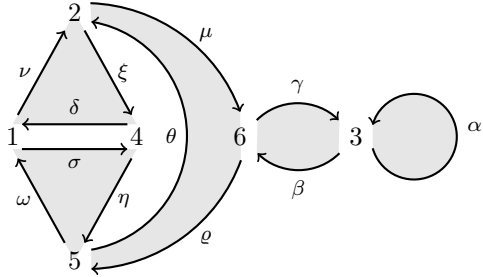
\[
\begin{tikzpicture}
[->,scale=1.1,auto]
\coordinate (4) at (-1.25,0);
\coordinate (4l) at (-1.25,-0.15);
\coordinate (4p) at (-1.25,0.15);
\coordinate (5) at (-2,1.5);
\coordinate (1) at (-2.75,0);
\coordinate (1l) at (-2.75,-0.15);
\coordinate (1p) at (-2.75,0.15);
\coordinate (2) at (-2,-1.5);
\coordinate (6) at (0,0);
\fill[fill=gray!20] (0,.3) arc (20:88:2) -- (-1.8,1.39) arc (80:-80:1.4) -- (-1.8,-1.6) arc (-88:-20:2) -- cycle;
\fill[fill=gray!20] (4p) -- (5) -- (1p) -- cycle;
\fill[fill=gray!20] (4l) -- (2) -- (1l) -- cycle;
\fill[fill=gray!20] (.2,-.2) arc (225:315:0.7) -- (1.2,.2) arc (45:135:0.7) -- cycle;
\node [circle,minimum size=15](A) at (1.4,0) { };
\node [circle,minimum size=32](B) at (2.1,0) {};
\coordinate  (C) at (intersection 2 of A and B);
\coordinate  (D) at (intersection 1 of A and B);
 \tikzAngleOfLine(B)(D){\AngleStart}
 \tikzAngleOfLine(B)(C){\AngleEnd}
\fill[gray!20]%
   let \p1 = ($ (B) - (D) $), \n2 = {veclen(\x1,\y1)}
   in
     (D) arc (\AngleStart-360:\AngleEnd:\n2); 
\node [fill=white,circle,minimum size=1.5](A) at (1.4,0) { };
\draw[thick,->]%
   let \p1 = ($ (B) - (D) $), \n2 = {veclen(\x1,\y1)}
   in
     (B) ++(60:\n2) node[right]{\footnotesize\ \ \ \raisebox{-7ex}{$\alpha$}}
     (D) arc (\AngleStart-360:\AngleEnd:\n2); 
\node (4) at (-1.25,0) {4};
\node (4l) at (-1.25,0.15) {};
\node (4p) at (-1.25,-0.15) {};
\node (5) at (-2,-1.5) {5};
\node (1) at (-2.75,0) {1};
\node (1l) at (-2.75,0.15) {};
\node (1p) at (-2.75,-0.15) {};
\node (2) at (-2,1.5) {2};
\node (6) at (0,0) {6};
\node (3) at (1.4,0) {3};
\draw[thick,<-] (0,.3) arc (20:54:2) node[right]{\footnotesize\ \raisebox{0ex}{$\mu$}} arc (54:88:2);
\draw[thick,<-] (-1.8,1.39) arc (80:0:1.4) node[left]{\footnotesize \raisebox{0ex}{$\theta$}} arc (0:-80:1.4) ;
\draw[thick,<-] (-1.8,-1.6) arc (-88:-54:2) node[right]{\footnotesize\ \raisebox{0ex}{$\varrho$}} arc (-54:-20:2);
\draw[thick,->] (1l) to node[left]{\footnotesize$\nu$} (2);
\draw[thick,->] (2) to node[right]{\footnotesize$\xi$} (4l);
\draw[thick,->] (4l) to node[above]{\footnotesize$\delta$} (1l);
\draw[thick,<-] (1p) to node[left]{\footnotesize$\omega$} (5);
\draw[thick,<-] (5) to node[right]{\footnotesize$\eta$} (4p);
\draw[thick,<-] (4p) to node[below]{\footnotesize$\sigma$} (1p);
\draw[thick,->] (1.2,-.2) arc (315:270:0.7) node[below]{\footnotesize$\beta$} arc (270:225:0.7);
\draw[thick,->] (.2,.2) arc (135:90:0.7) node[above]{\footnotesize$\gamma$} arc (90:45:0.7);
\end{tikzpicture}
\]
with the shaded $f$-orbits
$(\alpha\ \beta\ \gamma)$,
$(\varrho\ \theta\ \mu)$,
$(\omega\ \sigma\ \eta)$,
$(\nu\ \xi\ \delta)$.
Then the set $\cO(g)$ consists of the four $g$-orbits
\begin{align*}
 \cO(\alpha) &= (\alpha) ,
 &
 \cO(\sigma) &= (\sigma\ \delta) ,
 &
 \cO(\eta) &= (\eta\ \theta\ \xi) ,
 &
 \cO(\beta) &= (\beta\ \varrho\ \omega\ \nu\ \mu\ \gamma)
\end{align*}
of lengths $1, 2, 3, 6$.
Thus a weight function
$m_{\bullet} : \cO(g) \to \bN^*$
is given by natural numbers
$p = m_{\cO(\alpha)} \geq 3$,
$q = m_{\cO(\sigma)} \geq 2$,
$r = m_{\cO(\eta)} \geq 1$,
$s = m_{\cO(\beta)} \geq 1$
(due to the imposed general assumption
$m_{\cO(\zeta)} n_{\cO(\zeta)} \geq 3$
for any arrow $\zeta$).
The parameter function
$c_{\bullet} : \cO(g) \to K^*$
is given by four non-zero scalars
$a = c_{\cO(\alpha)}$,
$b = c_{\cO(\sigma)}$,
$c = c_{\cO(\eta)}$,
$d = c_{\cO(\beta)}$.
Then the weighted surface algebra
$\Lambda(S,\vec{T},m_{\bullet},c_{\bullet})$
is given by the above quiver, the commutativity relations
\begin{align*}
 \gamma\alpha &=  d(\varrho\omega\nu\mu\gamma\beta)^{s-1} \varrho\omega\nu\mu\gamma,
 &
 \alpha\beta &= d(\beta\varrho\omega\nu\mu\gamma)^{s-1} \beta\varrho\omega\nu\mu ,
 &
 \beta\gamma &= a \alpha^{p-1} ,
\\
 \varrho\theta &= d(\gamma\beta\varrho\omega\nu\mu)^{s-1}\gamma\beta\varrho\omega\nu,
 &
 \theta\mu &= d(\omega\nu\mu\gamma\beta\varrho)^{s-1} \omega\nu\mu\gamma\beta,
 &
 \mu\varrho &= c(\xi\eta\theta)^{r-1} \xi\eta,
\\
 \omega \sigma &= c(\theta\xi\eta)^{r-1} \theta\xi,
 &
 \sigma\eta &= d(\nu\mu\gamma\beta\varrho\omega)^{s-1} \nu\mu\gamma\beta\varrho,
 &
 \eta\omega &= b(\delta\sigma)^{q-1} \delta,
\\
 \nu\xi &= b(\sigma \delta)^{q-1} \sigma,
 &
 \xi\delta &= d(\mu\gamma\beta\varrho\omega\nu)^{s-1}\mu\gamma\beta\varrho\omega,
 &
 \delta\nu &= c(\eta\theta\xi)^{r-1} \eta\theta,
\end{align*}
and the zero relations
\begin{align*}
 \alpha\beta \varrho &= 0 ,
 &
 \beta \gamma \beta &= 0 ,
 &
 \gamma \alpha^2 &= 0 ,
 &
 \varrho \theta \xi &= 0 ,
 &
 \theta \mu \gamma  &= 0 ,
 &
 \mu \varrho \omega &= 0 ,
 \\
 \omega \sigma \delta &= 0 ,
 &
 \sigma \eta \theta &= 0 ,
 &
 \eta \omega \nu &= 0 ,
 &
 \nu \xi \eta &= 0 ,
 &
 \xi \delta \sigma &= 0 ,
 &
 \delta \nu \mu &= 0 .
\end{align*}
\end{example}

\section{Tetrahedral algebras}\label{sec:tetralg}

The aim of this section is to introduce the
tetrahedral algebras of arbitrary degree
and describe their basic properties.
Consider the tetrahedron
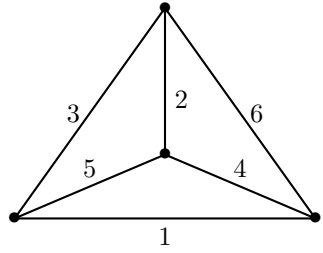
\[
\begin{tikzpicture}
[scale=1]
\node (A) at (-2,0) {$\bullet$};
\node (B) at (2,0) {$\bullet$};
\node (C) at (0,.85) {$\bullet$};
\node (D) at (0,2.8) {$\bullet$};
\coordinate (A) at (-2,0) ;
\coordinate (B) at (2,0) ;
\coordinate (C) at (0,.85) ;
\coordinate (D) at (0,2.8) ;
\draw[thick]
(A) edge node [left] {3} (D)
(D) edge node [right] {6} (B)
(D) edge node [below right] {2} (C)
(A) edge node [above] {5} (C)
(C) edge node [above] {4} (B)
(A) edge node [below] {1} (B) ;
\end{tikzpicture}
\]
with the coherent orientation
of triangles:
$(1\ 5\ 4)$, $(2\ 5\ 3)$, $(2\ 6\ 4)$, $(1\ 6\ 3)$.
Then, following \cite[Example~6.1]{ES4}, we have the associated
triangulation quiver $(Q,f)$ of the form
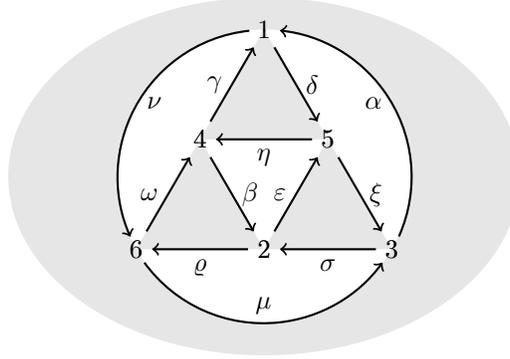
\[
\begin{tikzpicture}
[scale=.85]
\node (1) at (0,1.72) {$1$};
\node (2) at (0,-1.72) {$2$};
\node (3) at (2,-1.72) {$3$};
\node (4) at (-1,0) {$4$};
\node (5) at (1,0) {$5$};
\node (6) at (-2,-1.72) {$6$};
\coordinate (1) at (0,1.72);
\coordinate (2) at (0,-1.72);
\coordinate (3) at (2,-1.72);
\coordinate (4) at (-1,0);
\coordinate (5) at (1,0);
\coordinate (6) at (-2,-1.72);
\fill[fill=gray!20]
      (0,2.22cm) arc [start angle=90, delta angle=-360, x radius=4cm, y radius=2.8cm]
 --  (0,1.72cm) arc [start angle=90, delta angle=360, radius=2.3cm]
     -- cycle;
\fill[fill=gray!20]
    (1) -- (4) -- (5) -- cycle;
\fill[fill=gray!20]
    (2) -- (4) -- (6) -- cycle;
\fill[fill=gray!20]
    (2) -- (3) -- (5) -- cycle;

\node (1) at (0,1.72) {$1$};
\node (2) at (0,-1.72) {$2$};
\node (3) at (2,-1.72) {$3$};
\node (4) at (-1,0) {$4$};
\node (5) at (1,0) {$5$};
\node (6) at (-2,-1.72) {$6$};
\draw[->,thick] (-.23,1.7) arc [start angle=96, delta angle=108, radius=2.3cm] node[midway,right] {$\nu$};
\draw[->,thick] (-1.87,-1.93) arc [start angle=-144, delta angle=108, radius=2.3cm] node[midway,above] {$\mu$};
\draw[->,thick] (2.11,-1.52) arc [start angle=-24, delta angle=108, radius=2.3cm] node[midway,left] {$\alpha$};
\draw[->,thick]
(1) edge node [right] {$\delta$} (5)
(2) edge node [left] {$\varepsilon$} (5)
(2) edge node [below] {$\varrho$} (6)
(3) edge node [below] {$\sigma$} (2)
(4) edge node [left] {$\gamma$} (1)
(4) edge node [right] {$\beta$} (2)
(5) edge node [right] {$\xi$} (3)
(5) edge node [below] {$\eta$} (4)
(6) edge node [left] {$\omega$} (4)
;
\end{tikzpicture}
\]
where
$f$ is the permutation of arrows of order $3$ described by
the shaded subquivers.
Then $g$ is the permutation on the set of arrows of $Q$
whose $g$-orbits are the four white $3$-cycles.

Let $m \geq 1$ be a natural number and $\lambda \in K$.
We denote by $\Lambda(m,\lambda)$ the algebra given
by the above quiver and the relations:
\begin{align*}
 \gamma\delta &= \beta\varepsilon + \lambda (\beta\varrho\omega)^{m-1} \beta\varepsilon,
 &
 \delta\eta &= \nu\omega,
 &
 \eta\gamma &= \xi\alpha,
 &
 \nu \mu &= \delta\xi ,
\\
 \varrho\omega &= \varepsilon\eta + \lambda (\varepsilon\xi\sigma)^{m-1} \varepsilon\eta,
 &
 \omega\beta &= \mu\sigma,
 &
 \beta\varrho &= \gamma\nu ,
 &
 \mu \alpha &= \omega \gamma ,
\\
 \xi\sigma &= \eta\beta + \lambda (\eta\gamma\delta)^{m-1} \eta\beta,
 &
 \sigma\varepsilon &= \alpha\delta,
 &
 \varepsilon\xi &= \varrho\mu,
 &
 \alpha\nu &= \sigma\varrho,
\\
\omit\rlap{\qquad\quad$\big(\theta f(\theta) f^2(\theta)\big)^{m-1} \theta f(\theta) g\big(f(\theta)\big) = 0$ for any arrow $\theta$ in $Q$.}
\end{align*}
We call $\Lambda(m,\lambda)$
a \emph{tetrahedral algebra of degree $m$}.
Moreover, an algebra $\Lambda(m,\lambda)$ with
$\lambda \in K^* = K \setminus \{0\}$ is said to be
a \emph{non-singular tetrahedral algebra of degree $m$}.

We note that for $m = 1$,
$\Lambda(1,\lambda)$ is the \emph{tetrahedral algebra}
$\Lambda(\bS,\lambda+1)$, investigated in
\cite[Section~6]{ES4}, and is a weighted
surface algebra.
On the other hand, if $m \geq 2$, then
$\Lambda(m,\lambda)$ is not a  weighted
surface algebra,
and is called in \cite{ES5} a \emph{higher tetrahedral algebra}.

The following facts are consequence of results
established in \cite{ES4} and \cite{ES5}.

\begin{proposition}
\label{prop:tetr1}
Let $\Lambda = \Lambda(m,\lambda)$
be a tetrahedral algebra of degree $m \geq 1$.
Then $\Lambda$ is a finite-dimensional tame symmetric algebra
of dimension $36 m$.
\end{proposition}

\begin{theorem}
\label{th:tetr2}
Let $\Lambda = \Lambda(m,\lambda)$
be a tetrahedral algebra of degree $m \geq 1$.
Then the following statements are equivalent:
\begin{enumerate}[(i)]
 \item
  $\mod \Lambda$ admits a periodic simple module.
 \item
  All simple modules in $\mod \Lambda$ are periodic of period $4$.
 \item
  $\Lambda$ is a periodic algebra of period $4$.
 \item
  $\Lambda$ is non-singular.
\end{enumerate}
\end{theorem}

\section{Preliminary results}\label{sec:pre}

Throughout this section $A = K Q/I$ is a $2$-regular algebra
and $J = J(A)$.
Moreover, let $B = A/J^3$.
Then
$B$ is a $2$-regular algebra with the bound quiver presentation
$B = K Q/(R_Q^3 + I)$.
We start with the following known lemma.

\begin{lemma}
\label{lem:4.1}
Let $A$ be a symmetric algebra, $i$ a vertex of $Q$,
$\alpha, \bar{\alpha}$ the arrows of $Q$ starting at $i$,
$\delta, \delta^*$ the arrows of $Q$ ending at $i$,
and $x = s(\delta)$, $y = s(\delta^*)$.
Then the following statements hold.
\begin{enumerate}[(i)]
 \item
  $e_i J = \alpha A + \bar{\alpha} A$.
 \item
  $P_i/S_i$ is isomorphic to the submodule
  $(\delta,\delta^*)A$ of $P_x \oplus P_y$.
\end{enumerate}
\end{lemma}

\begin{proposition}
\label{prop:4.2}
Let $A$ be a tame algebra, $i$ a vertex of $Q$,
$\alpha, \bar{\alpha}$ the arrows of $Q$ starting at $i$,
and
$\delta, \delta^*$ the arrows of $Q$ ending at $i$.
Then the following statements hold.
\begin{enumerate}[(i)]
 \item
  There are two different arrows $\beta'$, $\gamma'$
  in $Q$ such that the paths $\alpha \beta'$,
  $\bar{\alpha} \gamma'$ are defined and involved
  in two independent minimal relations of $I$.
 \item
  There are two different arrows $\sigma'$, $\xi'$
  in $Q$ such that the paths $\sigma' \delta$,
  $\xi' \delta^*$ are defined and involved
  in two independent minimal relations of $I$.
\end{enumerate}
\end{proposition}

\begin{proof}
We will prove only (i), because the proof of (ii)
is dual.
Let $j = t(\alpha)$, $k = t(\bar{\alpha})$,
$\beta_1, \beta_2$ are the arrows of $Q$ starting at $j$,
and
$\gamma_1, \gamma_2$ are the arrows of $Q$ starting at $k$.

We first claim that one of the paths $\alpha \beta_1$
or $\alpha \beta_2$ is involved in a minimal relation of $I$.
Suppose that none of
$\alpha \beta_1$ and $\alpha \beta_2$ is involved
in a minimal relation of $I$.
Then $\alpha \beta_1 + J^3$ and $\alpha \beta_2 + J^3$ are
independent elements of $B = A/J^3$.
Consider the bound quiver algebra $D = K Q/L$, where
$L$ is the ideal of $K Q$ generated by all paths
of length two in $Q$ except $\alpha \beta_1$ and $\alpha \beta_2$.
Clearly, $D$ is a quotient algebra of $B$, and hence
a tame algebra. On the other hand, $D$ admits a Galois
covering $F : \tilde{D} \to \tilde{D}/G$,
with a finitely generated free group $G$, such that $\tilde{D}$
contains a full convex subcategory isomorphic to the path algebra
$K \Delta$ of the wild quiver $\Delta$
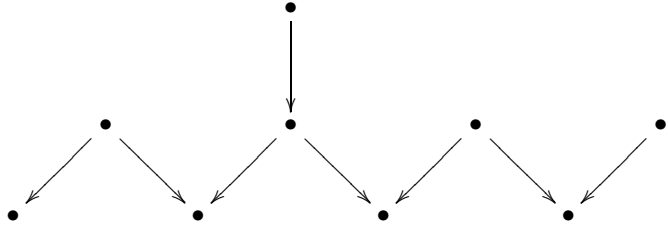
\[
 \xymatrix{
   &&& \save[] +<0pc,1pc> *+{\bullet} \ar[d] \restore \\
   & \bullet \ar[dl] \ar[dr] && \bullet \ar[dl] \ar[dr] &&
     \bullet \ar[dl] \ar[dr] && \bullet \ar[dl] \\
   \bullet && \bullet && \bullet && \bullet \\
 }
\]
of type $\tilde{\tilde{E}}_7$.
Hence $\tilde{D}$ is a wild locally bounded category
(see \cite[Theorem]{DS2}).
Then, applying \cite[Proposition~2]{DS1}, we conclude that
$D$ is a wild algebra.
Therefore, indeed one of the paths
$\alpha \beta_1$ and $\alpha \beta_2$ is involved
in a minimal relation of $I$.

Similarly, we prove that one of the paths
$\bar{\alpha} \gamma_1$ or $\bar{\alpha} \gamma_2$ is involved
in a minimal relation of $I$.

Assume that two paths $\alpha \beta_r$ and $\bar{\alpha} \gamma_s$,
for some $r,s \in \{1,2\}$, are involved in a minimal relation of $I$.
We may assume that $r = s = 1$.
Then $t(\beta_1) = t(\gamma_1)$ and there are elements
$a,b \in K^*$ such that
$a \alpha \beta_1 + b \bar{\alpha} \gamma_1 \in J^3$,
and clearly $\alpha \beta_1$ and $\bar{\alpha} \gamma_1$
are not in $J^3$.
We claim that one of the paths $\alpha \beta_2$ or
$\bar{\alpha} \gamma_2$ is involved in a minimal relation of $I$.
Suppose, for contradiction, that it is not the case.
In particular, $\alpha \beta_2$ and $\bar{\alpha} \gamma_2$
do not belong to $J^3$.
We have few cases to consider.

\smallskip

(1)
Assume that $j \neq k$.
Consider the bound quiver algebra $E = KQ / M$, where
$M$ is the ideal in $K Q$ generated by all paths
of length two in $Q$ different from
$\alpha \beta_1$, $\bar{\alpha} \gamma_1$,
$\alpha \beta_2$, $\bar{\alpha} \gamma_2$,
and $a \alpha \beta_1 + b \bar{\alpha} \gamma_1$.
Then $E$ is a quotient algebra of $B$ and there is
a Galois covering $F : \tilde{E} \to \tilde{E}/G$,
with a finitely generated free group $G$, such that $\tilde{E}$
admits a full convex subcategory  $\Lambda$ isomorphic to the
bound quiver algebra $K \Omega / N$, where the quiver $\Omega$
is of the form
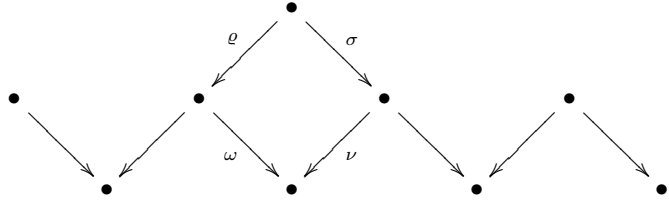
\[
 \xymatrix{
   &&& \bullet \ar[dl]_{\varrho} \ar[dr]^{\sigma} \\
   \bullet \ar[dr] && \bullet \ar[dl] \ar[dr]_{\omega} &&
     \bullet \ar[dl]^{\nu} \ar[dr] && \bullet \ar[dl] \ar[dr] \\
   & \bullet && \bullet && \bullet && \bullet \\
 }
\]
and $N$ is the ideal in $K \Omega$ generated by
$a \sigma \nu + b \varrho \omega$.
We note the $\Lambda$ is isomorphic to the bound quiver
algebra $\Lambda' = K \Omega / N'$,
where $N'$ is the ideal in $K Q$ generated by
$\sigma \nu - \varrho \omega$.
Since $\Lambda'$ is a wild concealed algebra
of type $\tilde{\tilde{E}}_7$, applying
\cite[Theorem]{DS2}) and \cite[Proposition~2]{DS1}
again, we conclude that $E$ is a wild algebra.
Therefore, indeed one of the paths
$\alpha \beta_2$ or $\bar{\alpha} \gamma_2$ is involved
in a minimal relation of $I$.

\smallskip

(2)
Assume that $j = k$ but $t(\beta_1) \neq t(\beta_2)$.
Then $\beta_1 = \gamma_1$, $\beta_2 = \gamma_2$,
and clearly these arrows are not loops.
It follows from the imposed assumption that
$\alpha \beta_2$ and $\bar{\alpha} \gamma_2 = \bar{\alpha} \beta_2$
are independent non-zero elements of $B$.
Consider the bound quiver algebra $R = K Q / U$, where
$U$ is the ideal of $K Q$ generated by all paths
in $Q$
of length two
except $\alpha \beta_2$ and $\bar{\alpha} \beta_2$.
Then $R$ is a quotient algebra of $B$ and admits
a Galois covering $F : \tilde{R} \to \tilde{R}/G = R$,
with a finitely generated free group $G$, such that $\tilde{R}$
admits a full convex subcategory isomorphic to the
path algebra $K \Sigma$ of the wild quiver $\Sigma$
of the form
\[
 \xymatrix@C=3pc{
   \bullet \ar@<-.5ex>[r] \ar@<+.5ex>[r] &
     \bullet \ar[r] & \bullet \\
 } .
\]
Then it follows from
\cite[Theorem]{DS1}) and \cite[Proposition~2]{DS1}
that $R$ is a wild algebra.
Therefore, we obtain that one of the paths
$\alpha \beta_2$ and $\bar{\alpha} \beta_2$
is involved in a minimal relation of $I$.

\smallskip

(3)
Assume that $j = k$ and $t(\beta_1) = t(\beta_2)$.
Clearly, we have $\{\beta_1,\beta_2\} = \{\gamma_1,\gamma_2\}$.
If $\beta_1 = \gamma_1$ and $\beta_2 = \gamma_2$,
then considering the bound quiver algebra $R = K Q / U$
defined in (2) we conclude that $R$ is a wild algebra.
Suppose that $\beta_1 = \gamma_2$ and $\beta_2 = \gamma_1$.
Consider the bound quiver algebra $E = K Q / M$, where
$M$ is the ideal in $K Q$ generated by all paths
of length two in $Q$ different from
$\alpha \beta_1$, $\alpha \beta_2$,
$\bar{\alpha} \beta_1$, $\bar{\alpha} \beta_2$,
and
$a \alpha \beta_1 + b \bar{\alpha} \gamma_1
 = a \alpha \beta_1 + b \bar{\alpha} \beta_2$.
Then $E$ is a quotient algebra of $B$ and there is
a Galois covering $F : \tilde{E} \to \tilde{E}/G = E$,
with a finitely generated free group $G$, such that $\tilde{E}$
admits a full convex subcategory $\Gamma$ isomorphic to the
bound quiver algebra $K \Phi / W$, where $\Phi$
is the quiver
\[
 \xymatrix@C=3pc{
   \save[] +<0pc,-3mm> *{1} \restore
   \bullet \ar@<-.5ex>[r]_{\sigma} \ar@<+.5ex>[r]^{\varrho} &
   \save[] +<0pc,-3mm> *{2} \restore
   \bullet \ar@<-.5ex>[r]_{\nu} \ar@<+.5ex>[r]^{\omega} &
   \save[] +<0pc,-3mm> *{3} \restore
   \bullet \\
 }
\]
and $W$ is the ideal in $K \Phi$ generated by
the element $a \varrho \omega + b \sigma \nu$.
Let $\Gamma' = (e_1 + e_3) \Gamma (e_1 + e_3)$.
Then $\Gamma'$ is isomorphic to the
path algebra $K \Psi$ of the wild quiver
$\Psi$ of the form
\[
 \xymatrix@C=3pc{
   \bullet \ar@<-.9ex>[r] \ar[r] \ar@<+.9ex>[r] &
     \bullet \\
 }
\]
and consequently $\Gamma'$ is a wild algebra.
Then it follows that $\Gamma$ is a wild algebra
(see \cite[Theorem]{DS2}).
Finally, applying \cite[Proposition~2]{DS1}
we conclude that $E$, and hence $B$, is a wild algebra.

\smallskip

Summing up, we proved that there exist arrows
$\beta' = \beta_r$ and $\gamma' = \gamma_s$,
for some $r,s \in \{1,2\}$, such that the paths
$\alpha \beta'$ and $\bar{\alpha} \gamma'$
are defined and involved in two independent
minimal relations of $I$.
\end{proof}

\begin{proposition}
\label{prop:4.3}
Let $A = K Q / I$ be a symmetric algebra, $i$ a vertex of $Q$,
$\alpha$, $\bar{\alpha}$ the arrows of $Q$ starting at $i$,
$\delta$, $\delta^*$ the arrows of $Q$ ending at $i$,
and $j = t(\alpha)$, $k = t(\bar{\alpha})$,
$x = s(\delta)$, $y = s(\delta^*)$.
Assume that the simple module $S_i$ in $\mod A$
is periodic of period $4$.
Then there is an exact sequence in $\mod A$
\[
  0 \to
  S_i \to
  P_i \xrightarrow{d_3}
  P_x \oplus P_y \xrightarrow{d_2}
  P_j \oplus P_k \xrightarrow{d_1}
  P_i \xrightarrow{d_0}
  S_i \to
  0 ,
\]
which give rise to a minimal projective resolution of $S_i$ in $\mod A$,
and the following statements hold.
\begin{enumerate}[(i)]
 \item
  $d_1(u,v) = \alpha u + \bar{\alpha} v$
  for any $(u,v) \in P_j \oplus P_k$.
 \item
  There exist elements
  $\varphi_1 \in e_j J e_x$,
  $\varphi_2 \in e_k J e_x$,
  $\psi_1 \in e_j J e_y$,
  $\psi_2 \in e_k J e_y$
  and two independent
  minimal relations of $I$,
  one from $i$ to $x$ and the other
  from $i$ to $y$,
  inducing the equalities in $A$
  \[
   \alpha \varphi_1 + \bar{\alpha} \varphi_2 = 0
   \quad
   \mbox{and}
   \quad
   \alpha \psi_1 + \bar{\alpha} \psi_2 = 0
   .
  \]
 \item
  The pairs
  $(\varphi_1,\varphi_2)$,
  $(\psi_1,\psi_2)$
  in $P_j \oplus P_k$
  generate $\Omega_A^2(S_i)$
  and
  $d_2(p,q) = (\varphi_1 p + \psi_1 q, \varphi_2 p + \psi_2 q)$
  for any $(p,q) \in P_x \oplus P_y$.
 \item
  There exist elements
  $\delta_1 \in e_x J e_i \setminus e_x J^2 e_i$,
  $\delta_2 \in e_y J e_i \setminus e_y J^2 e_i$
  and two independent
  minimal relations in $I$,
  one from $j$ to $i$ and the other
  from $k$ to $i$,
  inducing the equalities in $A$
  \[
   \varphi_1 \delta_1 + \psi_1 \delta_2 = 0
   \quad
   \mbox{and}
   \quad
   \varphi_2 \delta_1 + \psi_2 \delta_2 = 0
   .
  \]
 \item
  The pair $(\delta_1,\delta_2)$ in $P_x \oplus P_y$
  generates $\Omega_A^3(S_i)$
  and
  $d_3(w) = (\delta_1 w, \delta_2 w)$
  for any $w \in P_i$.
\end{enumerate}
\end{proposition}

\begin{proof}
Since $A$ is a symmetric $2$-regular algebra and $S_i$
is periodic of period four in $\mod A$, there is
an exact sequence
\[
  0 \to
  S_i \to
  P_i \xrightarrow{d_3}
  P_x \oplus P_y \xrightarrow{d_2}
  P_j \oplus P_k \xrightarrow{d_1}
  P_i \xrightarrow{d_0}
  S_i \to
  0 ,
\]
where $d_0$ is the canonical epimorphism $P_i \to S_i$
and $d_1(u,v) = \alpha u + \bar{\alpha} v$
for any $(u,v) \in P_j \oplus P_k$,
which give rise to a minimal projective resolution
of $S_i$ in $\mod A$.
Further, $\Omega_A^2(S_i) = \Ker d_1$
is generated by elements $\varphi$ and $\psi$
in $P_j \oplus P_k$ such that
$\varphi = d_2(e_x)$ and $\psi = d_2(e_y)$,
and we may define $d_2$ by
$d_2(p,q) = \varphi p + \psi q$
for any $(p,q) \in P_x \oplus P_y$.
Then we have
$\varphi = (\varphi_1,\varphi_2)$
and
$\psi = (\psi_1,\psi_2)$
for some elements
$\varphi_1 \in e_j J e_x$,
$\varphi_2 \in e_k J e_x$,
$\psi_1 \in e_j J e_y$,
$\psi_2 \in e_k J e_y$.
In particular, we have
$d_2(p,q) = (\varphi_1 p + \psi_1 q, \varphi_2 p + \psi_2 q)$
for $(p,q) \in P_x \oplus P_y$.
Therefore, we obtain the equalities in $A$
\[
  \alpha \varphi_1 + \bar{\alpha} \varphi_2
    = d_1(\varphi) =  d_1 \big( d_2 ( e_x ) \big) = 0,
\]
induced by a minimal relation $\varrho_1$ in $I$ from $i$ to $x$, and
\[
  \alpha \psi_1 + \bar{\alpha} \psi_2
    = d_1(\psi) =  d_1 \big( d_2 ( e_y ) \big) = 0,
\]
induced by a minimal relation $\varrho_2$ in $I$ from $i$ to $y$.
Obviously, the relations $\varrho_1$ and $\varrho_2$
are independent.
This proves (ii) and (iii).

The third syzygy
$\Omega_A^3(S_i) = \Ker d_2 = \Im d_3$
is isomorphic to $P_i/S_i$.
Hence, $\Omega_A^3(S_i)$
is generated by the element
$d_3(e_i) = (\delta_1,\delta_2) \in P_x \oplus P_y$.
Moreover, by
Lemma~\ref{lem:4.1},
$P_i/S_i$ is isomorphic to the submodule
$(\delta,\delta^*)A$ of $P_x \oplus P_y$.
This implies that
$\delta_1 \in e_x J e_i \setminus e_x J^2 e_i$
and
$\delta_2 \in e_y J e_i \setminus e_y J^2 e_i$.
Further, we obtain the equalities
\[
 (\varphi_1 \delta_1 + \psi_1 \delta_2,
  \varphi_2 \delta_1 + \psi_2 \delta_2)
   = \varphi \delta_1 + \psi \delta_2
   = d_2 d_3 (e_i)
   = 0,
\]
and hence the equalities in $A$
\[
  \varphi_1 \delta_1 + \psi_1 \delta_2 = 0,
\]
induced by a minimal relation $\omega_1$ in $I$ from $j$ to $i$, and
\[
  \varphi_2 \delta_1 + \psi_2 \delta_2 = 0,
\]
induced by a minimal relation $\omega_2$ in $I$ from $k$ to $i$.
Clearly, the relations $\omega_1$ and $\omega_2$
are independent.
This proves (iv) and (v).
\end{proof}

We note that if $A = K Q / I$ is a symmetric $2$-regular algebra
then the opposite algebra $A^{\op} = K Q^{\op} / I^{\op}$
is a symmetric $2$-regular algebra.
Moreover, if the simple module $S_i = e_i A / e_i J$
is periodic of period four in $\mod A$
then the simple module $S'_i = A e_i / J e_i$
in $\mod A^{\op}$ is periodic of period four.
Hence, we have the dual version of the above
proposition for the simple module $S'_i$.

We obtain the following direct consequence of
Proposition~\ref{prop:4.3} and its dual.

\begin{corollary}
\label{cor:4.4}
Let $A$ be a symmetric algebra, $i$ a vertex of $Q$,
$\alpha$, $\bar{\alpha}$ the arrows of $Q$ starting at $i$,
$\delta$, $\delta^*$ the arrows of $Q$ ending at $i$,
and $j = t(\alpha)$, $k = t(\bar{\alpha})$,
$x = s(\delta)$, $y = s(\delta^*)$.
Assume that the simple module $S_i$ in $\mod A$
is periodic of period four.
Then the following statements hold.
\begin{enumerate}[(i)]
 \item
  There are exactly two independent
  minimal relations in $I$ starting at $i$,
  one from $i$ to $x$ and the other
  from $i$ to $y$. 
 \item
  There are exactly two independent
  minimal relations in $I$ ending at $i$,
  one from $j$ to $i$ and the other
  from $k$ to $i$.
\end{enumerate}
\end{corollary}

In fact, by Proposition~\ref{prop:4.2}, in (i) above we known that
there are paths of length two from $i$
to $x$ and from $i$ to $y$ which are independent modulo
$J^2$, and the dual version for (ii) holds as well. We make therefore the
following convention.


\begin{notation}
\label{arrows} 
Let $i$ be a vertex, and $\alpha: i\to j$ and $\ba: i\to k$ be the arrows
starting at $i$, and let $\delta: x\to i$ and $\delta^*: y\to i$ 
be the arrows ending at $i$. Then we denote by $\beta_1: j\to y$ 
and $\gamma_1: k\to x$ be 
arrows such that $\alpha\beta_1$ and $\ba\gamma_1$ are terms of
a minimal relation.  
\end{notation}

\begin{corollary}
\label{cor:4.6}
Let $A$ be a symmetric algebra, $i$ a vertex of $Q$,
$\alpha, \bar{\alpha}$ the arrows of $Q$ starting at $i$,
and $\delta, \delta^*$ the arrows of $Q$ ending at $i$.
Assume that the simple module $S_i$ in $\mod A$
is periodic of period four.
Then the following statements hold.
\begin{enumerate}[(i)]
 \item
  There are two arrows $\beta$, $\gamma$ in $Q$
  such that the cosets $\alpha \beta + e_i J^3$
  and $\bar{\alpha} \gamma + e_i J^3$
  form a basis of $e_i J^2 / e_i J^3$.
 \item
  There are two arrows $\sigma$, $\varrho$ in $Q$
  such that the cosets $\sigma \delta + J^3 e_i$
  and $\varrho \delta^* + J^3 e_i$
  form a basis of $J^2 e_i / J^3 e_i$.
\end{enumerate}
In particular, we have
$\dim_K e_i J^2 / e_i J^3 = 2$
and
$\dim_K J^2 e_i / J^3 e_i = 2$.
\end{corollary}

\begin{proof}
(i)
It follows from
Proposition~\ref{prop:4.2}
that there exist two different arrows
$\beta_1$, $\gamma_1$ in $Q$ such that the paths
$\alpha \beta_1$ and $\bar{\alpha} \gamma_1$
are defined and involved in two independent
minimal relations of $I$.
Moreover, by
Corollary~\ref{cor:4.4},
there are exactly two independent
minimal relations in $I$
starting at $i$.
Then, for the arrows $\beta = \bar{\beta}_1$
and $\gamma = \bar{\gamma}_1$,
the cosets $\alpha \beta + J^3$
and $\bar{\alpha} \gamma + J^3$
form a basis of $e_i J^2 / e_i J^3$.

The proof of (ii) is similar, and applies
Proposition~\ref{prop:4.3}
and
Corollary~\ref{cor:4.4}.
\end{proof}

Let $A$ be a tame algebra.
We say that two paths
$\alpha_1 \beta_1$ and $\alpha_2 \beta_2$
of length two in $Q$ belong to a
\emph{minimal relation of type C} 
if none of $\alpha_1 \beta_1$ and $\alpha_2 \beta_2$
belongs to $J^3$ but
$c_1 \alpha_1 \beta_1 + c_2 \alpha_2 \beta_2 \in J^3$
for some non-zero elements $c_1$ and $c_2$ in $K$.
We note that then
$c_1 \alpha_1 \beta_1 + c_2 \alpha_2 \beta_2$
is a summand of a minimal relation of $I$.
A path $\alpha_1 \beta_1$ of length two in $Q$ belongs to a
\emph{minimal relation of type Z} 
if $\alpha_1 \beta_1 \in J^3$.
Clearly, $\alpha_1 \beta_1$
is a summand of a minimal relation of $I$.

\medskip

We show first that some arrows related to a type C relation cannot
be double arrows, or loops.

\begin{lemma} \label{lem:4.7}Assume we have a type C relation $a_1(\alpha\beta) + a_2(\ba\gamma_1) \in J^3$
 starting
at vertex $i$, and either a type C relation from $i$ of the form
$b_1(\alpha\beta_1) + b_2(\ba \gamma) \in J^3
$
or $\ba\gamma_1\in J^3$. Then
\begin{enumerate}[(i)]
\item 
   There are no double arrows starting  or ending at $i$, unless $Q=Q^M$ 
(see Section~\ref{sec:markov}).
\item There are no loops starting or ending at $i$.
\end{enumerate}
\end{lemma}


\begin{proof} 
We use Notation~\ref{arrows}.

\medskip

(i) \ Assume $\alpha, \ba$ are double arrows, and $Q\neq Q^M$.

Assume first  we have
two independent relations of type $C$ from $i$.
The arrows from $j$ are $\{ \beta, \beta_1\}$ and the arrows from
$k$ are $\{ \gamma, \gamma_1\}$, and we have $j=k$. Since $Q\neq Q^M$, we must
have $x\neq y$, and  it follows that $\beta=\gamma_1$ and $\beta_1=\gamma$. 
Then we can rewrite the relations as
$$
  (a_1\alpha + a_2\ba)\beta \in J^3, \ \ (b_1\alpha + b_2\ba)\gamma \in J^3.
$$
If $(a_1\alpha + a_2\ba)$ and $(b_1\alpha + b_2\ba)$ are dependent
modulo $J^3$ then this is an arrow $\alpha': i\to j$ such that
$\alpha'J \subseteq  J^3$. 
This contradicts Corollary~\ref{cor:4.6},
since with $\alpha'J\subseteq  J^3$ there cannot be a basis of $e_iA/e_iJ^2$
as we established in Corollary~\ref{cor:4.6}.

Therefore the two elements must be independent arrows. Then we
may replace $\alpha, \ba$ and have now two independent
relations of type $Z$.

\smallskip

Now
assume that the second minimal relation from $i$ is
$\ba \gamma_1 \in J^3$.
Then
we can write the first relation as $(b_1\alpha + b_2\ba)\gamma \in J^3$.
We can take independent arrows $i\to j$ to be $\ba$ and $b_1\alpha + b_2\ba$,
and then we have two relations of type $Z$.

\smallskip

Suppose $\delta, \delta^*$ are double arrows, then $x=y$ and $j\neq k$.
Assume first we have two relations of type $C$, then  $\beta=\beta_1$,
and $\gamma = \gamma_1$. 
Since the two relations are independent, the
coefficient matrix is invertible. 
It follows that both
$\alpha\beta$ and $\ba\gamma$ are in $J^3$, and we do not have
type C relations, a contradiction.

Now suppose we have only one type C relation from $i$, and assume $x=y$.
 Here we have double arrows $\gamma, \gamma_1: k\to y$. The
arrow $\beta_1$ also ends at $y$, and it follows that $\beta_1=\gamma$.
This means that $j=k$ and then
also $\alpha, \ba$ are
double arrows and $Q$ is the Markov quiver $Q^M$, a contradiction.

\medskip

(ii) \ Assume we have two relations of type C from $i$. Suppose say $\alpha$ is
a loop, then $j=i$ and  $k\neq i$, and we have
$\{ \beta, \beta_1\} = \{\alpha, \ba\}$.
The situation is symmetric, so we may assume $\alpha=\beta$ and $\ba = \beta_1$. This implies $x=i$ and $y=k$, and the subquiver
with vertices $i, k$ is 2-regular, and hence is $Q$, and $|Q_0|=2$, a contradiction.

Now assume that we have only one type C relation from $i$, and $\ba\gamma_1\in J^3$.
Assume $\ba$ is a loop, so that $k=i$ and $\{ \gamma, \gamma_1\} = \{ \alpha, \ba\}$.
If $\gamma = \alpha$ and $\gamma_1=\ba$ then $y=j$ and $x=i$, and
we have again a 2-regular subquiver with two vertices which is all of $Q$, a
contradiction.
If $\gamma = \ba$ and $\gamma_1=\alpha$, then $y=i$ and $x=j$, and we
have the same contradiction.

\medskip
Suppose $\alpha$ is a loop, so that $j=i$ and $\{ \beta_1, \beta\} = 
\{ \alpha, \ba\}$.

Assume $\beta_1=\alpha$, then $y=i$. The  arrows ending at $i$
are $\{\delta, \delta^*, \gamma, \alpha\}$. It follows that
$\delta^*=\alpha$ and $\gamma=\delta$, and then we have
a 2-regular subquiver with two vertices $i, k$, a contradiction.
Otherwise $\beta_1=\ba$, then $k=y$ and $\gamma$ is a loop, and we
have a 2-regular subquiver with vertices $i, k$ and again a contradiction.

If there is a loop ending at $i$ then it also starts at $i$, and we have
excluded this.
\end{proof}


The following proposition explains how relations of type C propagate in the
quiver. Assuming part (i) of the following, then the previous lemma implies that also 
$\beta, \beta_1$ and $\gamma, \gamma_1$ 
are not double arrows or loops. 
Similarly, part (ii)
of the following and the previous lemma implies that $\gamma, \gamma_1$ are
not double arrows or loops. 


\begin{proposition}
\label{prop:4.8}
Let $A$ be a tame symmetric algebra and $i$ a vertex in $Q$
such that the simple module $S_i$ in $\mod A$ is periodic of period four.
Moreover, let
$\alpha, \bar{\alpha}$ be the arrows of $Q$ starting at $i$,
$\delta, \delta^*$ the arrows of $Q$ ending at $i$,
and $j = t(\alpha)$, $k = t(\bar{\alpha})$,
$x = s(\delta)$, $y = s(\delta^*)$.
Assume also that $Q$ is not the Markov quiver. Then the following hold.
\begin{enumerate}[(i)]
 \item
  Assume that there are two minimal relations in $A$ of type C
  \[
    a_1 \alpha \beta + a_2 \bar{\alpha} \gamma_1 \in J^3
     \quad
     \mbox{and}
     \quad
    b_1 \alpha \beta_1 + b_2 \bar{\alpha} \gamma \in J^3
    ,
  \]
  where $\beta$, $\gamma_1$ are arrows ending at $x$
  and $\beta_1$, $\gamma$ are arrows ending at $y$.
  Then there are two minimal relations in $A$ of type
  C
  of the form
  \[
    c_1 \beta \delta + c_2 \beta_1 \delta^* \in J^3
     \quad
     \mbox{and}
     \quad
    d_1 \gamma_1 \delta + d_2 \gamma \delta^* \in J^3
    .
  \]
 \item
  Assume that there is a unique minimal relation in $A$ of type C
  starting at $i$
  \[
    b_1 \alpha \beta_1 + b_2 \bar{\alpha} \gamma \in J^3
    ,
  \]
  where $\beta_1$, $\gamma$ are arrows ending at $y$,
  and $\bar{\alpha} \gamma_1 \in J^3$ is the relation
  in $A$ of type Z starting at $i$.
  Then there are in $A$ a minimal relation of type C
  \[
    d_1 \gamma_1 \delta + d_2 \gamma \delta^* \in J^3
    ,
  \]
  and the relation $\beta_1 \delta^* \in J^3$ of type Z.
 \item
  Assume that there are in $A$ two minimal relations of type Z
  \[
    \alpha \beta_1 \in J^3
     \quad
     \mbox{and}
     \quad
    \bar{\alpha} \gamma_1 \in J^3
    ,
  \]
  starting from $i$.
  Then there are in $A$ two minimal relations of type Z
  \[
    \beta_1 \delta^* \in J^3
     \quad
     \mbox{and}
     \quad
    \gamma_1 \delta \in J^3
    ,
  \]
  ending at $i$. 
\end{enumerate}
\end{proposition}

\begin{proof}
It follows from
Proposition~\ref{prop:4.3}
that there exist elements
$\varphi_1 \in e_j J e_x$,
$\varphi_2 \in e_k J e_x$,
$\psi_1 \in e_j J e_y$,
$\psi_2 \in e_k J e_y$
and two independent
minimal relations in $I$,
one from $i$ to $x$ and the other
from $i$ to $y$,
which induce the equalities in $A$
\[
   \alpha \varphi_1 + \bar{\alpha} \varphi_2 = 0
   \quad
   \mbox{and}
   \quad
   \alpha \psi_1 + \bar{\alpha} \psi_2 = 0
   .
\]
Further, there exist elements
$\delta_1 \in e_x J e_i \setminus e_x J^2 e_i$,
$\delta_2 \in e_y J e_i \setminus e_y J^2 e_i$,
and two independent
minimal relations in $I$,
one from $j$ to $i$ and the other
from $k$ to $i$,
which induce the equalities in $A$
\[
   \varphi_1 \delta_1 + \psi_1 \delta_2 = 0
   \quad
   \mbox{and}
   \quad
   \varphi_2 \delta_1 + \psi_2 \delta_2 = 0
   .
\]
Assume first that there are no double arrows ending at vertex $i$, that is
$x\neq y$. (This holds in (i) and (ii) anyway).  Then 
there exist non-zero elements $\lambda_1,\lambda_2$ in $K$
such that
$\delta_1 = \lambda_1 \delta + t_1$
and
$\delta_2 = \lambda_2 \delta^* + t_2$,
for some elements $t_1, t_2$ from $J^2$.

\smallskip

(i)
It follows from the  assumption that
$s(\beta) = j = s(\beta_1)$,
$s(\gamma) = k = s(\gamma_1)$,
and
$\varphi_1 = a_1 \beta + r_1$,
$\varphi_2 = a_2 \gamma_1 + r_2$,
$\psi_1 = b_1 \beta_1 + s_1$,
$\psi_2 = b_2 \gamma + s_2$,
for some elements
$r_1, r_2, s_1, s_2$. We claim that these
must be  in $J^2$: If (say) $r_1$ is not in $J^2$ then it is an
arrow $j\to x$ not equal to $\beta$ and we have
double arrows ending at $x$, and $\{ r_1, \beta\} = \{ \beta, \gamma_1\}$ and $j=k$
and then $Q$ is the Markov quiver, a contradiction. Similarly $r_2, s_1, s_2$ must be in $J^2$. 
Then $r_1\delta_1 + s_2\delta_2 \in J^3$
and $r_2\delta_1 + s_2\delta_2 \in J^3$. 
Clearly, $a_1, a_2, b_1, b_2$
are non-zero elements of $K$.
We obtain the equalities in $A$
\begin{align*}
 (a_1 \beta + r_1)  (\lambda_1 \delta + t_1)
 + (b_1 \beta + s_1)  (\lambda_2 \delta^* + t_2)
 &= 0 , \\
 (a_2 \gamma_1 + r_2)  (\lambda_1 \delta + t_1)
 + (b_2 \gamma + s_2)  (\lambda_2 \delta^* + t_2)
 &= 0 .
\end{align*}
Then we conclude that
\[
    c_1 \beta \delta + c_2 \beta_1 \delta^* \in J^3
     \quad
     \mbox{and}
     \quad
    d_1 \gamma_1 \delta + d_2 \gamma \delta^* \in J^3
    ,
\]
for
$c_1 = a_1 \lambda_1$,
$c_2 = b_1 \lambda_2$,
$d_1 = a_2 \lambda_1$,
$d_1 = b_2 \lambda_2$.

\smallskip

(ii)
The assumption on the relation of type C
implies that
$\psi_1 = b_1 \beta_1 + s_1$
and
$\psi_2 = b_2 \gamma + s_2$,
for some elements $s_1, s_2$ which are in $J^2$, as one sees by
the same argument as in (i). 
On the other hand, since there is a minimal relation
$\bar{\alpha} \gamma_1 \in J^3$ of type Z,
we conclude that $t(\gamma_1) = x$
(see Corollary~\ref{cor:4.4}).
Then we obtain
$\varphi_2 = a \gamma_1 + r$,
for a non-zero element $a \in K$
and some $r \in J^2$, and
$\varphi_1 \in e_j J^2 e_x$ such that $\alpha \vf_1 + \ba r \in J^3$. 
We have the equalities in $A$
\begin{align*}
 \varphi_1 (\lambda_1 \delta + t_1)
 + (b_1 \beta_1 + s_1) (\lambda_2 \delta^* + t_2)
 &= 0 , \\
 (a \gamma_1 + r) (\lambda_1 \delta + t_1)
 + (b_2 \gamma + s_2) (\lambda_2 \delta^* + t_2)
 &= 0 .
\end{align*}
From the first equality, we obtain
$b_1 \lambda_2 \beta_1 \delta^* \in J^3$, 
and hence the minimal relation
$\beta_1 \delta^* \in J^3$ of type Z in $A$,
because $b_1 \lambda_2 \neq 0$.
The second equality implies the minimal relation
$d_1 \gamma_1 \delta + d_2 \gamma \delta^* \in J^3$
of type C, for
$d_1 = a \lambda_1$
and
$d_2 = b_2 \lambda_2$.

\smallskip

(iii)
The imposed assumption implies that
$\varphi_2 = a \gamma_1 + r$,
$\psi_1 = b \beta_1 + s$,
for some elements $a,b$ in $K$
and some elements $r, s \in J^3$,
and
$\varphi_1 \in e_j J e_x$,
$\psi_2 \in e_k J e_y$.
We have the equalities in $A$
\begin{align*}
 \varphi_1 (\lambda_1 \delta + t_1)
 + (b \beta_1 + s) (\lambda_2 \delta^* + t_2)
 &= 0 , \\
 (a \gamma_1 + r) (\lambda_1 \delta + t_1)
 + \psi_2 (\lambda_2 \delta^* + t_2)
 &= 0 .
\end{align*}
Then we conclude that
$b \lambda_2 \beta \delta^* \in J^3$
and
$a \lambda_1 \gamma_1 \delta \in J^3$,
with
$b \lambda_2 \neq 0$
and
$a \lambda_1 \neq 0$.
Therefore, we obtain the minimal relations
$\beta_1 \delta^* \in J^3$
and
$\gamma_1 \delta \in J^3$
of type $Z$
ending at $i$.

Now we consider the possibility that $\delta, \delta^*$ are double arrows.
By the previous lemma,  we must have 
two type Z relations.
Assume $\alpha\beta_1\in J^3$ and $\ba\gamma_1\in J^3$. 
Then $\beta_1, \gamma_1$ are the arrows ending at $x=y$. 

By Corollaries \ref{cor:4.4} and \ref{cor:4.6} (applied to the two arrows ending at $i$)
we get: one of $\beta_1\delta, \gamma_1\delta$ is involved in a minimal relation. 
As well one of
$\beta_1\delta^*, \gamma_1\delta^*$ is involved in a minimal relation.
Furthermore, either $\gamma_1\delta^*$ and $\beta_1\delta$ 
are independent modulo $J^3$, or
$\gamma_1\delta$ and $\beta_1\delta^*$ are independent modulo $J^3$.
Combining these leave two possibilities:
\begin{enumerate}[(I)]
 \item
  $\beta_1\delta^*$ and $\gamma_1\delta$ are involved in minimal relations, 
  and $\gamma_1\delta^*$ and $\beta_1\delta$ are
independent modulo $J^3$, or
 \item
  $\beta_1\delta$ and $\gamma_1\delta^*$ are involved in minimal relations, 
  and $\gamma_1\delta$ and $\beta_1\delta^*$ are
  independent modulo $J^3$.
\end{enumerate}
The two possibilities differ only in the names of the arrows and we may assume (I) holds.
\end{proof}

We will need the following lemma.

\begin{lemma}
\label{lem:4.9}
Let $A = K Q / I$ be a $2$-regular algebra of generalized
quaternion type.
Let $i$ be a vertex of $Q$,
$\alpha, \bar{\alpha}$ the arrows of $Q$ starting at $i$,
$\beta_1, \beta_2$ the arrows of $Q$ starting at $j = t(\alpha)$,
$\gamma_1, \gamma_2$ the arrows of $Q$ starting at $k = t(\bar{\alpha})$,
and $\sigma = \alpha^*$, $\varrho = \bar{\alpha}^*$.
Suppose that $c_1 \alpha\beta_1 + c_2 \bar{\alpha}\gamma_1 \in J^3$
is a unique minimal relation of type C starting at $i$,
and a unique minimal relation of type C ending at
$t(\beta_1) = t(\gamma_1)$.
Then the following statements hold.
\begin{enumerate}[(i)]
 \item
  Precisely one of the paths
  $\alpha \beta_2$ and $\bar{\alpha} \gamma_2$
  is involved in a minimal relation of $A$.
 \item
  Precisely one of the paths
  $\sigma \beta_1$ and $\varrho \gamma_1$
  is involved in a minimal relation of $A$.
 \item
  At least one of the paths
  $\sigma \beta_1$ and $\bar{\alpha} \gamma_2$
  is involved in a minimal relation of $A$.
\end{enumerate}
Hence, $\varrho \gamma_1$
is involved in a minimal relation of $A$
if and only if $\bar{\alpha} \gamma_2$
is involved in a minimal relation of $A$.
\end{lemma}

\begin{proof}
The statements (i) and (ii) follow from the imposed assumption
and the statements (i) and (ii) of Proposition~\ref{prop:4.2}.
Suppose that $\sigma \beta_1$ and $\varrho \gamma_1$
are not involved in minimal relations.
In particular, $\sigma \beta_1$ and $\varrho \gamma_1$
do not belong to $J^3$.
Consider the bound quiver algebra $R = K Q / M$, where
$M$ is the ideal in $K Q$ generated by all paths
of length two in $Q$ different from
$\alpha \beta_1$, $\sigma \beta_1$,
$\bar{\alpha} \gamma_1$, $\bar{\alpha} \gamma_2$,
and
$c_1 \alpha \beta_1 + c_2 \bar{\alpha} \gamma_1$.
Then $R$ is a quotient algebra of $B = A/J^3$
and there is a Galois covering $F : \tilde{R} \to \tilde{R}/G$,
with a finitely generated free group $G$, such that $\tilde{R}$
admits a full convex subcategory  $\Lambda$ isomorphic to the
bound quiver algebra $K \Omega / N$, where $\Omega$ is the quiver
of the form
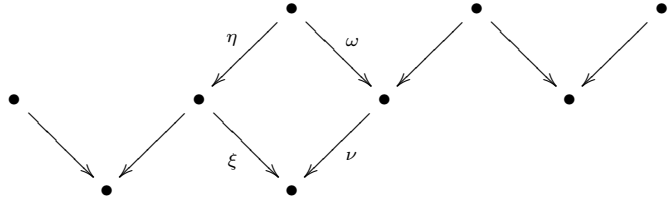
\[
 \xymatrix{
   &&& \bullet \ar[dl]_{\eta} \ar[dr]^{\omega} &&
      \bullet \ar[dl] \ar[dr] && \bullet \ar[dl] \\
   \bullet \ar[dr] && \bullet \ar[dl] \ar[dr]_{\xi} &&
     \bullet \ar[dl]^{\nu} && \bullet \\
   & \bullet && \bullet \\
 }
\]
and $N$ is the ideal in $K \Omega$ generated by
$c_1 \omega \nu + c_2 \eta \xi $.
We note the $\Lambda$ is isomorphic to the bound quiver
algebra $\Lambda' = K \Omega / N'$,
where $N'$ is the ideal in $K Q$ generated by
$\omega \nu - \eta \xi $.
Since $\Lambda'$ is a wild concealed algebra
of type $\tilde{\tilde{E}}_7$, applying
\cite[Proposition~2]{DS1} and \cite[Theorem]{DS2}
we conclude that $R$ is a wild algebra.
Therefore, indeed at least one of the paths
$\sigma \beta_1$ and $\bar{\alpha} \gamma_2$
is involved in a minimal relation of $A$.
The final statement follows from (ii) and (iii).
\end{proof}

\begin{corollary}\label{cor:4.10} 
Assume that at any vertex there is at most one type C relation. 
Then the type C relations come
in triples. 
Assume the  minimal relations from $i$ are $\alpha\beta_1 + \ba \gamma\in J^3$ 
of type $C$ and $\ba\gamma_1\in J^3$. 
Then,
with the notation as in Proposition~\ref{prop:4.8} 
and $\delta^*=\delta_2$ and $\delta = \delta_1$, we have
\begin{enumerate}[(i)]
\addtocounter{enumi}{1}
 \item
 The minimal relations from  $k$ are $\gamma_1\delta + \gamma\delta_2 \in J^3$ of type $C$ and $\gamma\bar{\delta}_2 \in J^3$.
 \item
 The minimal relations from $y$ are 
$\delta_2\ba + \bar{\delta}_2\mu\in J^3$ of type $C$  and $\delta_2\alpha \in J^3$ where $\mu: t(\bar{\delta}_2) \to k$. 
\end{enumerate}
\end{corollary}

\begin{proof} 
With the assumption on the relations from vertex $i$ we get from 
Proposition~\ref{prop:5.6} part (ii) that 
$\gamma_1\delta + \gamma\delta^* \in J^3$ is a type $C$ relation, and
also that $\beta_1\delta^*\in J^3$. We apply Lemma~\ref{lem:4.9} with relations
from $k$. 
It shows that since $\beta_1\delta_2$ is in $J^3$, 
it follows that $\gamma\bar{\delta}_2 \in J^3$. 
This proves (ii), and part (iii) follows similarly. 
[Note that we cannot have double arrows $i\to j$ with the assumption.]
\end{proof}

We describe this situation by the following  covering of the quiver.
We draw an arc to indicate that the product of the arrows is in $J^3$. 
Each square describes a type C relation.
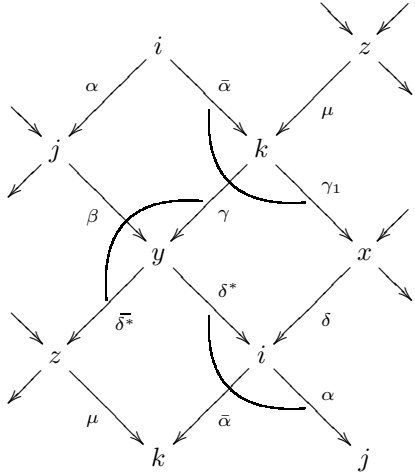
\[
\xymatrix@!=.125pc@C=1.5pc@R=1.5pc{
	&&&&&& \ar[rd] && \ar[ld] \\
	&&& i \ar[lldd]_{\alpha} \ar[rrdd]^{\bar{\alpha}}
	&&&& z \ar[lldd]^{\mu} \ar[rd]\\
	\ar[rd] &&&& \ar@{-}@/_4ex/[rrdd] &&&& \\
	& j \ar[ld] \ar[rrdd]_{\beta}
	&&&& k \ar[lldd]^{\gamma} \ar[rrdd]^{\gamma_1}\\    	
	&&&& \ar@{-}@/_4ex/[lldd] &&&& \ar[ld] \\
	&&& y \ar[lldd]^{\overbar{\delta^*}} \ar[rrdd]^{\delta^*}
	&&&& x \ar[lldd]^{\delta} \ar[rd]\\
	\ar[rd] &&&& \ar@{-}@/_4ex/[rrdd] &&&& \\
	& z \ar[ld] \ar[rrdd]_{\mu}
	&&&& i \ar[lldd]^{\bar{\alpha}} \ar[rrdd]^{\alpha}\\    	
	&&&&&& \\
	&&& k &&&& j 
}
\]

\section{Algebras with the Markov quiver}\label{sec:markov}

The following quiver $Q^M$
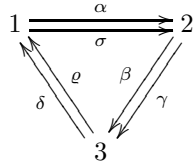
\[
  \xymatrix@R=3.pc@C=1.8pc{
    1
    \ar@<.45ex>[rr]^{\alpha}
    \ar@<-.45ex>[rr]_{\sigma}
    && 2
    \ar@<.45ex>[ld]^{\gamma}
    \ar@<-.45ex>[ld]_{\beta}
    \\
    & 3
    \ar@<.45ex>[lu]^{\delta}
    \ar@<-.45ex>[lu]_{\varrho}
  }
\]
is said to be the \emph{Markov quiver}
(see \cite{L}, \cite{M}
for justification of this name).
The quiver $Q^M$ has a canonical extension
to a triangulation quiver $(Q^M,f)$
for the permutation $f$ of arrows of $Q^M$
having the $f$-orbits $(\alpha \, \gamma \, \delta)$
and $(\sigma \, \beta \, \varrho)$.
Then the associated permutation $g$ of arrows
of $Q^M$ has only one $g$-orbit
$(\alpha \, \beta \, \delta \, \sigma \, \gamma \, \varrho)$.
Hence a weight function
$m_{\bullet} : \cO(g) \to \bN^*$
and a parameter function
$c_{\bullet} : \cO(g) \to K^*$
are given by a positive integer $m$ and
a parameter $c \in K^*$.
The associated weighted triangulation algebra
$\Lambda(Q^M,f,m_{\bullet},c_{\bullet})$
is given by the Markov quiver $Q^M$ and the
relations
\begin{align*}
 \alpha\gamma &= c (\sigma\gamma\varrho\alpha\beta\delta)^{m-1} \sigma\gamma\varrho\alpha\beta ,
 &
 \sigma\beta &= c (\alpha\beta\delta\sigma\gamma\varrho)^{m-1} \alpha\beta\delta\sigma\gamma,
 &
 \alpha\gamma\varrho &= 0,
 &
 \sigma\beta\delta &= 0,
\\
 \gamma\delta &= c (\beta\delta\sigma\gamma\varrho\alpha)^{m-1} \beta\delta\sigma\gamma\varrho ,
 &
 \beta\varrho &= c(\gamma\varrho\alpha\beta\delta\sigma)^{m-1} \gamma\varrho\alpha\beta\delta,
 &
 \gamma\delta\sigma &= 0,
 &
 \beta\varrho\alpha &= 0,
\\
 \delta\alpha &= c (\varrho\alpha\beta\delta\sigma\gamma)^{m-1} \varrho\alpha\beta\delta\sigma,
 &
 \varrho\sigma &= c (\delta\sigma\gamma\varrho\alpha\beta)^{m-1} \delta\sigma\gamma\varrho\alpha,
 &
 \delta\alpha\beta &= 0,
 &
 \varrho\sigma\gamma &= 0.
\end{align*}
We also note that the Markov triangulation quiver $(Q^M,f)$
is a triangulation quiver $(Q(\bS,\vec{T}),f)$
associated to the triangulation $T$ of the sphere $\bS$
in $\bR^3$ given by two unfolded triangles,
having a coherent orientation $\vec{T}$
(see \cite[Example~4.4]{ES4}).

The aim of this section is to prove the following theorem.

\begin{theorem}
\label{th:5.1}
Let $A$ be an algebra of generalized quaternion
type whose quiver is the Markov quiver $Q^M$.
Then $A$ is isomorphic to a weighted triangulation algebra
$\Lambda(Q^M,f,m_{\bullet},c_{\bullet})$.
\end{theorem}

We split the proof of the theorem into several steps.

We start with the following general lemma.

\begin{lemma}
\label{lem:5.2}
Let $A$ be a $2$-regular algebra of generalized quaternion
type whose quiver $Q$ has double arrows
\[
 \xymatrix@C=3pc{
   i \ar@<-.5ex>[r]_{\bar{\alpha}} \ar@<+.5ex>[r]^{\alpha} &
   j \ar@<-.5ex>[r]_{\beta} \ar@<+.5ex>[r]^{\gamma} &
   k \\
 } .
\]
Then $A$ admits a bound quiver presentation
$A = K Q/I$ such that, up to labelling,  
\begin{enumerate}[(i)]
 \item
  $\alpha\beta + e_i J^3$ and $\bar{\alpha}\gamma + e_i J^3$
  form a basis of $e_i J^2 / e_i J^3$.
 \item
  $\alpha\gamma \in J^3$ and $\bar{\alpha}\beta \in J^3$.
\end{enumerate}
Moreover, $Q$ is isomorphic to the Markov quiver $Q^M$.
\end{lemma}

\begin{proof}
Let $U = e_i J / e_i J^2$,
$V = e_i J^2 / e_i J^3$,
and
$W = e_j J / e_j J^2$.
Since the quiver $Q$ is $2$-regular,
we have $\dim_K U = 2$ and $\dim_K W = 2$.
Moreover, it follows from
Corollary~\ref{cor:4.6}
that $\dim_K V = 2$.
We choose representatives of $\alpha$ and $\bar{\alpha}$
in $A$ such that $\alpha+J^2$ and $\bar{\alpha}+J^2$
give a basis of $U$.
Consider the $K$-linear maps
$\alpha(-), \bar{\alpha}(-) : W \to V$
such that
$\alpha(\omega + e_j J^2) = \alpha \omega + e_i J^3$
and
$\bar{\alpha}(\omega + e_j J^2) = \bar{\alpha} \omega + e_i J^3$
for any $\omega \in e_j J$.
It follows from
Corollary~\ref{cor:4.6}(i)
that we may choose representatives of $\beta$ and $\gamma$
such that $\beta + e_jJ^2$ and $\gamma + e_j J^2$
form a basis of $W$, and $\alpha\beta + e_i J^3$
and $\bar{\alpha}\gamma + e_i J^3$
form a basis of $V$.
In particular, $\alpha(-)$ and $\bar{\alpha}(-)$
are non-zero homomorphisms from $W$ to $V$.
We also note that $V$ is the sum of the images of
$\alpha(-)$ and $\bar{\alpha}(-)$.
We claim that none of
$\alpha(-)$ and $\bar{\alpha}(-)$
is an isomorphism.
Indeed, suppose that $\alpha(-)$ is an isomorphism.
Then we conclude that
$\alpha\beta + e_i J^3$ and $\alpha\gamma + e_i J^3$
form a basis of $V$, which contradicts the first part of the proof of 
Proposition~\ref{prop:4.2}.
Similarly, if $\bar{\alpha}(-)$ is an isomorphism,
$\bar{\alpha}\beta + e_i J^3$ and $\bar{\alpha}\gamma + e_i J^3$
form a basis of $V$, which is again a contradiction to
Proposition~\ref{prop:4.2}.
Therefore, the both
$\alpha(-)$ and $\bar{\alpha}(-)$
have rank one.
Taking now representatives of $\beta$ and $\gamma$
such that
$\beta + e_jJ^2 \in \Ker \bar{\alpha}(-)$
and
$\gamma + e_j J^2 \in \Ker {\alpha}(-)$,
we conclude that
$\{\beta + e_jJ^2,\gamma + e_j J^2\}$
forms a basis of $W$,
$\{\alpha\beta + e_i J^3, \bar{\alpha}\gamma + e_i J^3\}$
forms a basis of $V$, and
$\alpha\gamma \in J^3$, $\bar{\alpha}\beta \in J^3$.
Hence $A$ admits a bound quiver presentation $A = K Q/I$
satisfying (i) and (ii).

Let $\delta, \delta^*$ be the arrows in $Q$ ending at $i$,
and $x = s(\delta)$, $y = s(\delta^*)$.
Then it follows from
Proposition~\ref{prop:4.3}
and
Corollary~\ref{cor:4.4}
that
$\alpha\gamma \in J^3$ and $\bar{\alpha}\beta \in J^3$
are consequences of two independent minimal relations in $I$,
one from $i$ to $x$ and the other from $i$ to $y$.
Hence we obtain $x = k = y$, and consequently $Q$
is isomorphic to the Markov quiver $Q^M$.
\end{proof}

From now on we assume that
$A = K Q / I$ is an algebra of generalized quaternion type
with $Q = Q^M$, and such that
$\alpha\beta + e_1 J^3$ and $\bar{\alpha}\gamma + e_1 J^3$
form a basis of
$e_1 J^2 / e_1 J^3$,
and $\alpha\gamma \in J^3$, $\bar{\alpha}\beta \in J^3$.
We will take  $\sigma: = \bar{\alpha}$.

We may write
\begin{align*}
  \alpha\gamma &= \alpha \bar{p} + \bar{\alpha} \bar{q} ,
 &
  \bar{\alpha}\beta &= \alpha p + \bar{\alpha} q ,
\end{align*}
for $p, q, \bar{p}, \bar{q} \in e_2 J^2 e_3 =  e_2 J^4 e_3$.
Replacing $\gamma$ by $\gamma - \bar{p}$
and $\beta$ by $\beta - q$,
we may assume that
\begin{align}
  \label{eq:1}
  \alpha\gamma &= \bar{\alpha} \bar{q} ,
 &
  \bar{\alpha}\beta &= \alpha p .
\end{align}
Since
$\dim_K \alpha J / \alpha J^2 = 1$
and
$\dim_K \bar{\alpha} J / \bar{\alpha} J^2 = 1$,
we have
$\alpha J = \alpha \beta J$
and
$\bar{\alpha} J = \bar{\alpha} \gamma J$.
Hence, we may assume that $p = \beta u$ and
$\bar{q} = \gamma v$ for some $u, v \in e_3 J^3 e_3$.
Then there is an exact sequence in $\mod A$
\[
  0 \to
  S_1 \to
  P_1 \xrightarrow{d_3}
  P_3 \oplus P_3 \xrightarrow{d_2}
  P_2 \oplus P_2 \xrightarrow{d_1}
  P_1 \xrightarrow{d_0}
  S_1 \to
  0 ,
\]
with $d_1(a,b) = \alpha a + \bar{\alpha} b$
for any $(a,b) \in P_2 \oplus P_2$,
which give rise to a minimal projective resolution
of $S_1$ in $\mod A$.
Consider the elements
$\varphi_1 = \gamma$,
$\varphi_2 = -\bar{q}$,
$\psi_1 = -p$,
$\psi_2 = \beta$
in $e_2 J e_3$.
Then we have the equalities
\begin{align*}
   \alpha \varphi_1 + \bar{\alpha} \varphi_2 &= 0 ,
 &
   \alpha \psi_1 + \bar{\alpha} \psi_2 &= 0 ,
\end{align*}
and
$\varphi = (\varphi_1,\varphi_2)$,
$\psi = (\psi_1,\psi_2)$
are generators of $\Omega_A^2(S_1) = \Ker d_1$.
We claim that
{\it at least one of $\alpha\gamma$ and $\ba\beta$ is non-zero:}
Otherwise $\varphi = (\gamma, 0)$ and $\psi = (0, \beta)$ gives a 
direct sum decomposition of $\Omega^2_A(S_1)$.
[We will see later that $\alpha\gamma$ and $\ba\beta$ are both non-zero.]
Further, applying
Proposition~\ref{prop:4.3},
we conclude that there exist elements
$\delta_1, \delta_2 \in e_3 J e_1 \setminus e_3 J^2 e_1$
such that
$\delta_1 + e_3 J^2 e_1$ and $\delta_2 + e_3 J^2 e_1$
give a basis of $e_3 J e_1 / e_3 J^2 e_1$, and
\begin{align*}
   \varphi_1 \delta_1 + \psi_1 \delta_2 &= 0 ,
 &
   \varphi_2 \delta_1 + \psi_2 \delta_2 &= 0 .
\end{align*}
Hence, we obtain the equalities
\begin{align}
  \label{eq:2}
  \gamma \delta_1 &= p \delta_2 ,
 &
  \beta \delta_2 &= \bar{q} \delta_1 .
\end{align}
We note that
$p \delta_2 = \beta u \delta_2 = \beta q_{\beta}$,
$\bar{q} \delta_1 = \gamma v \delta_1 = \gamma q_{\gamma}$,
and hence
\begin{align}
  \tag{2$^*$}
  \label{eq:2*}
  \gamma \delta_1 &= \beta q_{\beta} ,
 &
  \beta \delta_2 &= \gamma q_{\gamma} .
\end{align}
for some elements $q_{\beta}, q_{\gamma} \in e_3 J^4 e_1$.

There is an exact sequence in $\mod A$
\[
  0 \to
  S_2 \to
  P_2 \xrightarrow{\partial_3}
  P_1 \oplus P_1 \xrightarrow{\partial_2}
  P_3 \oplus P_3 \xrightarrow{\partial_1}
  P_2 \xrightarrow{\partial_0}
  S_2 \to
  0 ,
\]
with $\partial_1(a,b) = \beta a + \gamma b$
for any $(a,b) \in P_3 \oplus P_3$,
which give rise to a minimal projective resolution
of $S_2$ in $\mod A$.
We set
$\varphi'_1 = - q_{\beta}$,
$\varphi'_2 = \delta_1$,
$\psi'_1 = \delta_2$,
$\psi'_2 = - q_{\gamma}$.
Then we have the equalities
\begin{align*}
   \beta \varphi'_1 + \gamma \varphi'_2 &= 0 ,
 &
   \beta \psi'_1 + \gamma \psi'_2 &= 0 ,
\end{align*}
and
$\varphi' = (\varphi'_1,\varphi'_2)$,
$\psi' = (\psi'_1,\psi'_2)$
generate $\Omega_A^2(S_2) = \Ker \partial_1$.
Applying
Proposition~\ref{prop:4.3}
again,
we conclude that there exist elements
$\alpha_1, \alpha_2 \in e_1 J e_2 \setminus e_1 J^2 e_2$
such that
$\alpha_1 + e_1 J^2 e_2$ and $\alpha_2 + e_1 J^2 e_2$
give a basis of $e_1 J e_2 / e_1 J^2 e_2$, and
\begin{align*}
   \varphi'_1 \alpha_1 + \psi'_1 \alpha_2 &= 0 ,
 &
   \varphi'_2 \alpha_1 + \psi'_2 \alpha_2 &= 0 .
\end{align*}
Then we obtain the equalities
\begin{align}
  \label{eq:3}
  \delta_2 \alpha_2 &= q_{\beta} \alpha_1 ,
 &
  \delta_1 \alpha_1 &= q_{\gamma} \alpha_2 .
\end{align}
Further, since
$\{\alpha + e_1 J^2,\bar{\alpha} + e_1 J^2\}$
and
$\{\alpha_1 + e_1 J^2,\alpha_2 + e_1 J^2\}$
are two bases of $e_1 J / e_1 J^2$,
there exist scalars $u,v,r,s \in K$ with
$u s - r v \neq 0$ such that
\begin{align}
  \label{eq:4}
  \alpha_1 + e_1 J^2 &= (u \alpha + v \bar{\alpha}) + e_1 J^2 ,
 &
  \alpha_2 + e_1 J^2 &= (r \alpha + s \bar{\alpha}) + e_1 J^2 .
\end{align}
We take now a bound quiver presentation
$A = K Q / I'$ of $A$ such that
$\delta: = \delta_1$ and $\varrho: = \delta_2$.
Then we define
\begin{align*}
  f(\alpha) &= \gamma, &
  f(\gamma) &= \delta, &
  f(\sigma) &= \beta, &
  f(\beta) &= \varrho, \\
  g(\alpha) &= \beta, &
  g(\gamma) &= \varrho, &
  g(\sigma) &= \gamma, &
  g(\beta) &= \delta.
\end{align*}

We note the following consequence of the above equalities.

\begin{corollary}
\label{cor:5.3}
We have
$\alpha f(\alpha) f^2(\alpha) = \bar{\alpha} f(\bar{\alpha}) f^2(\bar{\alpha})$.
\end{corollary}
\begin{proof}
The following equalities hold
\[
  \alpha f(\alpha) f^2(\alpha)
   = \alpha \gamma \delta_1
   = \alpha p \delta_2
   = \bar{\alpha} \beta \delta_2
   = \bar{\alpha} f(\bar{\alpha}) f^2(\bar{\alpha})
   .
\]
\end{proof}

We investigate now the structure of the local tame symmetric
algebra $R = e_1 A e_1$.
We note that $R$ has two generators
\[
  X =\alpha\beta\delta = \alpha g(\alpha) g^2(\alpha)
     \quad
     \mbox{and}
     \quad
  Y =\bar{\alpha}\gamma\varrho = \bar{\alpha} g(\bar{\alpha}) g^2(\bar{\alpha})
  .
\]
We denote by $J_R$ the radical of $R$.

\begin{lemma}
\label{lem:5.4}
\begin{enumerate}[(i)]
 \item
  We have
  $u X^2 + v X Y \in J_R^3$
  and
  $r Y X + s Y^2 \in J_R^3$.
 \item
  There is a positive integer $m$ such that
  the socle of $R$ is generated by $(X Y)^m = (Y X)^m$,
  and $R$ has a basis consisting of $e_1$ and initial
  subwords of $(X Y)^m$, $(Y X)^m$.
\end{enumerate}
\end{lemma}

\begin{proof}
(i)
We have the equalities
\begin{align*}
  (u \alpha + v \bar{\alpha})
  (\beta+\gamma)
  \big(f(\beta)+f(\gamma)\big)
   &=
  (u \alpha\beta + u \alpha\gamma + v \bar{\alpha}\beta + v \bar{\alpha}\gamma)
  (\delta_2+\delta_1)
  \\&= u X + v Y +w ,
\end{align*}
with $w \in e_1 J^6 e_1$,
because
$\alpha\beta\delta_2,
 \alpha\gamma\delta_2,
 \alpha\gamma\delta_1,
 \bar{\alpha}\beta\delta_2,
 \bar{\alpha}\beta\delta_1,
 \bar{\alpha}\gamma\delta_1 \in e_1 J^6 e_1$.
We note that
$X \alpha_1 = \alpha\beta\delta_1\alpha_1
 = \alpha\beta q_{\gamma}\alpha_2 \in e_1 J^7 e_2$,
because
$q_{\gamma} \in e_3 J^4 e_1$.
Moreover,
$(u\alpha + v\bar{\alpha}) - \alpha_1 \in e_1 J^2 e_2$.
Hence we obtain that
\[
  u X^2 + v X Y = X (u X + v Y)
  = X (u\alpha + v\bar{\alpha})
    (\beta+\gamma)
    \big(f(\beta)+f(\gamma)\big)
    - X w
\]
belongs to $e_1 J^9 e_1 = J_R^3$.
Similarly, one proves that
$r Y X + s Y^2$ belongs to $J_R^3$.

\smallskip

(ii)
Assume for a contradiction that $e_1 A$ has a
basis consisting of elements of the form
\begin{align*}
 &&
  X^i,
 &&
  Y^j,
 &&
  X^i \alpha,
 &&
  Y^j \bar{\alpha},
 &&
  X^i \alpha\beta,
 &&
  Y^j \bar{\alpha}\gamma.
 &&
\end{align*}
It follows from the identity (\ref{eq:1}) that
$\alpha\gamma = \bar{\alpha}\bar{q} \in e_1 J^5 e_3$.
We expand $\bar{\alpha}\bar{q}$ in terms of the above basis,
so we have
\[
  \alpha\gamma = \sum_{j \geq 1} c_j Y^j \bar{\alpha}\gamma
\]
for some $c_j \in K$ (possibly all $c_j=0$ in case $\alpha\gamma = 0$).
We note that the monomials of length $\geq 3$
starting with $\bar{\alpha}$ must start with $Y$.
Similarly, since
$\bar{\alpha}\beta = \alpha p \in e_1 J^5 e_3$,
there are $d_i \in K$ (which might be zero) such that
\[
  \bar{\alpha}\beta = \sum_{i \geq 1} d_i X^i \alpha\beta .
\]
Consider the elements in $e_1 J e_2\setminus e_1 J^2 e_2$
\begin{align*}
  \alpha' &= \alpha - \sum_{j \geq 1} c_j Y^j \bar{\alpha},
 &
  \bar{\alpha}' &= \bar{\alpha} - \sum_{i \geq 1} d_i X^i \alpha.
\end{align*}
Then
$\alpha' + e_1 J^2 e_2$,
$\bar{\alpha}' + e_1 J^2 e_2$
form a basis of $e_1 J e_2 / e_1 J^2 e_2$,
and
$\alpha' \gamma = 0$,
$\bar{\alpha}' \beta = 0$.
There exist an exact sequence in $\mod A$
\[
  0 \to
  S_1 \to
  P_1 \xrightarrow{d_3}
  P_3 \oplus P_3 \xrightarrow{d_2}
  P_2 \oplus P_2 \xrightarrow{d_1}
  P_1 \xrightarrow{d_0}
  S_1 \to
  0 ,
\]
with $d_1(a,b) = \alpha' a + \bar{\alpha}' b$
for $(a,b) \in P_2 \oplus P_2$,
which give rise to a minimal projective resolution
of $S_1$ in $\mod A$.
Then the elements
$\varphi = (\gamma,0)$
and
$\psi = (0,\beta)$
in $P_2 \oplus P_2$
generate $\Omega_A^2(S_1) = \Ker d_1$,
and hence
$\Omega_A^2(S_1) = \varphi A + \psi A = \varphi A \oplus \psi A$.
Since $\varphi A$ and $\psi A$
are non-zero,
we obtain a contradiction with the indecomposability
of $\Omega_A^2(S_1)$.
Therefore, the algebra $R$ cannot be spanned by powers of $X$
together with powers of $Y$,
and consequently is spanned by elements
of the form
$e_1, X, Y, X Y, Y X, \dots$

Finally, recall that $\alpha\gamma$ and $\ba \beta$ 
are not both zero; it 
follows that $e_1J^5\neq 0$ and therefore
$J_R^2 \neq 0$.
This shows that the socle of $R$ cannot be spanned by $X$ or $Y$.
Then the socle of $R$ is spanned by an element of the form
$(X Y)^m$ for some $m \geq 1$
(see \cite[Theorem~III.1]{E4}).
Moreover, then $(X Y)^m = (Y X)^m$
because $R$ is a local symmetric algebra.
\end{proof}

\begin{lemma}
\label{lem:5.5} Assume the notation as in \ref{eq:4}. 
\begin{enumerate}[(i)]
 \item If $m \geq 2$, then $r = 0$ and $v = 0$.
 \item
  If $m = 1$, then we may assume that $r = 0$
  and $v = 0$ or $v = 1$.
\end{enumerate}
\end{lemma}

\begin{proof}
We exploit the relations in part (i) of Lemma~\ref{lem:5.4}.
Suppose $r \neq 0$.
Since $Y X \notin J_R^3$ it follows that also $s \neq 0$,
and we may replace $Y X$ by $Y^2$ in the basis of $R$
described in part (ii) of Lemma~\ref{lem:5.4}.
Similarly, if $v \neq 0$ then $X Y \notin J_R^3$
implies $u \neq 0$, and we can replace $X Y$ by $X^2$.
If both $r,v$ are non-zero, then we get a basis of $R$
excluded in part (ii) of Lemma~\ref{lem:5.4}.
This shows that at least one of $v,r$ must be zero.
Assume $r = 0$.
Then $u s \neq 0$ and we may assume $u = 1 = s$.
The relations in part (i) of Lemma~\ref{lem:5.4}
become
\begin{align*}
  X^2 + v X Y &\in J_R^3 ,
 &
  Y^2 &\in J_R^3 .
\end{align*}
Suppose first that $m \geq 2$.
Then $R / J_R^3$ is $5$-dimensional,
and applying \cite[Lemma~III.6]{E4})
we conclude that $X^2 \in J_R^3$.
Since $X Y \notin J_R^3$, it follows
that in this case $v = 0$ as well.

Assume now that $m = 1$.
Then $J_R^3 = 0$,
$R$ is $4$-dimensional,
and $Y^2 = 0$.
In this case we may assume that
$v = 0$ or $v = 1$.
\end{proof}

We get the following information from the above lemma.
\begin{itemize}
 \item
  If $m \geq 2$, then
  $\alpha_1 - \alpha$
  and
  $\alpha_2 - \bar{\alpha}$
  belong to
  $e_1 J^2 e_2 = e_1 J^4 e_2$.
 \item
  If $m = 1$, then
  $\alpha_1 - \alpha - v \bar{\alpha}$
  and
  $\alpha_2 - \bar{\alpha}$
  belong to
  $e_1 J^2 e_2 = e_1 J^4 e_2$,
  with $v = 0$ or $v = 1$.
\end{itemize}

We keep the bound quiver presentation $A = K Q/I$ of $A$
for which the arrows
$\alpha$,
$\bar{\alpha} = \sigma$,
$\beta$,
$\gamma$, 
$\delta$,
$\varrho$,
$\alpha_1$,
$\alpha_2$
satisfy the equalities and relations established above.
We define
$f(\delta) = \alpha$
and
$f(\varrho) = \sigma =  \bar{\alpha}$.
Then we obtain the permutation $f$ of arrows of $Q$
having two orbits
\begin{align*}
  (\alpha \, \gamma \, \delta),
  &&
  (\bar{\alpha} \, \beta \, \varrho),
\end{align*}
so $(Q,f)$ is the Markov triangulation quiver.
We note that the associated permutation $g$
has only one orbit
\[
 (\alpha \, \beta \, \delta \, \sigma \, \gamma \, \varrho) .
\]
Since $A$ is a symmetric algebra,
we may choose a $K$-linear form
$\varphi : A \to K$ such that
$\varphi(x y) = \varphi(y x)$
for all elements $x,y \in A$,
and $\Ker \varphi$ does not contain
a non-zero one-sided ideal of $A$.
We note that $\varphi(z) \neq 0$
for any non-zero element $z$ in $\soc(A)$,
because $z$ generates a one-dimensional
right ideal of $A$.

We consider also the local tame symmetric algebras
$R_2 = e_2 A e_2$
and
$R_3 = e_3 A e_3$.
Then $R_2$ is generated by the elements
\[
 X_2 = \beta f(\gamma) \bar{\alpha}
 \qquad
 \qquad
 \quad
   \mbox{and}
 \qquad
 \qquad
 \quad
 Y_2 = \gamma f(\beta) {\alpha},
\]
and $R_3$ is generated by the elements
\[
 X_3 = f(\beta) {\alpha} \beta
 \qquad
 \qquad
 \quad
   \mbox{and}
 \qquad
 \qquad
 \quad
 Y_3 = f(\gamma) \bar{\alpha} \gamma.
\]
We note the following equalities in $A$
\begin{align*}
 X Y &= \alpha g(\alpha) g^2(\alpha) g^3(\alpha) g^4(\alpha) g^5(\alpha) ,
 &
 Y X &= \bar{\alpha} g(\bar{\alpha}) g^2(\bar{\alpha}) g^3(\bar{\alpha})
        g^4(\bar{\alpha}) g^5(\bar{\alpha}) ,
\\
 X_2 Y_2 &= \beta g(\beta) g^2(\beta) g^3(\beta) g^4(\beta) g^5(\beta) ,
 &
 Y_2 X_2 &= \bar{\beta} g(\bar{\beta}) g^2(\bar{\beta}) g^3(\bar{\beta})
        g^4(\bar{\beta}) g^5(\bar{\beta}) ,
\\
 X_3 Y_3 &= \varrho g(\varrho) g^2(\varrho) g^3(\varrho) g^4(\varrho) g^5(\varrho) ,
 &
 Y_3 X_3 &= \bar{\varrho} g(\bar{\varrho}) g^2(\bar{\varrho}) g^3(\bar{\varrho})
        g^4(\bar{\varrho}) g^5(\bar{\varrho}) .
\end{align*}

The following proposition proves 
Theorem~\ref{th:5.1} in the case $m = 1$.

\begin{proposition}
\label{prop:5.6}
Assume $m = 1$.
Then $A$ is isomorphic to the algebra\linebreak
$\Lambda(Q^M,f,m_{\bullet},c_{\bullet})$
with $m_{\bullet}$ given by $m = 1$
and $c_{\bullet}$ given by some $c \in K^*$.
\end{proposition}

\begin{proof}
We have $\soc(A) = J^6$ and $R = e_1 A e_1$ is commutative.
The elements occurring in the identities
(\ref{eq:1}),
(\ref{eq:2}),
(\ref{eq:3})
belong to $J^5$, and hence they are in $\soc_2(A)$.

\smallskip

(a)
We prove the required commutativity relations
for $\alpha \gamma$
and $\sigma \beta = \bar{\alpha} \beta$.
By identity (\ref{eq:1}), we have
$\alpha \gamma \in e_1 J^5 e_2 \cap \bar{\alpha} A = K Y \alpha \beta$
and
$\bar{\alpha} \beta \in e_1 J^5 e_2 \cap \alpha A = K Y \bar{\alpha} \gamma$.
There exist $c,d \in K$ such that
\begin{align*}
   \alpha \gamma
    &= c Y \alpha \beta
    = c \sigma \gamma \varrho \alpha \beta ,
  \\
   \sigma \beta
    &= d X \bar{\alpha} \gamma
    = d \alpha \beta \delta \sigma \gamma .
\end{align*}
Recall that $\alpha\gamma, \ba\beta$ are not both zero. 
By Corollary~\ref{cor:5.3} we have
$\alpha \gamma \delta = \sigma \beta \varrho$,
and hence
$c Y X = d X Y = d Y X$,
and then $c = d$ and it must be $\neq 0$.

\smallskip

(b)
We prove the required commutativity relations
for $\gamma \delta = \gamma \delta_1$
and $\beta \varrho = \beta \delta_2$.
By identity (\ref{eq:2}), we have
$\gamma \delta \in e_2 J^5 e_3 \cap A \varrho = K X_2 \gamma \varrho$
and $\beta \varrho  \in e_2 J^5 e_3 \cap A \delta = K Y_2 \beta \delta$.
Similarly as in (a), 
we get
\begin{align*}
   \gamma \delta
    &= c_2 X_2 \gamma \varrho
    = c_2  \beta \delta \sigma \gamma \varrho ,
  \\
   \beta \varrho
    &= d_2 Y_2 \beta \delta
    = d_2 \gamma \varrho \alpha \beta \delta  ,
\end{align*}
for some $c_2,d_2 \in K$, not both zero.
Using now part (a), we obtain
\begin{align*}
   c Y X
    &= c \sigma \gamma \varrho \alpha \beta \delta
      =  \alpha \gamma \delta
      = c_2  \alpha \beta \delta \sigma \gamma \varrho
      = c_2 X Y
      = c_2 Y X
     ,
   \\
   c X Y
    &= c \alpha \beta \delta \gamma \varrho \sigma
      = \sigma \beta \varrho
      = d_2 \sigma  Y_2 \beta \delta
      = d_2 Y X
      = d_2 X Y
     ,
\end{align*}
and hence $c_2 = c$ and $d_2 = c$ (and both are non-zero).

\smallskip

(c)
We claim that $\alpha_1 - \alpha \in J^4$.
Assume this is false.
Then, by
Lemma~\ref{lem:5.5}~(ii)
and (\ref{eq:4}),
we have $\alpha_1 - \alpha - \bar{\alpha} \in J^4$.
Using part (b) and (\ref{eq:2*}),
we may assume that
$q_{\gamma} = c \varrho \alpha \beta \delta$.
Then, using (\ref{eq:3}),
we obtain
$\delta \alpha_1 = q_{\gamma} \alpha_2
 = c \varrho \alpha \beta \delta \alpha_2$,
and this lies in $\soc_2(e_3 A)$.
But then
$\delta \alpha_1 + J^5 = \delta (\alpha + \bar{\alpha}) + J^5
 = c \varrho \alpha \beta \delta \bar{\alpha} + J^5$.
By part (a) we have
$\alpha \gamma = c \sigma \gamma \varrho \alpha \beta$,
and hence
$\delta \alpha \gamma = c Y_3 X_3$.
We get now
\begin{align*}
   c  X_3 Y_3 + J^6
    &= c \varrho \alpha \beta \delta \bar{\alpha} \gamma + J^6
      =  \delta (\alpha + \bar{\alpha}) \gamma + J^6
      =  (\delta \alpha \gamma + \delta \bar{\alpha} \gamma) + J^6
\\&
      =  (c Y_3 X_3 + Y_3) + J^6
     ,
\end{align*}
which is a contradiction.
Therefore, indeed $\alpha_1 - \alpha \in J^4$.

\smallskip

(d)
We prove the required commutativity relations
for $\delta \alpha$
and $\varrho \bar{\alpha} = \varrho \sigma$.
Using part (b), we may assume that
$q_{\beta} = c  \delta \sigma \gamma \varrho$.
Then, by (\ref{eq:3}), we have
$\varrho \alpha_2 = q_{\beta} \alpha_1
 = c \delta \sigma \gamma \varrho \alpha_1$.
Hence we have
$\varrho \alpha_2 \in \delta A \cap J^5$
and
$\varrho \alpha_2 + J^5 = \varrho \bar{\alpha} + J^5
 = \varrho \sigma + J^5$.
Similarly, we have
$\delta \alpha_1 = q_{\gamma} \alpha_2
 = c \varrho \alpha \beta \delta \alpha_2$,
and hence
$\delta \alpha_1 \in \varrho A \cap J^5$
and
$\delta \alpha_1 + J^5 = \delta \alpha + J^5$.
Therefore, we obtain that
\begin{align*}
   \delta \alpha
    &= c_3 X_3 \delta \bar{\alpha}
    = c_3 \varrho \alpha \beta \delta \sigma,
  \\
   \varrho \sigma
    &= d_3 Y_3 \varrho \alpha
    = d_3 \delta \sigma \gamma \varrho \alpha  ,
\end{align*}
for some $c_3,d_3 \in K^*$.
Recall from (b) that
$\gamma \delta = c X_2 \gamma \varrho$
and
$\beta \varrho = c Y_2 \beta \delta$.
Then we conclude that
\begin{align*}
   c X_2 Y_2
    &= \gamma \delta \alpha
      = c_3 \gamma X_3 \delta \bar{\alpha}
      = c_3 Y_2 X_2
      = c_3 X_2 Y_2
      ,
   \\
   c Y_2 X_2
    &= \beta \varrho \sigma
      = d_3 \beta Y_3 \varrho \alpha
      = d_3 X_2 Y_2
      = d_3 Y_2 X_2
     ,
\end{align*}
and hence $c_3 = c$ and $d_3 = c$.

\smallskip

(e)
The required zero relations  of length three
\begin{align*}
 \alpha\gamma\varrho &= 0, &
 \sigma\beta\delta &= 0, &
 \gamma\delta\sigma &= 0, &
 \beta\varrho\alpha &= 0, &
 \delta\alpha\beta &= 0, &
 \varrho\sigma\gamma &= 0
\end{align*}
follow easily.
For example, we have
\[
  \alpha\gamma\varrho
  = c (Y \alpha \beta) \varrho
  = c (Y \alpha) \beta \varrho
  = 0,
\]
since $\beta \varrho \in \soc_2 (A)$.
\end{proof}

From now we assume that $m \geq 2$.
We will prove
Theorem~\ref{th:5.1}
in few steps.
We first note the following equality.

\begin{corollary}
\label{cor:5.7}
We have
$\gamma\delta\alpha_1 = \beta\varrho\alpha_2$.
\end{corollary}

\begin{proof}
Using equalities (\ref{eq:2*}) and (\ref{eq:3}) we obtain
\[
  \gamma\delta\alpha_1
   = \gamma(\delta\alpha_1)
   = \gamma(q_{\gamma}\alpha_2)
   = (\gamma q_{\gamma})\alpha_2
   = (\beta\varrho)\alpha_2
   = \beta\varrho\alpha_2
   .
\]
\end{proof}

The following lemma fixes the required relations
starting at the vertex $2$.

\begin{lemma}
\label{lem:5.8}
The following statements hold.
\begin{enumerate}[(i)]
 \item
  $\beta f(\beta) = c(Y_2 X_2)^{m-1} Y_2 \beta \delta$
  and
  $\gamma f(\gamma) = c(X_2 Y_2)^{m-1} X_2 \gamma \varrho$,
  for some $c \in K^*$, and are non-zero elements
  of $\soc_2(e_2 A)$.
 \item
  $\beta f(\beta) g(f(\beta)) = 0$
  and
  $\gamma f(\gamma) g(f(\gamma)) = 0$.
\end{enumerate}
\end{lemma}

\begin{proof}
Recall that the local algebra $R = e_1 A e_1$
has socle generated by $(X Y)^m$, and
$(Y X)^m = (X Y)^m$.
Since $A$ is a symmetric algebra, rotating $(X Y)^m$,
we get that the socle of $R_2 = e_2 A e_2$
is generated by $(X_2 Y_2)^m$, and
$(Y_2 X_2)^m = (X_2 Y_2)^m$.
We will temporarily use a basis where $\alpha, \bar{\alpha}$
are replaced by $\alpha_1,\alpha_2$.
Let
$X'_2 = \beta\delta\alpha_2$
and
$Y'_2 = \gamma\varrho\alpha_1$.
Then
$(X'_2 Y'_2)^m = (X_2 Y_2)^m$
and hence it also generates the socle of $R_2$.
Therefore, we have a basis of $R_2$ consisting of initial
subwords of $(X'_2 Y'_2)^m$ and $(Y'_2 X'_2)^m$.
We may then get a basis of $e_2 A$ by taking initial
subwords of these.
Clearly, we have
$(Y'_2 X'_2)^m = (Y_2 X_2)^m$,
and hence
$(X'_2 Y'_2)^m = (Y'_2 X'_2)^m$.

\smallskip

(i)
From (\ref{eq:2}) and (\ref{eq:2*}), we have
$p \varrho = \gamma \delta = \beta q_{\beta}$,
and this is a linear combination of monomials
starting with $\beta$ and ending with $\varrho$, or possibly is zero.
Therefore, there are elements $c_i \in K$ such that
\[
  \gamma \delta
    = \sum_{i \geq 0} c_i (X'_2 Y'_2)^i X'_2 \gamma \varrho .
\]
Similarly, by (\ref{eq:2}) and (\ref{eq:2*}), we have
$\gamma q_{\gamma} = \beta \varrho = \bar{q} \delta$,
and this is a linear combination of paths
starting with $\gamma$ and ending with $\delta$.
Hence, there are elements $d_i \in K$ such that
\[
  \beta \varrho
    = \sum_{i \geq 0} d_i (Y'_2 X'_2)^i Y'_2 \beta \delta .
\]
The argument we used
for $\alpha\gamma$ and $\ba\beta$ applies here as
well and shows  that $\gamma\delta$ and $\beta\rho$ cannot be both equal
to zero. By Corollary~\ref{cor:5.7}
we get
\[
  \sum_{i \geq 0} c_i (X'_2 Y'_2)^{i+1}
   = \sum_{i \geq 0} d_i (Y'_2 X'_2)^{i+1}
   .
\]
These are expressions in the basis of $R_2$,
and we can equate coefficients.
We get  $c_i = d_i = 0$ for $i < m - 1$,
and $c_{m-1} = d_{m-1}$.
We set $c = c_{m-1}$.
Since one of $\gamma \delta$ or $\beta\rho$ is non-zero
we conclude that $c \neq 0$.
As well, we obtain that
$\gamma \delta$ and $\beta \varrho$
are in the second socle of $e_2 A$.
Therefore, we may replace
$\alpha_1,\alpha_2$
by
$\alpha, \bar{\alpha}$
in these expressions, and hence also
$X'_2,Y'_2$ by $X_2,Y_2$.
Hence we obtain the required equalities
\[
  \beta f(\beta)
    = \beta \varrho
    = c(Y_2 X_2)^{m-1} Y_2 \beta \delta
  \qquad
  \quad
  \mbox{and}
  \qquad
  \quad
  \gamma f(\gamma)
    = \gamma \delta
    = c(X_2 Y_2)^{m-1} X_2 \gamma \varrho .
\]

\smallskip

(ii)
We claim that
$\gamma f(\gamma) g(f(\gamma)) = 0$,
that is,
$\gamma \delta \bar{\alpha} = 0$.
By the previous facts and (\ref{eq:3}),
we have
\[
 \gamma \delta  \bar{\alpha}
    = c(X_2 Y_2)^{m-1} X_2 \gamma \varrho \alpha_2
    = c(X_2 Y_2)^{m-1} X_2 \gamma q_{\beta} \alpha_1
    .
\]
For the equality $\gamma \delta = \beta q_{\beta}$ in
(\ref{eq:2*}),
we may take
\[
  q_{\beta} = \delta \bar{\alpha} (Y_2 X_2)^{m-1} \gamma \varrho,
\]
and then
$\gamma q_{\beta}
 = \gamma \delta \bar{\alpha} (Y_2 X_2)^{m-1} \gamma \varrho = 0$,
because $\gamma \delta$ is in $\soc_2(e_2 A)$.
It follows that $\gamma \delta \bar{\alpha} = 0$.
Similarly, one shows that
$\beta f(\beta) g(f(\beta)) = \beta \varrho \alpha = 0$.
\end{proof}

The following lemma fixes the required relations
starting at the vertex $1$.

\begin{lemma}
\label{lem:5.9}
The following statements hold.
\begin{enumerate}[(i)]
 \item
  $\alpha f(\alpha) = c (Y X)^{m-1} Y \alpha \beta$
  and
  $\bar{\alpha} f(\bar{\alpha}) = c (X Y)^{m-1} X \bar{\alpha} \gamma$,
  with $c \in K^*$ fixed above, and are non-zero elements
  of $\soc_2(e_1 A)$.
 \item
  $\alpha f(\alpha) g(f(\alpha)) = 0$
  and
  $\bar{\alpha} f(\bar{\alpha}) g(f(\bar{\alpha})) = 0$.
\end{enumerate}
\end{lemma}

\begin{proof}
By Lemma~\ref{lem:5.8} we have  $\gamma \delta \alpha$ is a non-zero element of $\soc(A)$,
and so is every rotation of $\gamma \delta \alpha$,
therefore 
$\alpha \gamma \delta$
and
$\delta \alpha \gamma$
are non-zero elements of $\soc(A)$.
In particular  $\alpha \gamma \neq 0.$
By (\ref{eq:1}) we have also
$\alpha \gamma = \bar{\alpha} \bar{q}$.
Then using the basis of $R$ we conclude that
there are elements $a_i, b_i \in K$ such that
\[
  \alpha \gamma
    = \sum_{i \geq 0} a_i (Y X)^i Y \alpha \beta
      + \sum_{i \geq 1} b_i (Y X)^i \bar{\alpha} \gamma  .
\]
Observe also that
$\delta Y = \delta \bar{\alpha} \gamma \varrho = Y_3 \varrho$
and
$\varrho X = \varrho \alpha \beta \delta = X_3 \delta$.
Hence we obtain
\[
  \delta \alpha \gamma
    = \sum_{i \geq 0} a_i (Y_3 X_3)^{i+1}
      + \sum_{i \geq 1} b_i (Y_3 X_3)^{i} Y_3  .
\]
This is in terms of the basis of $R_3$, and
it is in the socle.
It follows that all $b_i = 0$ and that $a_i = 0$
for $i < m-1$.
We claim that $a_{m-1} = c$.
We postmultiply the above expression
for $\alpha \gamma$ with $\delta$ and obtain
\[
  \alpha \gamma \delta
    = a_{m-1} (Y X)^{m} .
\]
On the other hand, if we premultiply the known
identity for $\gamma \delta$ obtained in
Lemma~\ref{lem:5.8} with $\alpha$,
then we get
\[
  \alpha \gamma \delta
    = c (X Y)^{m}
    = c (Y X)^{m}
    \neq 0
    ,
\]
because
$\alpha X_2 = \alpha \beta \delta \bar{\alpha} = X \bar{\alpha}$
and
$\bar{\alpha} Y_2 = \bar{\alpha} \gamma\varrho \alpha = Y \alpha$.
Hence $a_{m-1} = c$.
Therefore, we obtain that
\[
 \alpha f(\alpha) = c (Y X)^{m-1} Y \alpha \beta .
\]
The proof of the equality
\[
 \bar{\alpha} f(\bar{\alpha}) = c (X Y)^{m-1} X \bar{\alpha} \gamma
\]
is similar.
The zero relations
$\alpha f(\alpha) g(f(\alpha)) = \alpha \gamma \varrho = 0$
and
$\bar{\alpha} f(\bar{\alpha}) g(f(\bar{\alpha}))
 = \bar{\alpha} \beta \delta = 0$
follow by arguments similar to those applied in the
proof of (ii) of
Lemma~\ref{lem:5.8}.
\end{proof}

\begin{lemma}
\label{lem:5.10}

The following statements hold.
\begin{enumerate}[(i)]
 \item
  $\delta f(\delta) = c (X_3 Y_3)^{m-1} X_3 f(\gamma) \bar{\alpha}$
  and
  $\bar{\delta} f(\bar{\delta}) = c (Y_3 X_3)^{m-1} Y_3 f(\beta) \alpha$,
  with $c \in K^*$ fixed above,
  and are non-zero elements
  of $\soc_2(e_3 A)$.
 \item
  $\delta f(\delta) g(f(\delta)) = 0$
  and
  $\bar{\delta} f(\bar{\delta}) g(f(\bar{\delta})) = 0$.
\end{enumerate}
\end{lemma}

\begin{proof}
(i)
Since $X_3 Y_3$ and $Y_3 X_3$ are cycles of length six
consisting of arrows of the whole $g$-orbit
$(\alpha \beta \delta \sigma \gamma \varrho)$,
using the symmetric $K$-linear form $\varphi : A \to K$
and the fact that $(X Y)^m = (Y X)^m$
generates the socle of $R$,
we conclude that $(X_3 Y_3)^m = (Y_3 X_3)^m$
generates the socle of $R_3$.
Then $R_3$ has a basis consisting of $e_3$ and
initial subwords of
$(X_3 Y_3)^m$ and $(Y_3 X_3)^m$.
It follows from
Lemma~\ref{lem:5.8}
that
\begin{align*}
  \beta f(\beta)
   &= c (Y_2 X_2)^{m-1} Y_2 \beta f(\gamma) ,
  &
  \gamma f(\gamma)
   &= c (X_2 Y_2)^{m-1} X_2 \gamma f(\beta) .
\end{align*}
Further, from the equalities
(\ref{eq:3}) we have
\begin{align*}
  f(\gamma) \alpha_1
   &= \delta_1 \alpha_1
    = q_{\gamma} \alpha_2 ,
  &
  f(\beta) \alpha_2
   &= \delta_2 \alpha_2
    = q_{\beta} \alpha_1 .
\end{align*}
Then we obtain
\begin{align*}
  \beta q_{\beta} \alpha_1
   &= \gamma f(\gamma) \alpha_1
    = c (X_2 Y_2)^{m-1} X_2 \gamma f(\beta) \alpha_1 ,
\\
  \gamma q_{\gamma} \alpha_2
   &= \beta f(\beta) \alpha_2
    = c (Y_2 X_2)^{m-1} Y_2 \beta f(\gamma) \alpha_2 ,
\end{align*}
which are elements of $\soc(e_2A)$.
Since
$\alpha_1 - \alpha$ and $\alpha_2 - \bar{\alpha}$
belong to $e_1 J^2 e_2 = e_1 J^5 e_2$,
we conclude that
\begin{align*}
  \beta q_{\beta} \alpha_1
   &= c (X_2 Y_2)^{m-1} X_2 \gamma f(\beta) \alpha
    = c (X_2 Y_2)^{m} ,
\\
  \gamma q_{\gamma} \alpha_2
   &= c (Y_2 X_2)^{m-1} Y_2 \beta f(\gamma) \bar{\alpha}
    = c (Y_2 X_2)^{m} ,
\end{align*}
Then we infer that
\begin{align*}
  q_{\beta} \alpha_1
   &= c f(\gamma) \bar{\alpha} Y_2 (X_2 Y_2)^{m-1}
    ,
  &
  q_{\gamma} \alpha_2
   &= c f(\beta) \alpha X_2 (Y_2 X_2)^{m-1}
    ,
\end{align*}
and are non-zero elements in $\soc_2(e_3 A)$.
Hence
\begin{align*}
  f(\beta) \alpha_2
   &= c f(\gamma) \bar{\alpha} Y_2 (X_2 Y_2)^{m-1}
    ,
  &
  f(\gamma) \alpha_1
   &= c f(\beta) \alpha X_2 (Y_2 X_2)^{m-1}
    .
\end{align*}
Since
$f(\beta) \alpha_2, f(\gamma) \alpha_1 \in \soc_2(e_3 A)$
and
$\alpha_1 - \alpha$, $\alpha_2 - \bar{\alpha}$
belong to $e_1 J^2 e_2 = e_1 J^5 e_2$,
we have
$f(\beta) \alpha_2 = f(\beta) \bar{\alpha}$
and
$f(\gamma) \alpha_1 = f(\gamma) \alpha$.
We note also that
$f(\gamma) \bar{\alpha} Y_2 = Y_3 f(\beta) \alpha$
and
$f(\beta) \alpha X_2 = X_3 f(\gamma) \bar{\alpha}$.
Therefore, we obtain the required equalities
\begin{align*}
  \delta f(\delta)
   &= f(\gamma) \alpha
    = c (X_3 Y_3)^{m-1} X_3 f(\gamma) \bar{\alpha}
    ,
  &
  \bar{\delta} f(\bar{\delta})
   &= f(\beta) \bar{\alpha}
    = c (Y_3 X_3)^{m-1} Y_3 f(\beta) \alpha
    ,
\end{align*}
which are non-zero elements of $\soc_2(e_3 A)$.

\smallskip

(ii)
We note that
$g(f(\delta)) = g(\alpha) = \beta$
and
$g(f(\bar{\delta})) = g(\bar{\alpha}) = \gamma$.
Then we obtain
\[
  \delta f(\delta) g\big(f(\delta)\big)
   = f(\gamma) \alpha \beta
    = c (X_3 Y_3)^{m-1} X_3 f(\gamma) \bar{\alpha} \beta
    = 0
\]
and
\[
  \bar{\delta} f(\bar{\delta}) g\big(f(\bar{\delta})\big)
   = f(\beta) \bar{\alpha} \gamma
    = c (Y_3 X_3)^{m-1} Y_3 f(\beta) \alpha \gamma
    = 0
    ,
\]
because
$f(\gamma) \bar{\alpha} \beta = f(\gamma) \alpha p$,
$f(\beta) \alpha \gamma = f(\beta) \bar{\alpha} \bar{q}\in e_3 J^6 e_3$.
\end{proof}

Summing up, we obtain that $A$ is isomorphic to
the weighted triangulation algebra
$\Lambda(Q^M,f,m_{\bullet},c_{\bullet})$,
with the weight function
$m_{\bullet} : \cO(g) \to \bN^*$
given by the integer $m \geq 2$ and
the parameter function
$c_{\bullet} : \cO(g) \to K^*$
given by the chosen scalar $c \in K^*$.

\section{Triangulation of the quiver}\label{sec:triangulation}

The aim of this section is to prove the following
theorem on the triangulation of the quiver
of a $2$-regular algebra of generalized
quaternion type, which completes the proof of the Triangulation Theorem.

\begin{theorem}
\label{th:6.1}
Let $A = K Q/I$ be a $2$-regular algebra of generalized quaternion
type whose quiver $Q$ is not the Markov quiver.
Then there is a permutation $f$  of arrows of $Q$
such that the following statements hold.
\begin{enumerate}[(i)]
 \item
  $(Q,f)$ is a triangulation quiver.
 \item
  For each arrow $\alpha$ of $Q$,
  $\alpha f(\alpha)$ occurs in a minimal
  relation of $I$.
\end{enumerate}
\end{theorem}

We divide the proof into three lemmas
and two propositions.
In order to simplify the notation we will
make the following convention for
minimal relations of type C.
Namely, if
$c_1 \alpha_1 \beta_1 + c_2 \alpha_2 \beta_2 \in J^3$
is a minimal relation of type C in $A$,
we will write simply
$\alpha_1 \beta_1 + \alpha_2 \beta_2 \in J^3$.
In the combinatorics below
the scalars $c_1$ and $c_2$ of such relation
do not play any role.

We fix a vertex $i$ of $Q$,
denote by
$\alpha, \bar{\alpha}$ the arrows in $Q$ starting at $i$,
and set $j = t(\alpha)$, $k = t(\bar{\alpha})$.
It follows from
Corollary~\ref{cor:4.6}
that there are two different arrows
$\beta$, $\gamma$ in $Q$ with
$s(\beta) = j$ and $s(\gamma) = k$
such that
$\alpha \beta + e_i J^3$,
$\bar{\alpha} \gamma + e_i J^3$
form a basis of $e_i J^2/e_i J^3$.
We define
$\beta_1 = \bar{\beta}$
and
$\gamma_1 = \bar{\gamma}$,
and $x = t(\gamma_1)$, $y = t(\beta_1)$.
It follows from
Corollary~\ref{cor:4.4}
and
Proposition~\ref{prop:4.8}
that we have the following possibilities
for the minimal relations starting at $i$:
\begin{enumerate}[(a)]
 \item
  $\alpha \beta_1 \in J^3$
  and
  $\bar{\alpha} \gamma_1 \in J^3$;
 \item
  $\alpha \beta_1 + \bar{\alpha} \gamma \in J^3$
  and
  $\bar{\alpha} \gamma_1 \in J^3$;
 \item
  $\alpha \beta + \bar{\alpha} \gamma_1 \in J^3$
  and
  $\alpha \beta_1 \in J^3$;
 \item
  $\alpha \beta_1 + \bar{\alpha} \gamma \in J^3$
  and
  $\alpha \beta + \bar{\alpha} \gamma_1 \in J^3$.
\end{enumerate}

Recall that  there are two different arrows
$\delta_1 (=\delta), \delta_2 (=\delta^*)$ ending at $i$,
and with $x = s(\delta_1)$, $y = s(\delta_2)$.
We first analyse the cases when $\alpha$ is a loop.

\begin{lemma}
\label{lem:6.2}
Assume that $\alpha$ is a loop with $\alpha^2 \in J^3$.
Then $Q$ contains a subquiver of the form
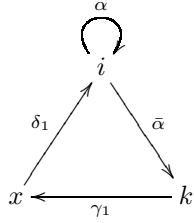
\[
  \xymatrix@R=3.pc@C=1.8pc{
    & i
     \ar[rd]^{\bar{\alpha}}
     \ar@(lu,ur)[]^{\alpha}
     \\
    x \ar[ru]^{\delta_1}
    && k \ar[ll]^{\gamma_1}
  }
\]
and
$\bar{\alpha} \gamma_1, \gamma_1 \delta_1 \in J^3$, and 
$\delta_1\ba$ is either in $J^3$ or part of a type $C$ relation. In both
cases we may define
\begin{align*}
 f(\alpha) &= \alpha, &
 f(\bar{\alpha}) &= \gamma_1, &
 f(\gamma_1) &= \delta_1, &
 f(\delta_1) &= \bar{\alpha}.
\end{align*}
\end{lemma}

\begin{proof}
We note that $i = j$ and $k \neq i$, because $Q$ is $2$-regular
with at least three vertices.
Moreover, $\alpha = \beta_1$ and $\bar{\alpha} = \beta$.
By Lemma~\ref{lem:4.7} there is no type C relation starting at $i$.
Hence $\bar{\alpha} \gamma_1 \in J^3$.
Then it follows from
Proposition~\ref{prop:4.8}
that $\gamma_1 \delta_1 \in J^3$.
Observe also that $x \neq k$, because $Q$ is $2$-regular
with at least three vertices.
Finally, by
Proposition~\ref{prop:4.3},
we have also a minimal relation from $x$ to $k$ invoking
a path of length $2$ starting from $\delta_1$,
and so either $\delta_1 \bar{\alpha} \in J^3$, or
there is a type $C$ relation from $x$ to $k$.
Therefore, we may define in both cases
$f(\alpha) = \alpha$,
$f(\bar{\alpha}) = \gamma_1$,
$f(\gamma_1) = \delta_1$,
$f(\delta_1) = \bar{\alpha}$.
\end{proof}

\begin{lemma}
\label{lem:6.3}
Assume
$\alpha$ is a loop with $\alpha^2 \notin J^3$.
Then $Q$ contains a subquiver
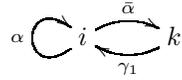
\[
   \xymatrix{
      i \ar@(dl,ul)[]^{\alpha} \ar@/^1.5ex/[r]^{\bar{\alpha}}
       & k \ar@/^1.5ex/[l]^{\gamma_1}
    }
\]
and
$\alpha \bar{\alpha}, \bar{\alpha} \gamma_1, \gamma_1 \alpha$
are in $J^3$.
Therefore, we may define
\begin{align*}
 f(\alpha) &= \bar{\alpha}, &
 f(\bar{\alpha}) &= \gamma_1, &
 f(\gamma_1) &= \alpha.
\end{align*}
\end{lemma}

\begin{proof}
It follows from the assumption that $j = i$,
$\alpha = \beta$ and $\bar{\alpha} = \beta_1$.
We know from Lemma~\ref{lem:4.7} that no relation of type C can 
start or end at vertex $i$, hence  $\alpha\bar{\alpha} \in J^3$ and 
$\bar{\alpha} \gamma_1 \in J^3$ and $k=y$. Then $k\neq i$ since otherwise we
would have a loop at $i$. Hence $x=t(\gamma_1)= i$ and we have
a type Z relation
$\gamma_1 \alpha \in J^3$.
Therefore, we may define
$f(\alpha) = \bar{\alpha}$,
$f(\bar{\alpha}) = \gamma_1$,
$f(\gamma_1) = \alpha$.
\end{proof}

From now we will assume that $\alpha$ and $\bar{\alpha}$
are not loops.

\begin{lemma}
\label{lem:6.4}
Assume that at most one minimal relation of type C
starts at any of the vertices $i,j,k,x,y$.
Then we may define
\begin{align*}
 f(\alpha) &= \beta_1, &
 f(\beta_1) &= \delta_2, &
 f(\delta_2) &= \alpha, &
 f(\bar{\alpha}) &= \gamma_1, &
 f(\gamma_1) &= \delta_1, &
 f(\delta_1) &= \bar{\alpha}.
\end{align*}
\end{lemma}

\begin{proof}
We consider two cases.
We assume first that $\alpha, \bar{\alpha}$
are not double arrows.

\smallskip

(A)
Assume that the minimal relations from $i$ are of type Z,
that is, we have
$\alpha \beta_1 \in J^3$ and $\bar{\alpha} \gamma_1 \in J^3$.
Then by
Proposition~\ref{prop:4.8},
we have minimal relations
$\beta_1 \delta_2 \in J^3$ and $\gamma_1 \delta_1 \in J^3$
of type Z.
We must therefore define
\begin{align*}
 f(\alpha) &= \beta_1, &
 f(\bar{\alpha}) &= \gamma_1, &
 f(\beta_1) &= \delta_2, &
 f(\gamma_1) &= \delta_1.
\end{align*}
[In the case when double arrows end at $i$ this may have involved
a choice for $\delta_1, \delta_2$.] It remains to show that
$f(\delta_2) = \alpha$ and $f(\delta_1) = \bar{\alpha}$.

\smallskip

(i)
Consider the minimal relations from vertex $j$,
one of them is $\beta_1 \delta_2 \in J^3$.
Let $\mu$ be the arrow starting at $t(\beta)$
such that the other minimal relation from $j$
is either
\[
 \beta \mu \in J^3
 \qquad
 \qquad
 \quad
   \mbox{or}
 \qquad
 \qquad
 \quad
 \beta \mu + \beta_1 \bar{\delta}_2 \in J^3
 .
\]
The arrows  ending at $j$ are $\alpha$ together with $\varrho$
starting at $t(\mu)$.
We note that $t(\mu) \neq i$, because there are no double
arrows from $i$.

\smallskip

(i.1)
Suppose the minimal relations from $j$ are
$\beta \mu \in J^3$ and $\beta_1 \delta_2 \in J^3$.
These imply minimal relations
$\mu \varrho \in J^3$ and $\delta_2 \alpha \in J^3$.
So we must define $f(\delta_2) = \alpha$ as required.

\smallskip

(i.2)
Suppose the minimal relations from $j$ are
\[
 \beta \mu + \beta_1 \bar{\delta}_2 \in J^3
 \qquad
 \qquad
 \quad
   \mbox{and}
 \qquad
 \qquad
 \quad
 \beta_1 \delta_2 \in J^3
 .
\]
Then it follows from
Proposition~\ref{prop:4.8}
that the minimal relations ending at $j$ are
\[
 \bar{\delta}_2 \varrho + \delta_2 \alpha \in J^3
 \qquad
 \qquad
 \quad
   \mbox{and}
 \qquad
 \qquad
 \quad
 \mu \varrho \in J^3
 .
\]
We have the relation $\mu \varrho \in J^3$ ending at $j$
and one relation of type $C$ from $y$ to $j$.
Then it follows from the assumption and the final part of
Lemma~\ref{lem:4.9}
(for the vertex $y$) that $\bar{\delta}_2 \bar{\varrho} \in J^3$.
This means that we must define $f(\delta_2) = \alpha$.

\smallskip

(ii)
Using the same arguments and the minimal relations starting from $k$
one shows that $f(\delta_1) = \bar{\alpha}$.

\smallskip

(B)
Assume that one of the minimal relations from $i$ is of type C.
Without loss of generality we may assume that
\[
 \alpha \beta_1 + \bar{\alpha} \gamma \in J^3
 \qquad
 \qquad
 \quad
   \mbox{and}
 \qquad
 \qquad
 \quad
 \bar{\alpha} \gamma_1 \in J^3
 .
\]
Then it follows from
Proposition~\ref{prop:4.8}
that we have minimal relations
\[
  \beta_1 \delta_2  \in J^3
 \qquad
 \qquad
 \quad
   \mbox{and}
 \qquad
 \qquad
 \quad
  \gamma \delta_2 + \gamma_1 \delta_1 \in J^3
 .
\]
We must therefore define
\begin{align*}
 f(\alpha) &= \beta_1, &
 f(\beta_1) &= \delta_2, &
 f(\bar{\alpha}) &= \gamma_1.
\end{align*}
It remains to show that
$f(\gamma_1) = \delta_1$,
$f(\delta_2) = \alpha$,
$f(\delta_1) = \bar{\alpha}$.

\smallskip

(1)
The minimal relations from $j$
are $\beta_1 \delta_2 \in J^3$
and either
\[
 \beta \mu \in J^3
 \qquad
 \qquad
 \quad
   \mbox{or}
 \qquad
 \qquad
 \quad
 \beta \mu + \beta_1 \bar{\delta}_2 \in J^3
 ,
\]
where $\mu$ is an arrow starting at $t(\beta)$.
The arrows  ending at $j$ are $\alpha$ and
an arrow $\varrho$ with $s(\varrho) = t(\mu)$.
Since $\alpha, \bar{\alpha}$ are not double arrows
we have $t(\mu) \neq i$.
Then it follows from
Proposition~\ref{prop:4.8}
that the minimal relations ending at $j$ are
either
\begin{align*}
 \tag{a}
 \label{eq:a}
 \delta_2 \alpha \in J^3
 \qquad
 \qquad
 \quad
   \mbox{and}
 \qquad
 \qquad
 \quad
 \mu \varrho \in J^3
 ,
\end{align*}
or
\begin{align*}
 \tag{b}
 \label{eq:b}
 \delta_2 \alpha + \bar{\delta}_2 \varrho \in J^3
 \qquad
 \qquad
 \quad
   \mbox{and}
 \qquad
 \qquad
 \quad
 \mu \varrho \in J^3
 .
\end{align*}
If (\ref{eq:a}) holds, we have to define
$f(\delta_2) = \alpha$.
Assume that (\ref{eq:b}) holds.
In this case, we have one minimal relation of type C
from $y$ to $j$.
We know that $\mu \varrho \in J^3$.
Then, by the final part of
Lemma~\ref{lem:4.9},
we obtain
$\bar{\delta}_2 \bar{\varrho} \in J^3$.
Therefore, we must define $f(\delta_2) = \alpha$.

\smallskip

(2)
The minimal relations from $k$ are the known
relation of type C and, by assumption, a
relation of type Z.
So we have
\[
  \gamma \bar{\delta}_2  \in J^3
 \qquad
 \qquad
 \quad
   \mbox{and}
 \qquad
 \qquad
 \quad
  \gamma \delta_2 + \gamma_1 \delta_1 \in J^3
 .
\]
Therefore, we must then define
$f(\gamma_1) = \delta_1$ as required.
As well the above relations imply
that $\delta_1 \bar{\alpha} \in J^3$,
using
Proposition~\ref{prop:4.8}.
Hence we must define $f(\delta_1) = \bar{\alpha}$.

\smallskip

Finally, if $\alpha, \bar{\alpha}$
are double arrows, then the arrows
starting from $j = k$
are not double arrows,
because the quiver $Q$ is not the Markov quiver.
We replace $i$ by $j$ and use the proof of the
previous case.
\end{proof}

\begin{proposition}
\label{prop:6.5}
Assume that two minimal relations of type C
start at the vertex $i$ but from any of the vertices
$j,k,x,y$ at most one minimal relation of type C starts.
Then the following statements hold.
\begin{enumerate}[(i)]
 \item
  One of $\beta \bar{\delta}_1$ and $\beta_1 {\delta}_2$
  is in $J^3$.
 \item
  We may assume $\beta \bar{\delta}_1 \in J^3$.
 \item
  There are arrows $\xi$ from $w = t(\bar{\delta}_1)$ to $j$
  and $\omega$ from $t = t(\bar{\delta}_2)$ to $k$,
  and the following elements belong to $J^3$
  \begin{align*}
   \beta \bar{\delta}_1,
   &&
   \delta_2 \alpha,
   &&
   \delta_1 \bar{\alpha},
   &&
   \gamma \bar{\delta}_2
   &&
   \omega \gamma,
   &&
   \xi \beta .
  \end{align*}
 \item
  The part of $f$ is defined as product of $3$-cycles
  \[
    (\beta_1 \, \delta_2 \, \alpha)
    (\gamma_1 \, \delta_1 \, \bar{\alpha})
    (\beta \, \bar{\delta}_1 \, \xi)
    (\gamma \, \bar{\delta}_2 \, \omega)
    .
  \]
\end{enumerate}
\end{proposition}

\begin{proof}
The two minimal relations of type C starting at $i$ are
\begin{align*}
 \tag{s.i}
 \label{eq:s.i}
 \alpha \beta_1 + \bar{\alpha} \gamma \in J^3
 \qquad
 \qquad
 \quad
   \mbox{and}
 \qquad
 \qquad
 \quad
 \alpha \beta + \bar{\alpha} \gamma_1 \in J^3
 .
\end{align*}
Then it follows from
Proposition~\ref{prop:4.8}(i)
that there are two minimal relations
of type C starting at $j$ and $k$ respectively
\begin{align*}
 \tag{s.j}
 \label{eq:s.j}
  \beta \delta_1 + \beta_1 \delta_2 \in J^3 ,
 \\
 \tag{s.k}
 \label{eq:s.k}
   \gamma \delta_2 +  \gamma_1 \delta_1 \in J^3
 .
\end{align*}

\smallskip

(1)
It follows from the assumption that the other minimal relation
starting at $j$ is of type Z.
We may assume that $\beta \bar{\delta}_1 \in J^3$,
after relabeling if necessary.
Then $\beta_1 \bar{\delta}_2$ is independent.
Moreover, there must be an arrow $\xi$ from
$w = t(\bar{\delta}_1)$ to $j$.
Then we must define
$f(\beta) = \bar{\delta}_1$
and
$f(\beta_1) = \delta_2$.
The minimal relations (\ref{eq:s.j})
and $\beta \bar{\delta}_1 \in J^3$ imply, by
Proposition~\ref{prop:4.8},
the minimal relations ending at $j$
\begin{align*}
 \tag{e.j}
 \label{eq:e.j}
 \delta_1 \alpha  + \bar{\delta}_1 \xi \in J^3
 \qquad
 \qquad
 \quad
   \mbox{and}
 \qquad
 \qquad
 \quad
 \delta_2 \alpha \in J^3
 .
\end{align*}
Hence we must set
$f(\delta_2) = \alpha$.

\smallskip

(2)
It follows from the assumption that the minimal relations
starting from $k$ are (\ref{eq:s.k})
and either
$\gamma_1 \bar{\delta}_1 \in J^3$
or
$\gamma \bar{\delta}_2 \in J^3$.
Since $\beta \bar{\delta}_1 \notin J^3$,
we conclude from
Lemma~\ref{lem:4.9}
that
$\gamma \bar{\delta}_2 \in J^3$,
and then
$\gamma_1 \bar{\delta}_1$
is independent.
Therefore, we must define $f(\gamma) = \bar{\delta}_2$.
Moreover, there is an arrow
$\omega$ from $t = t(\bar{\delta}_2)$ to $k$.
As well the minimal relations starting from $k$
show that we must define
$f(\gamma_1) = \delta_1$.
Furthermore, the minimal relations starting from $k$ imply, by
Proposition~\ref{prop:4.8},
the minimal relations ending at $k$
\begin{align*}
 \tag{e.k}
 \label{eq:e.k}
 \delta_2 \bar{\alpha} + \bar{\delta}_2 \omega \in J^3
 \qquad
 \qquad
 \quad
   \mbox{and}
 \qquad
 \qquad
 \quad
 \delta_1 \bar{\alpha} \in J^3
 .
\end{align*}
So we must define
$f(\delta_1) = \bar{\alpha}$.

\smallskip

(3)
The minimal relations starting at $y$ are
\begin{align*}
 \tag{s.y}
 \label{eq:s.y}
 \delta_2 \bar{\alpha} + \bar{\delta}_2 \omega \in J^3
 \qquad
 \qquad
 \quad
   \mbox{and}
 \qquad
 \qquad
 \quad
 \delta_2 \alpha \in J^3
 .
\end{align*}
From these, applying
Proposition~\ref{prop:4.8},
we deduce that $\omega \gamma \in J^3$.
Hence we must define
$f(\bar{\delta}_2) = \omega$
and
$f(\omega) = \gamma$.

\smallskip

(4)
Similarly, the minimal relations starting from $x$ are
\begin{align*}
 \tag{s.x}
 \label{eq:s.x}
 \delta_1 \alpha  + \bar{\delta}_1 \xi \in J^3
 \qquad
 \qquad
 \quad
   \mbox{and}
 \qquad
 \qquad
 \quad
 \delta_1 \bar{\alpha} \in J^3
 .
\end{align*}
Then, applying
Proposition~\ref{prop:4.8},
we conclude that $\xi \beta \in J^3$.
Hence we must define
$f(\bar{\delta}_1) = \xi$
and
$f(\xi) = \beta$.

\smallskip

Summing up, we have defined $f$ on the all relevant
arrows except for $f(\alpha)$ and $f(\bar{\alpha})$,
and for these we may chose
$f(\alpha) = \beta_1$ and $f(\bar{\alpha}) = \gamma_1$.
Then $f$ has the required part
\[
    (\beta_1 \, \delta_2 \, \alpha)
    (\gamma_1 \, \delta_1 \, \bar{\alpha})
    (\beta \, \bar{\delta}_1 \, \xi)
    (\gamma \, \bar{\delta}_2 \, \omega)
    .\vspace*{-3mm}
\]
\end{proof}

We may visualize the situation described in the above
proposition as follows
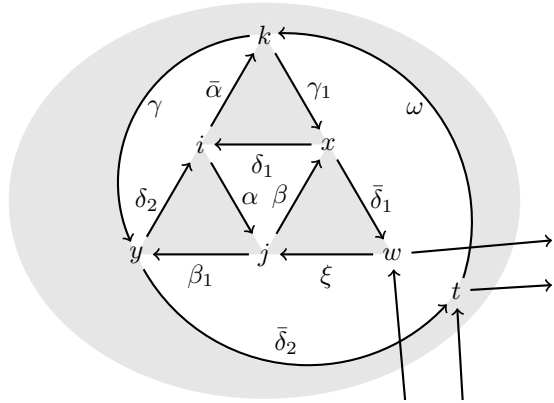
\[
\begin{tikzpicture}
[scale=.85]
\coordinate (1) at (0,1.72);
\coordinate (2) at (0,-1.72);
\coordinate (3) at (2,-1.72);
\coordinate (4) at (-1,0);
\coordinate (5) at (1,0);
\coordinate (6) at (-2,-1.72);
\coordinate (3p) at (3,-2.29);
\coordinate (3a) at (2.2,-4);
\coordinate (3b) at (4.5,-1.5);
\coordinate (3pa) at (3.1,-4);
\coordinate (3pb) at (4.5,-2.2);
\fill[fill=gray!20]
      (0,2.22cm) arc [start angle=90, delta angle=-360, x radius=4cm, y radius=3.1cm]
 --
  (0,1.72cm) arc [start angle=90, delta angle=120, radius=2.3cm]
    arc [start angle=203.4, delta angle=120, radius=2.9cm]
    arc [start angle=-23.4, delta angle=120, radius=2.9cm]
     -- cycle;
\fill[fill=gray!20]
    (1) -- (4) -- (5) -- cycle;
\fill[fill=gray!20]
    (2) -- (4) -- (6) -- cycle;
\fill[fill=gray!20]
    (2) -- (3) -- (5) -- cycle;

\node (1) at (0,1.72) {$k$};
\node (2) at (0,-1.72) {$j$};
\node (3) at (2,-1.72) {$w$};
\node (4) at (-1,0) {$i$};
\node (5) at (1,0) {$x$};
\node (6) at (-2,-1.72) {$y$};
\node (3p) at (3,-2.29) {$t$};
\draw[->,thick] (-.23,1.7) arc [start angle=96, delta angle=108, radius=2.3cm] node[midway,right] {$\gamma$};
\draw[->,thick] (-1.89,-1.953) arc [start angle=-150.6, delta angle=108, radius=2.95cm] node[midway,above] {$\bar{\delta}_2$};
\draw[->,thick] (3.11,-2.09) arc [start angle=-17.4, delta angle=109, radius=2.95cm] node[midway,left] {$\omega$};
\draw[->,thick]
 (1) edge node [right] {$\gamma_1$} (5)
(2) edge node [left] {$\beta$} (5)
(2) edge node [below] {$\beta_1$} (6)
(3) edge node [below] {$\xi$} (2)
 (4) edge node [left] {$\bar{\alpha}$} (1)
(4) edge node [right] {$\alpha$} (2)
(5) edge node [right] {$\bar{\delta}_1$} (3)
 (5) edge node [below] {$\delta_1$} (4)
(6) edge node [left] {$\delta_2$} (4)
;
\draw[->,thick] (3a) to (3);
\draw[->,thick] (3) to (3b);
\draw[->,thick] (3pa) to (3p);
\draw[->,thick] (3p) to (3pb);
\end{tikzpicture}
\]
where the shaded triangles denote $f$-orbits, and
possibly $w = t$.

\begin{proposition}
\label{prop:6.6}
Assume that the minimal relations starting from $i$ and $j$
are of type C.
Then the following statements hold.
\begin{enumerate}[(i)]
 \item
  The quiver $Q$ is the tetrahedral quiver
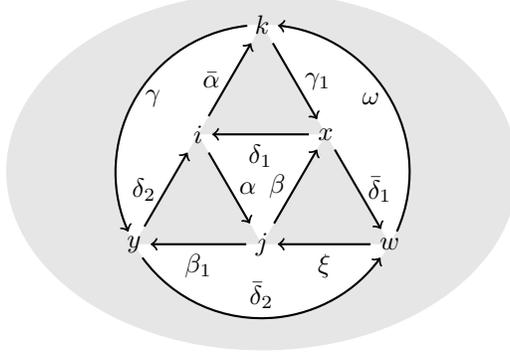
\[
\begin{tikzpicture}
[scale=.85]
\coordinate (1) at (0,1.72);
\coordinate (2) at (0,-1.72);
\coordinate (3) at (2,-1.72);
\coordinate (4) at (-1,0);
\coordinate (5) at (1,0);
\coordinate (6) at (-2,-1.72);
\fill[fill=gray!20]
      (0,2.22cm) arc [start angle=90, delta angle=-360, x radius=4cm, y radius=2.8cm]
 --  (0,1.72cm) arc [start angle=90, delta angle=360, radius=2.3cm]
     -- cycle;
\fill[fill=gray!20]
    (1) -- (4) -- (5) -- cycle;
\fill[fill=gray!20]
    (2) -- (4) -- (6) -- cycle;
\fill[fill=gray!20]
    (2) -- (3) -- (5) -- cycle;

\node (1) at (0,1.72) {$k$};
\node (2) at (0,-1.72) {$j$};
\node (3) at (2,-1.72) {$w$};
\node (4) at (-1,0) {$i$};
\node (5) at (1,0) {$x$};
\node (6) at (-2,-1.72) {$y$};
\draw[->,thick] (-.23,1.7) arc [start angle=96, delta angle=108, radius=2.3cm] node[midway,right] {$\gamma$};
\draw[->,thick] (-1.87,-1.93) arc [start angle=-144, delta angle=108, radius=2.3cm] node[midway,above] {$\bar{\delta}_2$};
\draw[->,thick] (2.11,-1.52) arc [start angle=-24, delta angle=108, radius=2.3cm] node[midway,left] {$\omega$};
\draw[->,thick]
 (1) edge node [right] {$\gamma_1$} (5)
(2) edge node [left] {$\beta$} (5)
(2) edge node [below] {$\beta_1$} (6)
(3) edge node [below] {$\xi$} (2)
 (4) edge node [left] {$\bar{\alpha}$} (1)
(4) edge node [right] {$\alpha$} (2)
(5) edge node [right] {$\bar{\delta}_1$} (3)
 (5) edge node [below] {$\delta_1$} (4)
(6) edge node [left] {$\delta_2$} (4)
;
\end{tikzpicture}
\]
  and, up to labelling of arrows, the shaded triangles denote $f$-orbits.

 \item
  If the minimal relations starting from $k$
  are of type C, then all minimal relations
  are of type C.
 \item
  If there is a  minimal relation
   of type Z starting at $k$,
  then we may assume:
  \begin{enumerate}[(a)]
   \item
    Precisely the paths
    $\gamma \bar{\delta}_2$,
    $\delta_1 \bar{\alpha}$,
    $\xi \beta$
    are in $J^3$.
   \item
    Precisely the paths
    $\bar{\delta}_1 \omega$,
    $\omega \gamma_1$,
    $\gamma_1 \bar{\delta}_1$
    are independent.
   \item
    The remaining paths of length $2$ occur in minimal relations
    of type C.
  \end{enumerate}
 \end{enumerate}
\end{proposition}

\setcounter{equation}{0}

\begin{proof}
It follows from the assumption that we have
two minimal relations of type C starting from $i$
\begin{align}
 \label{eq:6.1}
 \alpha \beta_1 + \bar{\alpha} \gamma \in J^3
 \qquad
 \qquad
 \quad
   \mbox{and}
 \qquad
 \qquad
 \quad
 \alpha \beta + \bar{\alpha} \gamma_1 \in J^3
 ,
\end{align}
and two minimal relations of type C starting from $j$
\begin{align}
 \label{eq:6.2}
 \beta \delta_1 + \beta_1 \delta_2 \in J^3
 \qquad
 \qquad
 \quad
   \mbox{and}
 \qquad
 \qquad
 \quad
 \beta \bar{\delta}_1 + \beta_1 \bar{\delta}_2 \in J^3
 .
\end{align}
In particular, $\bar{\delta}_1$ and $\bar{\delta}_2$
end at the same vertex $w$.
Further, it follows from
Proposition~\ref{prop:4.8}
that there are an arrow $\xi$ from $w$ to $j$
and two minimal relations of type C ending at $j$
\begin{align}
 \label{eq:6.3}
 \delta_1 \alpha + \bar{\delta}_1 \xi \in J^3
 \qquad
 \qquad
 \quad
   \mbox{and}
 \qquad
 \qquad
 \quad
 \delta_2 \alpha + \bar{\delta}_2 \xi \in J^3
 .
\end{align}

We consider now the minimal relations starting from $k$.
It follows from (\ref{eq:6.1}) and
Proposition~\ref{prop:4.8}
that we have a minimal relation of type C starting from $k$
and ending at $i$
\[
  \gamma \delta_2 + \gamma_1 \delta_1 \in J^3 .
\]
The other two paths of length two starting at $k$ are
$\gamma \bar{\delta}_2$ and $\gamma_1 \bar{\delta}_1$,
and both end at $w$.
Either they give a relation of type C,
or one of them is in $J^3$.
In both cases, there is an arrow $\omega$ from $w$ to $k$.
We note that the subquiver of $Q$ given by the vertices
$i,j,k,x,y,w$
is $2$-regular, and hence it must be all of $Q$.
We see that $Q$ is the tetrahedral quiver presented in (i).
We will show below that we may define permutation $f$
which is the product of four $3$-cycles given by the shaded
triangles.

\smallskip

(ii)
Assume that there are two minimal relations of type C
starting from $k$.
Then, applying
Proposition~\ref{prop:4.8},
repeatedly, we conclude that all minimal relations
are of type C. In this case, we may define 
$f$ as shown in the diagram, and get a consistent choice so that
it is a product of 3-cycles. 

\smallskip

(iii)
Assume now that there is no second minimal relation of type C
starting from $k$.
Then either
$\gamma \bar{\delta}_2 \in J^3$
or
$\gamma_1 \bar{\delta}_1 \in J^3$.
We may assume that
$\gamma \bar{\delta}_2 \in J^3$,
and then
$\gamma_1 \bar{\delta}_1$
is independent.
Applying
Proposition~\ref{prop:4.8} again,
we conclude that
\begin{align}
 \label{eq:6.4}
 \delta_2 \bar{\alpha} + \bar{\delta}_2 \omega \in J^3
 \qquad
 \qquad
 \quad
   \mbox{and}
 \qquad
 \qquad
 \quad
 \delta_1 \bar{\alpha} \in J^3
 .
\end{align}
In particular, it follows from
(\ref{eq:6.3}),
(\ref{eq:6.4})
and
Proposition~\ref{prop:4.8}
that we have
two minimal relations of type C ending at $y$
\begin{align}
 \label{eq:6.5}
 \alpha \beta_1 + \bar{\alpha} \gamma \in J^3
 \qquad
 \qquad
 \quad
   \mbox{and}
 \qquad
 \qquad
 \quad
 \xi \beta_1 + \omega \gamma \in J^3
 .
\end{align}
Consider now the minimal relations
starting from the vertex $x$.
We have one of type C from
(\ref{eq:6.3}),
and we also know that
$\delta_1 \bar{\alpha} \in J^3$.
Hence, by
Corollary~\ref{cor:4.4},
the path $\bar{\delta}_1 \omega$
must be independent.
Then, applying
Proposition~\ref{prop:4.8},
we deduce that $\xi \beta \in J^3$.

Finally consider the minimal relations
starting from  $w$.
We have from
(\ref{eq:6.5})
a relation of type C,
and we know that
$\xi \beta \in J^3$.
So there cannot be another
minimal relation of type C
starting from  $w$,
and $\omega \gamma_1$
is independent.
This has proved all details for the statement (iii).

Part (a) shows that we must define
\begin{align*}
 f(\gamma) &= \bar{\delta}_2, &
 f(\delta_1) &= \bar{\alpha}, &
 f(\xi) &= \beta.
\end{align*}
Moreover, by part (b), we have
$f(\bar{\delta}_1)  \neq \omega$,
$f(\omega) \neq \gamma_1$
$f(\gamma_1) \neq \bar{\delta}_1$,
so we must define
\begin{align*}
 f(\bar{\delta}_1) &= \xi, &
 f(\omega) &= \gamma, &
 f(\gamma_1) &= \delta_1.
\end{align*}

There are no further constraints.
We want to define $f$ so that is a product of $3$-cycles
and this has now a unique solution, which is
\[
    f=
    (\alpha \, \beta_1 \, \delta_2)
    (\bar{\delta}_1 \, \xi \, \beta)
    (\omega \, \gamma \, \bar{\delta}_2)
    (\gamma_1 \, \delta_1 \, \bar{\alpha})
    .
\]
This completes the proof.
\end{proof}

Let $A = K Q/I$ be a $2$-regular algebra of generalized quaternion
type whose quiver $Q$ is not the Markov quiver.
We fixed in the proof of Theorem~\ref{th:6.1}
a permutation $f$ of arrows of $Q$ such that $(Q,f)$
is a triangulation quiver.
Then we have the associated permutation $g$ of arrows of $Q$
such that $g(\alpha) = \overbar{f(\alpha)}$
for any arrow $\alpha$ of $Q$.
Moreover, the following properties of the permutations
$f$ an $g$ were established:
\begin{itemize}
 \item
  For each arrow $\alpha$ of $Q$,
  $\alpha f(\alpha)$ occurs in a minimal
  relation of $I$.
 \item
  For each vertex $i$ of $Q$ and the
  arrows $\alpha, \bar{\alpha}$ of $Q$
  starting at $i$, the cosets
  $\alpha g(\alpha) + e_i J^3$
  and
  $\bar{\alpha} g(\bar{\alpha}) + e_i J^3$
  form a basis of $e_i J^2 / e_i J^3$.
\end{itemize}

Our next aim is to describe an explicit basis
for the indecomposable projective module $e_i A$
using the $g$-orbits of the arrows
$\alpha, \bar{\alpha}$ starting at $i$.
We start with some observations.

\begin{lemma}
\label{lem:6.7}
Let $\alpha$ be an arrow in $Q$ and $n_{\alpha} = | \cO(\alpha) |$.
Then $g^{n_{\alpha}-1}(\alpha) = f^2(\bar{\alpha})$.
\end{lemma}

\begin{proof}
Let $i = s(\alpha)$.
Then $g^{n_{\alpha}-1}(\alpha) = g^{-1}(\alpha)$
is a unique arrows in $Q$ ending at $i$ which is not
in the $f$-orbit of $\alpha$.
Then the equality $g^{n_{\alpha}-1}(\alpha) = f^2(\bar{\alpha})$
follows.
\end{proof}

\begin{lemma}
\label{lem:6.8}
Assume there is a minimal relation
$\alpha f(\alpha) + \bar{\alpha} \gamma \in J^3$
of type C starting at vertex $i$.
Then
\[
  \cO(\bar{\alpha}) = \big(\bar{\alpha} \, \gamma \, f^2(\alpha)\big).
\]
\end{lemma}

\begin{proof}
The arrows ending at vertex $y$ are
$f(\alpha)$
and
$g(\bar{\alpha}) = \gamma$.
The arrows starting at $y$ are
$f^2(\alpha)$
and
$g(f(\alpha))$.
Clearly, we have
$g(f^2(\alpha)) = \bar{\alpha}$.
Therefore, we must have
$f^2(\alpha) = g^2(\bar{\alpha})$.
By Lemma~\ref{lem:6.7},
we have
$f^2(\alpha) = g^{n_{\bar{\alpha}}-1}(\bar{\alpha})$.
Hence $n_{\bar{\alpha}} = 3$ and
$\cO(\bar{\alpha}) = (\bar{\alpha} \, \gamma \, f^2(\alpha))$.
\end{proof}

\begin{corollary}
\label{cor:6.9}
Let $\cO$ be a $g$-orbit of length $> 3$.
Then, for any arrow $\delta$ in $\cO$, we have
$\bar{\delta} f(\bar{\delta}) \in J^3$.
\end{corollary}

\begin{proposition}
\label{prop:6.10}
Let $i$ be a vertex of $Q$, and
$\alpha, \bar{\alpha}$ the arrows starting at $i$.
Assume that there is at most one minimal relation of type C
starting from $i$.
Then $e_iA$ is spanned by monomials along the $g$-cycles
of $\alpha$ and $\bar{\alpha}$.
In particular, we have $\dim_K e_i J^k / e_i J^{k+1} \leq 2$
for any $k \geq 1$.
\end{proposition}

\begin{proof}
(1)
Assume first that there is no relation of type C
starting from $i$.
Then
$\alpha f(\alpha)$
and
$\bar{\alpha} f(\bar{\alpha})$
belong to $J^3$,
and $e_i J^2 / e_i J^3$ has basis consisting of the cosets
$\alpha \beta + e_i J^3$
and
$\bar{\alpha} \gamma + e_i J^3$,
where
$\beta = g(\alpha)$
and
$\gamma = g(\bar{\alpha})$.
The module $e_i J^3$ is generated by the elements
$\alpha \beta g(\beta)$,
$\bar{\alpha} \gamma g(\gamma)$,
$\alpha \beta f(\beta)$,
$\bar{\alpha} \gamma f(\gamma)$.
We must show that the last two may be omitted.
Consider the element $\alpha \beta f(\beta)$.
If $\beta f(\beta) \in J^3$
then $\alpha \beta f(\beta) \in J^4$,
and so this element is not needed.
Otherwise, there is a minimal relation of type C of the form
\[
   \beta f(\beta) + c f(\alpha) \delta \in J^3,
\]
for some non-zero $c \in K$ and some arrow $\delta$.
Then we conclude that
$\alpha \beta f(\beta) + J^4 = - c \alpha f(\alpha) \delta + J^4$,
and hence we do not need $\alpha \beta f(\beta)$.
Similarly, we show that
$\bar{\alpha} \gamma f(\gamma)$
is not needed.
Inductively, repeating these arguments, we prove the claim
in this case.

\smallskip

(2)
Now assume there is a minimal relation of type C
starting at vertex $i$, so we have (say) minimal relations
\[
 \alpha f(\alpha) + a \bar{\alpha} \gamma \in J^3
 \qquad
 \qquad
 \quad
   \mbox{and}
 \qquad
 \qquad
 \quad
 \bar{\alpha} f(\bar{\alpha}) \in J^3
 .
\]
for some non-zero $a \in K$.
Then $e_i J^2$ is generated by
$\alpha \beta$ and $\bar{\alpha} \gamma$,
and $\beta = g(\alpha)$, $\gamma = g(\bar{\alpha})$.
Hence the module  $e_i J^3$ is generated by
the elements
$\alpha \beta g(\beta)$,
$\bar{\alpha} \gamma g(\gamma)$,
$\alpha \beta f(\beta)$,
$\bar{\alpha} \gamma f(\gamma)$.
We must show that the last two elements may be omitted.

\smallskip

(i)
If $\gamma f(\gamma) \in J^3$
then $\bar{\alpha} \gamma f(\gamma) \in J^4$,
and so this element is not needed.
Otherwise we have a minimal relation of type C of the form
\[
   \gamma f(\gamma) + b f(\bar{\alpha}) \delta \in J^3,
\]
for some non-zero $b \in K$ and some arrow $\delta$.
Then we conclude that
$\bar{\alpha} \gamma f(\gamma) + J^4 = - b \bar{\alpha} f(\bar{\alpha}) \delta + J^4$,
and hence we do not need $\bar{\alpha} \gamma f(\gamma)$.

\smallskip

(ii)
If $\beta f(\beta) \in J^3$
then $\alpha \beta f(\beta) \in J^4$,
and thus this element is not needed.
Otherwise, we have a minimal relation of type C of the form
\[
   \beta f(\beta) + c f(\alpha) \delta' \in J^3,
\]
for some non-zero $c \in K$ and some arrow $\delta'$.
Then we obtain the equalities
\[
   \alpha \beta f(\beta) + J^4
   = - c \alpha f(\alpha) \delta' + J^4
   = - c a \bar{\alpha} \gamma \delta' + J^4
   .
\]
If $\delta' = f(\gamma)$,
we do not need $\alpha \beta f(\beta)$, by (i).
Otherwise, $\delta' = g(\gamma)$,
and $\bar{\alpha} \gamma g(\gamma)$
is in our list of required generators of $e_i J^3$.
By continuing with the same arguments the claim follows.
\end{proof}

With the assumptions as in Proposition \ref{prop:6.10} we want to identify
the second socle of $e_iA$.
We introduce notation: write $\alpha_i = g^{i-1}(\alpha)$ and $\ba_i = g^{i-1}(\ba)$.
Then let $\eta_r:= \alpha_1\ldots \alpha_r$ and $\bar{\eta}_r= \ba\ldots \ba_r$, for $r\geq 1$ (allowing repetitions of arrows). We know there
are generators $\zeta, \rho$ of the second socle of $e_iA$
such that $\zeta = \zeta e_x$ and $\rho = \rho e_y$ and $\zeta g^{-1}(\alpha)$
and $\rho g^{-1}(\ba)$ are non-zero in the socle.


\begin{lemma} \label{lem:6.11}
 Assume  $s\geq 1$ is such that $e_iJ^s/e_iJ^{s+1}$ is 2-dimensional
and $\dim e_iJ^{s+k}/e_iJ^{s+k+1} \leq 1$ for $k\geq 1$. Then 
\begin{enumerate}[(i)]
 \item
 If $e_iJ^{s+1} = {\rm soc}(e_iA)$ then  $e_iJ^s$ is the second socle of $e_iA$.
 \item
 Otherwise, let $e_iJ^{s+t}= {\rm soc} e_iA = \langle \bar{\eta}_{s+t} \rangle$ (say) and $t > 1$. Then the second socle of $e_iA$ is generated
by $\eta_s$ ending at $y$ and $\bar{\eta}_{s+t-1}$ ending at $x$ (unless
possibly $x=y$ and we have double arrows ending at $i$).
 \item
 If $\cO(\alpha) = \cO(\ba)$ then we must have \emph{(i)}.
\end{enumerate}
\end{lemma}

Note that with the notation of Section~\ref{sec:wsalg}, $\eta_{s+1} = B_{\alpha}$ and
$\bar{\eta}_{s+t} = B_{\bar{\alpha}}$, and similarly we can identify 
$A_{\alpha}$ and $A_{\bar{\alpha}}$.


\begin{proof} 
(i) \ By Proposition~\ref{prop:6.10}, $e_iJ^s/e_iJ^{s+1}$ has a basis 
consisting of the cosets of $\eta_s, \bar{\eta}_s$, and
$e_iJ^{s+1}/e_iJ^{s+2}$ is spanned by (say) the coset of $\bar{\eta}_{s+1}$, 
suppose it ends at vertex $x$.
If $e_iJ^{s+1}$ is the socle then clearly $e_iJ^s$ is the second socle, 
this must hold since $A$ is symmetric and since we have
two arrows ending at $i$.
Part (i) follows.

\medskip

(ii) \ 
Assume now that
the socle of $e_iA$ is $e_iJ^{s+t}$ for $t>1$, then with the choice above 
it is spanned by $\bar{\eta}_{s+t}$. Then the second socle
contains $\bar{\eta}_{s+t-1},$ but there must be another element $\zeta$ 
where $\zeta g^{-1}(\alpha)$ is non-zero
in the socle.
As well we know that
$\bar{\eta}_{s+v}$ for $0\leq v\leq t-2$ are not in the second socle. 
Then the required element of the second socle must be of the form
$$\zeta = d\eta_s +  \sum_v b_v\bar{\eta}_{s+v} \ \in e_iJ^se_x \setminus e_iJ^{s+1}e_x
$$
for $d, b_v\in K$ and $d\neq 0$.

\medskip

{\sc Case 1}. The cycle $\cO(\ba)$ does not pass through vertex $x$, in particular $\cO(\alpha)\neq \cO(\ba)$ and there are no double arrows ending at $i$. 
Then it
follows directly that $\zeta = d\eta_s$ and hence $\eta_s$ is in the second socle (and it ends at $x$).

\medskip

{\sc Case 2}. Assume the cycle $\cO(\ba)$ passes through vertex $x$ but that there are no double arrows ending at $i$.
Say $\ba_{r}$ is the arrow in the cycle starting at $x$, this ends at a vertex $w\neq i$.
Consider $\zeta\ba_{r}$, this must lie in the socle of $e_iA$ but it is annihilated by $e_i$ and therefore it is zero.
Suppose a monomial $\bar{\eta}_u$ ending at $y$ occurs in the expression for $\zeta$. 
Then
$\bar{\eta}_u\ba_r$ is only zero if $\bar{\eta}_u$ is in the socle (since otherwise it is
a basis element).
It follows that $\zeta = d\eta_s$ modulo the socle and $\eta_s$ is in the second socle.

\medskip

(iii) \ Assume $\cO(\alpha) = \cO(\ba)$. We assume 
that 
$\eta_s\alpha_s = \lambda \bar{\eta}_{s+t}$ where $\lambda \in K.$
Then $\eta_{s+t}$ is a rotation of $\bar{\eta}_{s+t}$ and is also a non-zero
element in the socle. So $\eta_{s+t}\neq 0$ and $\eta_{s+t+1}=0$, and it follows that $t=1$.
\end{proof}


\begin{remark}
\label{rem:6.12}
We cannot have
$\alpha f(\alpha) = \alpha p + \bar{\alpha} q \in J^3$
with $\alpha p$, $\bar{\alpha} q$ not in $J^3$.
Namely, if this were the case, then  $p$ and $q$ would be 
arrows with $p: j\to y$ and $q: k\to y$ where $y=t(f(\alpha))$. 
If $k\neq j$ then $Q$ is not 2-regular, so $j=k$, but then 
$\alpha, \ba$ are double arrows and also $f(\alpha), p$ are
double arrows. 
But then, by
Lemma~\ref{lem:5.2},
$Q$ is the Markov quiver, which is excluded.
\end{remark}

\section{Algebras with almost tetrahedral subquiver}\label{sec:almosttetrahedral}

The aim of this section is to describe a local
presentation of algebras discussed in
Proposition \ref{prop:6.5} and \ref{prop:6.6}(iii).

\begin{lemma}
\label{lem:7.1}
Assume that the both minimal relations of $A$ starting at $i$
are of type C.
With the notation of the diagrams following Proposition~\ref{prop:6.5}
and in Proposition~\ref{prop:6.6} the following statements hold.
\begin{enumerate}[(i)]
 \item
We may assume\\
$$\alpha f(\alpha) = \ba \gamma \ \ \mbox{ and } \ 
\ba f(\ba) = \alpha\beta
$$
$$f(\alpha) f^2(\alpha) = \beta f^2(\ba) \ \ \mbox{ and }
\ f(\ba)f^2(\ba) = \gamma f^2(\alpha).
$$
Furthermore, $\alpha f(\alpha) f^2(\alpha) = \ba f(\ba)f^2(\ba)$.

\item
 The local algebra $e_iAe_i$ is of finite representation type, generated by
$X_i:= \alpha\beta f^2(\ba)$. 

\item 
  Assume the setting of Proposition~\ref{prop:6.5}, or of Proposition~\ref{prop:6.6}~(iii). 
  Then  we have
$\soc(e_iA) = e_iJ^3$, spanned by $X_i$. The module $e_iA$ is 6-dimensional
with basis $\{ e_i, \alpha, \ba, \alpha\beta, \ba\gamma, X_i\}$.

\item
We have $\dim e_iAe_u=\dim e_uAe_i$ for $u\in \{j, k, x, y\}$ and
$0 = e_iAe_v=e_vAe_i$ for any other $v$ and $i\neq v$.
\end{enumerate}
\end{lemma}

\begin{proof}
(i) We assume there are two type C relations from $i$. Consider
one of them, we can write
 $\alpha f(\alpha) + a \ba\gamma \in J^3$ with $0\neq a\in K$.
Therefore
$$\alpha f(\alpha)  + a\ba\gamma = \alpha p + \ba q \in e_iAe_y$$
and then we must have that both $p$ and $q$ are in $J^2$ (since only two
arrows end at $y$).
Replace $f(\alpha)$ by $f(\alpha) - p$ and $\gamma$ by $(-a)\gamma + q$, this
gives the first identity.

The other type $C$ relation from $i$ is 
$\ba f(\ba) + a \alpha\beta \in J^3$ with $0\neq a\in K$. Similarly as
above we may assume after adjusting $f(\ba)$ and $\beta$ that
$$\ba f(\ba)  = \alpha \beta.$$
Now we apply Proposition~\ref{prop:4.8}
with these relations, and obtain directly the other two identities.
The last statement follows from these.

\medskip

(ii) All cyclic paths of length three from $i$ are equal to $X_i$
in $A$, by part (i).
This implies that $e_iAe_i$ is of finite type.

\medskip

(iii) \ In the setting of Proposition~\ref{prop:6.5} (with $w\neq t$) we know that
several paths of length two are in $J^3$, in particular
$\beta f(\beta)$ and $ \gamma f(\gamma)$, and $f^2(\alpha)\alpha$ and
$f^2(\ba)\ba$.
Any other path of length three from $i$ ends at $w$ or at $t$.
Such paths ending at $w$
are $ \ba f(\ba) f(\beta) = \alpha\beta f(\beta) \in J^4$
and ending at $t$ are
$\alpha f(\alpha)f(\gamma) = \ba \gamma f(\gamma) \in J^4$.
Hence $e_iJ^3 = X_iA$. Moreover
$X_i\alpha = \ba \gamma f^2(\alpha)\alpha \in J^5$ and
$X_i\ba = \alpha\beta f^2(\ba)\ba \in J^5$, therefore
$e_iJ^4\subseteq e_iJ^5$ and then $e_iJ^4=0$.
Clearly $X_i\neq 0$ since $A$ is symmetric and not simple.
This proves that $\soc (e_iA) = e_iJ^3$ and spanned by $X_i$.
The rest of the lemma follows in this case. 

\bigskip

Now assume the setting of Proposition~\ref{prop:6.6}. Then we have a second type
C relation from $j$, and  by the argument in (i) we may assume that
$$
  \beta f(\beta) = f(\alpha) f(\gamma),
$$
after  adjusting $f(\beta)$ and
$f(\gamma)$ if necessary.

Assume (iii) of Proposition~\ref{prop:6.6} holds, then we know that
 $\gamma f(\gamma), f^2(\ba)\ba$ and $f^2(\beta)\beta$ are
 in $J^3$:
In this case, non-cyclic paths of length three from $i$
end at $w$, they are
$$\ba f(\ba) f(\beta) = \alpha \beta f(\beta) = 
\alpha f(\alpha) f(\gamma)  = \ba\gamma f(\gamma) \in J^4.
$$
So $e_iJ^3$ is generated by $X_i$. Moreover
$X_i\ba = \alpha\beta f^2(\ba) \ba \in J^5$.

As well we have $f^2(\alpha)\alpha + c f(\gamma)f^2(\beta) \in J^3$
is a type C relation for some $0\neq c\in K$
(see the proof of Proposition~\ref{prop:6.6}). 
Hence modulo $J^4$
$$X_i\alpha = \alpha f(\alpha) f^2(\alpha)\alpha \equiv (-c) \ba \gamma f(\gamma)
f^2(\beta) \in J^5$$
and
$e_iJ^4 = X_iJ\subseteq e_iJ^5 = 0$.  The
rest follows as in the previous case.
\end{proof}

In the proposition below, we use the notation from
Proposition~\ref{prop:6.5}
and the quiver below it. 

\begin{proposition}
\label{prop:7.2}
Assume that $A$ has two minimal relations of type C
starting at $i$ but from any of the vertices $j, k, x, y$
at most one minimal relation of type C starts.
Then the following statements hold.
\begin{enumerate}[(i)]
 \item
  We may assume that we have the commutatively relations of type C
  \begin{align*}
    \alpha f(\alpha) &= \bar{\alpha} \gamma,
  &
    \bar{\alpha} f(\bar{\alpha}) &= \alpha \beta,
  &
    f(\alpha) f^2(\alpha)
   &= \beta f^2(\bar{\alpha}),
\\
  f(\bar{\alpha}) f^2(\bar{\alpha})
  &= \gamma f^2(\alpha),
  &
f(\beta) f^2(\beta)
  &= f^2(\bar{\alpha}) \alpha, 
&
    f(\gamma) f^2(\gamma)
  &= f^2(\alpha) \bar{\alpha},
  \end{align*}
  the first two starting at $i$ and the remaining ones starting
  at the vertices $j,k,x,y$.
 \item
  We have  relations
  \begin{align*}
   \beta f(\beta)
    =  f(\alpha) q
    &\in J^3 ,
   &
    f^2(\alpha) \alpha
    = q f^2(\beta)
    =  f(\gamma)  p
    &\in J^3 ,
   \\
   \gamma f(\gamma)
    =  f(\bar{\alpha}) \bar{q}
    &\in J^3 ,
   &
    f^2(\bar{\alpha}) \bar{\alpha}
    = \bar{q} f^2(\gamma)
    =  f(\beta)  \bar{p}
    &\in J^3 ,
   \\
   f^2(\beta) \beta
    &\in J^3 ,
   &
    f^2(\gamma) \gamma
    &\in J^3 ,
  \end{align*}
  where the elements $p,q,\bar{p},\bar{q}$
  are along cycles of $g$.
 \item
  The module $e_i A$ is $6$-dimensional, with basis
  $\{ e_i, \alpha, \bar{\alpha}, \alpha \beta, \bar{\alpha} \gamma, \alpha \beta f^2(\bar{\alpha})\}$.
  In particular, we have
  $e_i A e_w = 0$,
  $e_i A e_t = 0$,
  $e_w A e_i = 0$,
  $e_t A e_i = 0$.
 \item
  The paths of length two in (i) and (ii)  belong to $\soc_2(A)\setminus \soc(A)$.
\item 
  The local algebras at vertex $u$ for $u\in \{j, k, x, y\}$ have finite 
representation type.
\end{enumerate}
\end{proposition}

\begin{proof}
Most of part (i) and part (iii) is proved in
Lemma~\ref{lem:7.1}.
As well, by Proposition~\ref{prop:6.5}, we know that the six paths in (ii)
belong to $J^3$. 

We will now determine the details for  (ii)
and the last two commutativity relations in (i).

(a)
We may assume that
\[
   \beta f(\beta) = f(\alpha) q \in J^3 ,
\]
where
$q \in J^3$ is along the cycle of $g$ containing $f(\alpha)$.
Indeed, $\beta f(\beta) \in J^3$ by assumption. We can 
write  $\beta f(\beta) = f(\alpha) q + \beta \rho$,
and then by Remark~\ref{rem:6.12} we
know $\beta \rho \in J^3$ and we may replace 
$f(\beta)$ by $f(\beta) - \rho$. From the shape of the quiver $q \in e_yAe_w
\subset J^3$. 
Applying now
Proposition~\ref{prop:4.8},
we get
\[
 f(\beta) f^2(\beta) = f^2(\bar{\alpha}) \alpha'
 \qquad
 \qquad
   \mbox{and}
 \qquad
 \qquad
 f^2(\alpha) \alpha' =  q f^2(\beta)
 .
\]
where $\alpha' = \alpha + \zeta$ for $\zeta \in e_iJ^4e_j$ and hence 
$\zeta =0$ by (iii) and $\alpha'=\alpha$. 

We may write $q f^2(\beta) = f(\gamma) p$ for $p \in J^3$
along the cycle of $g$ containing $f(\gamma)$.

\smallskip

(b)
By assumption, $\gamma f(\gamma) \in J^3$. 
As in part (a), we may assume, after possibly adjusting $f(\gamma)$, that
\[
   \gamma f(\gamma) = f(\bar{\alpha}) \bar{q} ,
\]
where
$\bar{q} \in J^3$ is along the cycle of $g$ containing $f(\bar{\alpha})$.
Applying
Proposition~\ref{prop:4.8}
again, we obtain
\[
 f(\gamma) f^2(\gamma) = f^2(\alpha) \bar{\alpha}'
 \qquad
 \qquad
   \mbox{and}
 \qquad
 \qquad
 f^2(\bar{\alpha}) \bar{\alpha}' =  \bar{q} f^2(\gamma)
 .
\]
where $\ba' = \ba + \zeta'$ with $\zeta' \in e_iJ^4e_k$ which is zero 
by (iii). 
We may also write $\bar{q} f^2(\gamma) = f(\beta) \bar{p}$ for $\bar{p} \in J^3$
along the cycle of $g$ containing $f(\beta)$.

\smallskip

 As already mentioned $f^2(\beta)\beta$ and $f^2(\gamma)\gamma$ are in 
$J^3$. We have now proved parts (i) and (ii) completely. 

\smallskip

(iv) \ Any cyclic path of length three passing through $i$ is
a rotation of the socle element $X_i$ and hence is non-zero in the
socle. Using the relations in part (i) we get that also any rotations of
$\beta f(\beta) f^2(\beta)$ and
$\gamma f(\gamma) f^2(\gamma)$ are non-zero in the socle.

\smallskip

(1) We claim that $\beta f(\beta)$ and $\gamma f(\gamma)$ are  in $\soc_2(A)
\setminus \soc(A)$.
The first statement follows since
 $f^2(\beta)\beta f(\beta)$ is non-zero in the socle and
$\alpha \beta f(\beta) \in e_iAe_w = 0$. The second part is similar.

\smallskip

(2) We claim that $f^2(\alpha)\alpha$ and $f^2(\ba)\ba$ are in $\soc_2(A)
\setminus \soc(A)$.
First, $f(\alpha)f^2(\alpha)\alpha$ is non-zero in the socle, and
$\gamma f^2(\alpha)\alpha = \gamma f(\gamma) p = 0$ by (1) and since
$p= e_tpe_j \in J^2$. 
Hence, $f^2({\alpha}) {\alpha}$ belongs to $\soc_2(A) \setminus \soc (A)$.
Similarly, 
$f^2(\bar{\alpha}) \bar{\alpha} \in \soc_2(A) \setminus \soc (A)$
follows.

\smallskip

(3) We claim that $f^2(\beta)\beta$ and $f^2(\gamma)\gamma$ are in
$\soc_2(A)\setminus \soc(A)$.
For the first part, $f^2(\beta)\beta f(\beta)$ is non-zero in the socle, and
$f^2(\beta)\beta f^2(\ba) \in e_wAe_i =0$. The other part is similar.

\smallskip

(4) The first four paths in (i) are in $\soc_2(A)\setminus \soc(A)$
 by part (iii).
Next,  $f(\beta)f^2(\beta) \beta$ is non-zero in the socle, and
$f(\beta)f^2(\beta) f(\alpha) = f^2(\ba)\alpha f(\alpha) \in \soc(A)\cap e_xAe_y = 0$.
Hence $f(\beta)f^2(\beta) \in \soc_2(A)\setminus \soc(A)$. 
Similarly, we have $f(\gamma)f^2(\gamma) \in \soc_2(A)\setminus \soc(A)$.
This completes the proof of 
(iv). 

\medskip

(v) Consider $e_xAe_x$, this has generators $X'=f^2(\ba)\alpha \beta$ and
$Y'$ which is the product over the arrows along $\cO(f(\beta))$. Now,
$X'$ is the rotation of a socle element and hence lies in the socle.
Therefore $e_xAe_x$ is generated by $Y'$ and is therefore of finite representation type.
Similarly, the local algebras at $y, j, k$ are of finite representation type.
\end{proof}

\section{Algebras with  tetrahedral quiver}\label{sec:tetrahedralquiver}

The aim of this section is to describe the algebras
of generalized quaternion type discussed in
Proposition~\ref{prop:6.6}.
We use the notation for tetrahedral quiver presented in
Proposition~\ref{prop:6.6}.
We have to deal with two types of algebras, the following determines
the first type.

\begin{theorem}
\label{th:8.1}
Let $A = K Q/I$ be a $2$-regular algebra of generalized quaternion
type whose quiver $Q$ is not the Markov quiver.
Assume that there is an arrow $\alpha$ of $Q$
such that two minimal relations of type C start at both
$i = s(\alpha)$ and $j = t(\alpha)$
but one of minimal relations starting from
$k = t(\bar{\alpha})$  is not  of type C.
Then $A$ is isomorphic to a weighted
triangulation algebra with  tetrahedral quiver.
\end{theorem}

\begin{proof}
(I)
We know from
Proposition~\ref{prop:6.6}
that $Q$ is the  tetrahedral quiver presented there.
We  summarize the information obtained in Proposition~\ref{prop:6.6}
and Lemma~\ref{lem:7.1}.
We  have the commutativity relations
\setcounter{equation}{0}
\begin{align}
 \label{eq:8.1}
  &&
  \alpha f(\alpha) &=
  \bar{\alpha}\gamma
  &&
 \mbox{(from $i$ to $y$)},
  &&
 \\
 \label{eq:8.2}
  &&
  \bar{\alpha} f(\bar{\alpha}) &=
  \alpha\beta
  &&
 \mbox{(from $i$ to $x$)},
  &&
\\
 \label{eq:8.3}
  &&
  f(\alpha) f^2(\alpha)
   &= \beta f^2(\bar{\alpha})
  &&
 \mbox{(from $j$ to $i$)},
  &&
\\
 \label{eq:8.4}
  &&
  f(\bar{\alpha}) f^2(\bar{\alpha})
   &= \gamma f^2(\alpha)
  &&
 \mbox{(from $k$ to $i$)}.
  &&
\end{align}
It follows from
Proposition~\ref{prop:6.6}
that there are relations of type C
\begin{align}
\label{eq:8.5}
 &&
   \beta f(\beta) + a_1 f(\alpha) f(\gamma) &\in J^3
  &&
   \mbox{(from $j$ to $w$)},
 &&
 \\
 \label{eq:8.6}
  &&
   f^2(\alpha) \alpha+ a_2 f(\gamma) f^2(\beta) &\in J^3
  &&
   \mbox{(from $y$ to $j$)},
  &&
 \\
 \label{eq:8.7}
  &&
   f^2(\alpha) \bar{\alpha} + a_3 f(\gamma) f^2(\gamma) &\in J^3
  &&
   \mbox{(from $y$ to $k$)},
 &&
 \\
\label{eq:8.8}
  &&
   f^2(\gamma) \gamma+ a_4 f^2(\beta) f(\alpha) &\in J^3
  &&
   \mbox{(from $w$ to $y$)},
  &&
 \\
 \label{eq:8.9}
  &&
   f(\beta) f^2(\beta) + a_5 f^2(\bar{\alpha}) \alpha &\in J^3
  &&
   \mbox{(from $x$ to $j$)},
  &&
\end{align}
for some non-zero elements $a_1,a_2,a_3,a_4,a_5 \in K$.
Moreover, we know that up to labelling,
$\gamma f(\gamma)$,
$f^2(\bar{\alpha}) \bar{\alpha}$ and
$f^2(\beta) \beta$
      are unique paths of length two belonging to $J^3$, and
$f(\beta) f^2(\gamma)$,
$f^2(\gamma) f(\bar{\alpha})$ and 
$f(\bar{\alpha}) f(\beta)$
are paths of length $2$ which do not occur in a minimal
relation.
We can apply
Lemma~\ref{lem:7.1}
to each of vertex $i, j, y$ and get that
each of the local algebras
$e_i A e_i$,
$e_j A e_j$,
$e_y A e_y$
have finite representation type, and moreover 
the projective module $e_iA$ is 6-dimensional.

In fact, also $e_jA$ and $e_yA$ are 6-dimensional. To prove this, it
suffices to have the identities (5) to (9) rather than precise commutativity.
It then also follows that $e_jAe_k=0=e_kAe_j$
and $e_yAe_x=0=e_xAe_y$. 

\smallskip

(III)
Next we prove that the relations
(5) to (9) 
may be taken as commutativity relations.

Relation (5) involves $\beta f(\beta)$ and $f(\alpha)f(\gamma)$, they
are both non-zero in the 1-dimensional space $e_jAe_w$. We may assume they
are equal (if necessary, we replace $f(\beta)$ by a non-zero scalar
multiple) and get 
\begin{align*}
 \tag{5}
 \label{eq:8.5p}
   \beta f(\beta) = f(\alpha) f(\gamma).
\end{align*}
By the same arguments, 
by 
rescaling $f(\gamma)$ if necessary 
we obtain the equalities
\begin{align*}
 \tag{6}
 \label{eq:8.6p}
   f^2(\alpha) \alpha = f(\gamma) f^2(\beta).
\end{align*}
Similarly, 
by 
rescaling $f^2(\gamma)$ 
we obtain 
\begin{align*}
 \tag{7}
 \label{eq:8.7p}
   f(\gamma) f^2(\gamma) = f^2(\alpha) \bar{\alpha}.
\end{align*}
Now all arrows are fixed, but 
we conclude similarly that
\begin{gather*}
 \tag{8}
 \label{eq:8.8p}
   f^2(\gamma) \gamma =  a f^2(\beta) f(\alpha)
 \\
  \tag{9}
 \label{eq:8.9p}
   f(\beta) f^2(\beta) = b f^2(\bar{\alpha}) \alpha,
\end{gather*}
where $0\neq a, b \in K$. 
We claim that $a=1$ and $b=1$. 
Indeed, from
(\ref{eq:8.1}),
(\ref{eq:8.6p}),
(\ref{eq:8.7p}),
(\ref{eq:8.8p})
we have the following equalities of non-zero elements of $\soc(e_j A)$
\[
  a f^2(\alpha) \alpha f(\alpha)
   =  a f(\gamma)  f^2(\beta)  f(\alpha)
   =  f(\gamma)  f^2(\gamma)  \gamma
   =  f^2(\alpha) \bar{\alpha} \gamma
   =  f^2(\alpha) \alpha f(\alpha)
  ,
\]
and hence $a = 1$.
Similarly, using
(\ref{eq:8.3}),
(\ref{eq:8.5p}),
(\ref{eq:8.6p}),
(\ref{eq:8.9p}),
we obtain the equalities of non-zero elements $\soc(e_y A)$
\[
  b \beta f^2(\bar{\alpha}) \alpha
   =  \beta f(\beta)  f^2(\beta)
   =  f(\alpha)  f(\gamma) f^2(\beta)
   =  f(\alpha)  f^2(\alpha) \alpha
   =  \beta f^2(\bar{\alpha}) \alpha
  ,
\]
and so $b = 1$.

\smallskip

(IV)
We will show now that the elements
$\gamma f(\gamma)$,
$f^2(\bar{\alpha}) \bar{\alpha}$,
$f^2(\beta) \beta$
are in $\soc_2(A)$.

\smallskip

(a)
Consider $\gamma f(\gamma)$.
Observe that
$f(\gamma) f^2(\gamma) \gamma \in e_y A e_y \cap J^3 = \soc(e_y A)$.
Then by rotation
$\gamma f(\gamma) f^2(\gamma) \in \soc(A)$.
Further, we have $\gamma f(\gamma) \xi \in e_k J^3 e_j = 0$.
Hence $\gamma f(\gamma) J \subseteq \soc (A)$,
and consequently
$\gamma f(\gamma) \in \soc_2 (A)$.

\smallskip

(b)
Consider $f^2(\bar{\alpha}) \bar{\alpha}$.
We note that
$\bar{\alpha} f(\bar{\alpha}) f^2(\bar{\alpha}) \in e_i A e_i \cap J^3 = \soc(e_i A)$.
Then by rotation
$f^2(\bar{\alpha}) \bar{\alpha} f(\bar{\alpha}) \in \soc(A)$.
Moreover, we have $f^2(\bar{\alpha}) \bar{\alpha}\gamma \in e_x J^3 e_y = 0$.
Hence $f^2(\bar{\alpha}) \bar{\alpha} J \subseteq \soc (A)$,
and then
$f^2(\bar{\alpha}) \bar{\alpha} \in \soc_2 (A)$.

\smallskip

(c)
Consider $f^2(\beta) \beta$.
We have
$\beta f(\beta) f^2(\beta) \in e_j A e_j \cap J^3 = \soc(e_j A)$,
and then by rotation
$f^2(\beta) \beta f(\beta) \in \soc(A)$.
Further, $f^2(\beta) \beta f^2(\bar{\alpha}) \in e_w J^3 e_i = 0$.
Hence $f^2(\beta) \beta J \subseteq \soc (A)$,
and consequently
$f^2(\beta) \beta \in \soc_2 (A)$.

\smallskip

(V)
It follows from the above considerations that for any arrow
$\sigma$ in $Q$ we have $\sigma f (\sigma) \in \soc_2(A)$,
$\sigma f(\sigma) f^2(\sigma)$ generates $\soc (e_{s(\sigma)} A)$,
and consequently we have $\sigma f(\sigma) g(f(\sigma)) = 0$, because
$t(g(f(\sigma))) \neq t(f^2(\sigma)) = s(\sigma)$.
Similarly, we have $g^{-1}(\sigma) \sigma f(\sigma) = 0$ for any
arrow $\sigma$ in $Q$.

\smallskip

(VI)
We determine  now the required presentations of the elements
$\gamma f(\gamma)$,
$f^2(\bar{\alpha}) \bar{\alpha}$,
$f^2(\beta) \beta$.

\smallskip

(a) 
First consider 
the module $e_k A$.
Since there is only one minimal relation of type C starting at $k$,
it follows from
Proposition~\ref{prop:6.10}
that $e_kA$ is spanned by monomials along the $g$-cycles
$(\gamma \, f^2(\alpha) \, \bar{\alpha})$
and
$(f(\bar{\alpha}) \, f(\beta) \, f^2(\gamma) )$
of the arrows $\gamma$ and $\bar{\gamma} = \gamma_1 = f(\bar{\alpha})$
starting at $k$.

Since $\ba \gamma f^2(\alpha)$ is non-zero in the socle, it follows that
 the rotation
$\gamma f^2(\alpha)\ba$ is non-zero in the socle. Therefore the local algebra
$e_kAe_k$ is generated by $X_k:= f(\ba)f(\beta)f^2(\gamma)$, and therefore the socle of
$e_kA$ is spanned by $X_k^m$ for some $m\geq 1$.  It follows that
$\soc_2(e_kA)$ is generated by $\gamma f^2(\alpha)$ and 
$X_k^{m-1}f(\ba)f(\beta)$. Now, $\gamma f(\gamma)$ is in the second socle
and in $e_kAe_w$, and therefore 
there exists a non-zero element $c_k \in K$
such that
\begin{align}
 \label{eq:8.10}
     \gamma f(\gamma) = c_k X_k^{m - 1} f(\ba) f(beta) .
\end{align}

\smallskip

(b)
Consider the module $e_x A$.
The same argument as in (a) shows that
the local algebra $e_xAe_x$ is generated by $X_x:= f(\beta) f^2(\gamma) f(\ba)$, 
which is a rotation of $X_k$. Therefore 
the socle of $e_xA$ is spanned by the rotation of $X_k^m$, which is
$X_x^m$. The element 
$f^2(\ba)\ba$ lies in $\soc_2(e_xA)e_k$. 
As in (a) we may identify generators for the second socle, and deduce
that  there exists a non-zero element $c_x \in K$
such that
\begin{align}
 \label{eq:8.11}
     f^2(\bar{\alpha}) \bar{\alpha} = c_x X_x^{m - 1} f(\beta) f^2(\gamma).
\end{align}

\smallskip

(c)
Similarly, the socle of $e_wA$ is 
spanned by the rotation $X_w^m$ of $X_k^m$ where
$X_w = f^2(\gamma)f(\ba)f(\beta)$. 
By considering the second socle of the module $e_w A$,
we see that
there exists a non-zero element $c_w \in K$
such that
\begin{align}
 \label{eq:8.12}
     f^2(\beta)  \beta = c_w X_w^{m - 1}f^2(\gamma)f(\ba). 
\end{align}

\bigskip

To finish, we must  show that
the three scalar factors are equal.
We have the following equalities of non-zero elements of
$\soc(e_k A)$,
$\soc(e_x A)$
and
$\soc(e_w A)$,
respectively
\begin{align*}
 \tag{10'}
 \label{eq:8.10p}
     \gamma f(\gamma) f^2(\gamma) &= c_k X_k^{m} ,
\\
 \tag{11'}
 \label{eq:8.11p}
     f^2(\bar{\alpha}) \bar{\alpha} f(\bar{\alpha}) &= c_x X_x^{m} ,
\\
 \tag{12'}
 \label{eq:8.12p}
     f^2(\beta) \beta f(\beta) &= c_w X_w^{m} .
\end{align*}
Moreover, we have
$\gamma_1 X_x = X_k \gamma_1$,
$\bar{\delta}_1 X_w = X_x \bar{\delta}_1$,
$\omega X_k = X_w \omega$.
We note that $m \geq 2$.
Finally, we claim that $c_k = c_x = c_w$.

Using the identities (7) and (4) we get
$$
 \gamma f(\gamma)f^2(\gamma) = \gamma f^2(\alpha) \ba = f(\ba)f^2(\ba)\ba ,
$$
which is a rotation of $f^2(\ba)\ba f(\ba)$. It follows that $c_k=c_x$. 
By using the identites (5) and (8) we have
$$f^2(\beta) \beta f(\beta) = f^2(\beta) f(\alpha)f(\gamma)
 = f^2(\gamma)\gamma f(\gamma),$$
which is a rotation of $\gamma f(\gamma)f^2(\gamma)$. It follows that
$c_w=c_k$.
Hence we obtain $c_k = c_x = c_w$, and we denote
this scalar by $c$.

Recall that  $g$ has four orbits on  arrows in $Q$, 
\begin{align*}
  \cO(\alpha) &= \big(\alpha \, \beta \, f^2(\bar{\alpha})\big) ,
 &
  \cO\big(f(\gamma)\big) &= \big(  f(\gamma) \, f^2(\beta) \, f(\alpha) \big) ,
\\
  \cO(\gamma) &= \big(\gamma \, f^2(\alpha)  \, \bar{\alpha}\big) ,
 &
  \cO\big(f(\bar{\alpha})\big) &= \big( f(\bar{\alpha}) \, f(\beta) \, f^2(\gamma) \big) .
\end{align*}
We define the weight function
$m_{\bullet} : \cO(g) \to \bN^*$
and the parameter function
$c_{\bullet} : \cO(g) \to K^*$
as follows
\begin{align*}
  m_{\cO(\alpha)} &= 1, &
  m_{\cO(\gamma)} &= 1, &
  m_{\cO(f(\gamma))} &= 1, &
  m_{\cO(f(\bar{\alpha}))} &= m, \\
  c_{\cO(\alpha)} &= 1, &
  c_{\cO(\gamma)} &= 1, &
  c_{\cO(f(\gamma))} &= 1, &
  c_{\cO(f(\bar{\alpha}))} &= c.
\end{align*}
Then it follows from the equalities established above
that $A$ is isomorphic to the weighted triangulation algebra
$\Lambda(Q,f,m_{\bullet},c_{\bullet})$.
\end{proof}

The following theorem deals with 
the other type of algebras with
tetrahedral quiver.

\begin{theorem}
\label{th:8.2}
Let $A = K Q/I$ be a $2$-regular algebra of generalized quaternion
type whose quiver $Q$ is not the Markov quiver.
Assume that there is an arrow $\alpha$ of $Q$
such that two minimal relations of type C start at each
of the vertices
$s(\alpha) = s(\bar{\alpha})$, $t(\alpha)$, $t(\bar{\alpha})$.
Then $A$ is isomorphic to a non-singular tetrahedral algebra
or a higher non-singular tetrahedral algebra.
\end{theorem}

\begin{proof}
It follows from
Proposition~\ref{prop:6.6}
that $Q$ is the tetrahedral triangulation quiver,
and that two minimal relations of type C
start at each vertex of $Q$.
Moreover, we may choose the labeling of arrows in $Q$
as follows
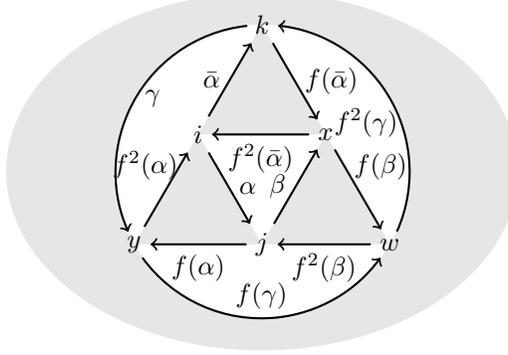
\[
\begin{tikzpicture}
[scale=.85]
\coordinate (1) at (0,1.72);
\coordinate (2) at (0,-1.72);
\coordinate (3) at (2,-1.72);
\coordinate (4) at (-1,0);
\coordinate (5) at (1,0);
\coordinate (6) at (-2,-1.72);
\fill[fill=gray!20]
      (0,2.22cm) arc [start angle=90, delta angle=-360, x radius=4cm, y radius=2.8cm]
 --  (0,1.72cm) arc [start angle=90, delta angle=360, radius=2.3cm]
     -- cycle;
\fill[fill=gray!20]
    (1) -- (4) -- (5) -- cycle;
\fill[fill=gray!20]
    (2) -- (4) -- (6) -- cycle;
\fill[fill=gray!20]
    (2) -- (3) -- (5) -- cycle;

\node (1) at (0,1.72) {$k$};
\node (2) at (0,-1.72) {$j$};
\node (3) at (2,-1.72) {$w$};
\node (4) at (-1,0) {$i$};
\node (5) at (1,0) {$x$};
\node (6) at (-2,-1.72) {$y$};
\draw[->,thick] (-.23,1.7) arc [start angle=96, delta angle=108, radius=2.3cm] node[midway,right] {$\gamma$};
\draw[->,thick] (-1.87,-1.93) arc [start angle=-144, delta angle=108, radius=2.3cm] node[midway,above] {$f(\gamma)$};
\draw[->,thick] (2.11,-1.52) arc [start angle=-24, delta angle=108, radius=2.3cm] node[midway,below left] {$f^2(\gamma)$\!\!\!\!};
\draw[->,thick]
 (1) edge node [right] {$f(\bar{\alpha})$} (5)
(2) edge node [left] {$\beta$} (5)
(2) edge node [below] {$f(\alpha)$} (6)
(3) edge node [below] {$f^2(\beta)$} (2)
 (4) edge node [left] {$\bar{\alpha}$} (1)
(4) edge node [right] {$\alpha$} (2)
(5) edge node [above right] {$\!\!\!f(\beta)$} (3)
 (5) edge node [below] {$f^2(\bar{\alpha})$} (4)
(6) edge node [above left] {$f^2(\alpha)\!\!\!\!\!$} (4)
;
\end{tikzpicture}
\]
where the shaded triangles denote $f$-orbits.
Then the cycles of the permutation $g$ are
\begin{align*}
  \big(\alpha \, \beta \, f^2(\bar{\alpha}) \big) ,
   &&
  \big(\bar{\alpha} \, \gamma \, f^2(\alpha) \big) ,
   &&
  \big(f(\alpha) \, f(\gamma) \, f^2(\beta) \big) ,
   &&
  \big(f(\bar{\alpha}) \, f(\beta) \, f^2(\gamma) \big) .
\end{align*}
Further, it follows from
Lemma~\ref{lem:7.1}
that for teach vertex $t$, the local algebra
$e_tAe_t$ is of 
finite representation type.
We will chose now representatives of arrows of $Q$ in $A$
such that almost all minimal relations of type $C$ can be taken
as commutativity relations.

We start with the vertex $i$ and apply
Lemma~\ref{lem:7.1}.
This gives the commutativity relations
\setcounter{equation}{0}
\begin{align}
 \label{eq:8.1n}
  &&
  \alpha f(\alpha) &=
  \bar{\alpha}\gamma
  &&
 \mbox{(from $i$ to $y$)},
  &&
 \\
 \label{eq:8.2n}
  &&
  \bar{\alpha} f(\bar{\alpha}) &=
  \alpha\beta
  &&
 \mbox{(from $i$ to $x$)},
  &&
\\
 \label{eq:8.3n}
  &&
  f(\alpha) f^2(\alpha)
   &= \beta f^2(\bar{\alpha})
  &&
 \mbox{(from $j$ to $i$)},
  &&
\\
 \label{eq:8.4n}
  &&
  f(\bar{\alpha}) f^2(\bar{\alpha})
   &= \gamma f^2(\alpha)
  &&
 \mbox{(from $k$ to $i$)}.
  &&
\end{align}
We consider now the second minimal relation starting at $j$.
Since $e_j A e_j$ and $e_w A e_w$ are of finite representation
type, applying arguments from the proof of
Lemma~\ref{lem:7.1},
we conclude that
$\beta f(\beta) = f(\alpha) f(\gamma) v$
for a unit $v$ of $e_w A e_w$.
Then, replacing $f(\beta)$ by $f(\beta) v^{-1}$,
we obtain the commutativity relation
\begin{align}
 \label{eq:8.5n}
  &&
   \beta f(\beta) = f(\alpha) f(\gamma) 
  &&
   \mbox{(from $j$ to $w$)}.
  &&
\end{align}
Using the relations (\ref{eq:8.3n}) and (\ref{eq:8.5n}),
we obtain by
Proposition~\ref{prop:4.3}
two commutativity relations ending at $j$
\[
   f^2(\alpha) \alpha' = f(\gamma) f^2(\beta)'
 \qquad
 \quad
   \mbox{and}
 \qquad
 \quad
   f^2(\bar{\alpha}) \alpha' = f(\beta) f^2(\beta)'
 ,
\]
with $\alpha = \alpha' a$ and $f^2 (\beta)' = f^2 (\beta) b$
for some units $a, b$ in $e_j A e_j$.
Then
$ f^2(\alpha) \alpha = f(\gamma) f^2(\beta) b a^{-1}$
and
$f^2(\bar{\alpha}) \alpha = f(\beta) f^2(\beta) b a^{-1}$.
Hence, replacing $f^2(\beta)$ by $f^2(\beta) b a^{-1}$,
we obtain the equalities
\begin{align}
 \label{eq:8.6n}
  &&
   f^2(\alpha) \alpha &= f(\gamma) f^2(\beta) 
  &
   \mbox{(from $y$ to $j$)},
  &&
 \\
  \tag{9}
 \label{eq:8.9n}
  &&
   f(\beta) f^2(\beta) &= f^2(\bar{\alpha}) \alpha 
  &
   \mbox{(from $x$ to $j$)}.
  &&
\end{align}

With these, all arrows except $f^2(\gamma)$ are fixed.
We consider the relations starting from $k$.
Using that the local algebra $e_w A e_w$
is of finite representation type,
say with  generator
$Z = f^2(\gamma) f(\bar{\alpha}) f(\beta) \in e_w J^3 e_w$,
we conclude that $\gamma f(\gamma)$ can be written as
\begin{align}
 \tag{10}
 \label{eq:8.10n}
     \gamma f(\gamma) = f(\bar{\alpha}) f(\beta) (c 1 + \zeta)
\end{align}
for $c \in K^*$,
$1$ the identity of $e_w A e_w$, and $\zeta \in e_w J^3 e_w$.
Using the minimal relations (\ref{eq:8.4n}) and (\ref{eq:8.10n})
starting from $k$,
we obtain by
Proposition~\ref{prop:4.3}
two minimal relations ending at $k$
\[
   f^2(\alpha)  \bar{\alpha}' = f(\gamma) f^2(\gamma)'
 \qquad
 \quad
   \mbox{and}
 \qquad
 \quad
   f^2(\bar{\alpha}) \bar{\alpha}' = f(\beta) (c 1 + \zeta) f^2(\gamma)'
 ,
\]
with
$\bar{\alpha}' = \bar{\alpha} u$
and
$f^2 (\gamma)' = f^2 (\gamma) v$
for some units $u, v$ in $e_k A e_k$.
Then, replacing $f^2(\gamma)$ by $f^2(\gamma) v u^{-1}$,
we obtain the equalities
\begin{align}
 \label{eq:8.7n}
  &&
   f(\gamma) f^2(\gamma) &= f^2(\alpha) \bar{\alpha} 
  &
   \mbox{(from $y$ to $k$)},
  &&
 \\
  \tag{11}
 \label{eq:8.11n}
  &&
   f^2(\bar{\alpha}) \bar{\alpha} &= f(\beta) (c 1 + \zeta) f^2(\gamma)
  &
   \mbox{(from $x$ to $k$)}.
  &&
\end{align}
Further, using the minimal relations (\ref{eq:8.6n}) and (\ref{eq:8.7n})
starting from $y$,
we obtain by
Proposition~\ref{prop:4.3}
two minimal relations ending at $y$
\begin{align*}\tag{*}\label{(*)}
   \bar{\alpha} \gamma = \alpha f(\alpha)'
 \qquad
 \qquad
 \quad
   \mbox{and}
 \qquad
 \qquad
 \quad
   f^2(\gamma) \gamma = f^2(\beta) f(\alpha)'
 ,
\end{align*}
with
$f(\alpha)' = f(\alpha) r$
for some unit $r$ in $e_y A e_y$.
We claim that $f(\alpha)' = f(\alpha)$.
Indeed, by (\ref{eq:8.1n}), we have
$\alpha f(\alpha) = \bar{\alpha}\gamma = \alpha f(\alpha)'$,
and hence $\alpha(f(\alpha) -  f(\alpha)') = 0$.
This means that $f(\alpha) -  f(\alpha)'$
belongs to the submodule $\Omega_A(\alpha A)$
of $e_j A$.
By Lemma~\ref{lem:8.3} below, 
we have
$$\Omega_A(\alpha A) \subseteq \soc_2 (e_j A) \subseteq e_j A e_j + e_j A e_i + e_j A e_w,
$$
but $f(\alpha) -  f(\alpha)' \in e_j A e_y$,
and therefore $f(\alpha) -  f(\alpha)' = 0$.
In particular, we obtain the commutativity relation
\begin{align}
 \tag{8}
 \label{eq:8.8n}
  &&
   f^2(\gamma) \gamma
   = f^2(\beta) f(\alpha)
   = f^2(\beta) g\big(f^2(\beta)\big)
  &&
   \mbox{(from $w$ to $y$)}.
  &&
\end{align}
Finally, using the minimal relations (\ref{eq:8.9n}) and (\ref{eq:8.11n})
starting from $x$,
we obtain by
Proposition~\ref{prop:4.3}
two minimal relations ending at $x$
\begin{align*}\tag{**}\label{(**)}
   \alpha \beta = \bar{\alpha} f(\bar{\alpha})'
 \qquad
 \quad
   \mbox{and}
 \qquad
 \quad
   f^2(\beta) \beta = (c 1 + \zeta) f^2(\gamma)f(\bar{\alpha})'
 ,
\end{align*}
with
$f(\bar{\alpha})' = f(\bar{\alpha}) s$
for some unit $s$ in $e_x A e_x$.
Observe that
$\bar{\alpha} f(\bar{\alpha}) = \alpha \beta = \bar{\alpha} f(\bar{\alpha})'$,
by (\ref{eq:8.2n}),
and hence
$\bar{\alpha} (f(\bar{\alpha}) - f(\bar{\alpha})') = 0$.
Then using Lemma~\ref{lem:8.3} again, we get
$f(\bar{\alpha}) = f(\bar{\alpha})'$.
In particular, we obtain the
new minimal relation
\begin{align}
 \tag{12}
 \label{eq:8.12n}
  &&
   f^2(\beta) \beta = (c 1 + \zeta) f^2(\gamma)f(\bar{\alpha})
  &&
   \mbox{(from $w$ to $x$)}.
  &&
\end{align}
Using the commutativity relations
(\ref{eq:8.1n})--(\ref{eq:8.9n}),
we may define the elements
\begin{align*}
  X_i
   &= \alpha f(\alpha) f^2(\alpha)
    = \bar{\alpha} f(\bar{\alpha}) f^2(\bar{\alpha}) ,
  &
  X_j
   &= f(\alpha) f^2(\alpha) \alpha
    = \beta f(\beta) f^2(\beta) ,
  \\
  X_y
   &= f^2(\alpha) \alpha f(\alpha)
    = f(\gamma) f^2(\gamma) \gamma ,
  &
  X_k
   &= f(\bar{\alpha}) f^2(\bar{\alpha}) \bar{\alpha}
    = \gamma f(\gamma) f^2(\gamma) ,
  \\
  X_x
   &= f^2(\bar{\alpha}) \bar{\alpha} f(\bar{\alpha})
    = f(\beta) f^2(\beta) \beta ,
  &
  X_w
   &= f^2(\beta) \beta f(\beta)
    = f^2(\gamma) \gamma f(\gamma) ,
\end{align*}
which generate the local algebras
$e_i A e_i$,
$e_j A e_j$,
$e_y A e_y$,
$e_k A e_k$,
$e_x A e_x$,
$e_w A e_w$,
respectively.
Let $X_i^m$ span the socle of $e_iA$. The generators of the other local
algebras are rotations of $X_i$, and since  $A$ is symmetric, it follows that
for every vertex $t$ of $Q$, the element $X_t^m$ spans the
socle of $e_tA$. 

We have to consider two cases.

\smallskip

(I)
Assume that $c \neq 1$.
We claim that
$X_i^2$ is zero (so $m=2$ and $X_t^2=0$ for each vertex $t$ of $Q$ and 
consequently
each  indecomposable projective module
$e_t A$ is $6$-dimensional.)
We first show  that
$\alpha f(\alpha) f(\gamma) = 0$.
Namely, using
(\ref{eq:8.1n}), (\ref{eq:8.10n}), (\ref{eq:8.2n}), and (\ref{eq:8.5n}),
we obtain that
$\alpha f(\alpha) f(\gamma) = \alpha f(\alpha) f(\gamma)  (c 1 + \zeta)$.
Hence
$\alpha f(\alpha) f(\gamma) ((c-1) 1 + \zeta) = 0$,
and then
$\alpha f(\alpha) f(\gamma) = 0$.
Then we obtain the equality
\begin{align*}
  X_i^2
   &= \alpha f(\alpha) f^2(\alpha) \alpha f(\alpha) f^2(\alpha)
    =  \alpha f(\alpha) f(\gamma) f^2(\beta) f(\alpha) f^2(\alpha)
    = 0.
\end{align*}
In particular, we conclude that
$e_i A e_w = 0$,
$e_j A e_k = 0$,
$e_y A e_x = 0$,
$e_k A e_j = 0$,
$e_x A e_y = 0$,
$e_w A e_i = 0$.
But then
$\theta  f(\theta) g(f(\theta)) = 0$ for any arrow $\theta$ of $Q$.
Observe also that $J^4 = 0$, and then
the relations
(\ref{eq:8.10n}),
(\ref{eq:8.11n}),
(\ref{eq:8.12n})
reduce to the relations
\begin{align}
 \tag{10'}
 \label{eq:8.10np}
   \gamma f(\gamma)
      &= c f(\bar{\alpha}) f(\beta), 
 \\
  \tag{11'}
 \label{eq:8.11np}
   f^2(\bar{\alpha}) \bar{\alpha}
     &= c f(\beta) f^2(\gamma),
\\
 \tag{12'}
 \label{eq:8.12np}
   f^2(\beta) \beta
     &= c f^2(\gamma) f(\bar{\alpha}).
\end{align}
Summing up, in the considered case,
the algebra $A$ is isomorphic to the non-singular tetrahedral
algebra $\Lambda(Q,f,m_{\bullet},c_{\bullet})$
with the weight function
$m_{\bullet} : \cO(g) \to \bN^*$
and the parameter function
$c_{\bullet} : \cO(g) \to K^*$
as follows
\begin{align*}
  m_{\cO(\alpha)} &= 1, &
  m_{\cO(\gamma)} &= 1, &
  m_{\cO(f(\gamma))} &= 1, &
  m_{\cO(f(\bar{\alpha}))} &= 1, \\
  c_{\cO(\alpha)} &= 1, &
  c_{\cO(\gamma)} &= 1, &
  c_{\cO(f(\gamma))} &= 1, &
  c_{\cO(f(\bar{\alpha}))} &= c.
\end{align*}


(II)
Assume that $c = 1$.
We claim first that $\zeta \neq 0$.
Suppose $\zeta = 0$.
Then we have the commutativity relations
(\ref{eq:8.10np}),
(\ref{eq:8.11np}),
(\ref{eq:8.12np}),
with $c = 1$.
Observe also that for any arrow $\theta$ in $Q$
we have
$X_{s(\theta)}^{m-1} \theta f(\theta) g(f(\theta)) = 0$.
Therefore, $A$ is isomorphic either to singular
tetrahedral algebra (if $m=1$) or to a singular
higher tetrahedral algebra (if $m \geq 2$),
which contradicts to
Theorems \ref{th:2.4} and \ref{th:tetr2}.
Hence, indeed $\zeta \neq 0$.

The element
$Z^m = (f^2(\gamma)  f(\bar{\alpha})   f(\beta))^m$
also generates the socle of $e_w A$.
If $m = 1$, then
\begin{align*}
  f^2(\bar{\alpha}) f(\beta) \zeta &= 0, &
  f(\beta) \zeta f^2(\bar{\alpha}) &= 0, &
  \zeta f^2(\bar{\alpha}) f(\beta) &= 0,
\end{align*}
and we are in the situation considered above.
Hence, we may assume that $m \geq 2$.
We will show that we may take $\zeta = \lambda Z^{m-1}$
for some $\lambda \in K^*$.
Applying
(\ref{eq:8.1n}),
(\ref{eq:8.10n}),
(\ref{eq:8.2n}),
(\ref{eq:8.5n}),
we obtain that
$\alpha f(\alpha) f(\gamma) = \alpha f(\alpha) f(\gamma) (1 + \zeta)$,
and hence $\alpha f(\alpha) f(\gamma) \zeta = 0$.
Now it follows from
(\ref{eq:8.2n})
and
(\ref{eq:8.5n})
that
\[
  \bar{\alpha} f(\bar{\alpha}) f(\beta) \zeta
  = \alpha \beta f(\beta) \zeta
  = \alpha f(\alpha) f(\gamma) \zeta
  = 0.
\]
Hence $f(\bar{\alpha}) f(\beta) \zeta$ belongs to the submodule
$\Omega_A(\bar{\alpha} A)$ of $e_k A$.
Since $\Omega_A(\bar{\alpha} A)$ is $2$-dimensional by Lemma~\ref{lem:8.3},
we conclude that $f(\bar{\alpha}) f(\beta) \zeta$ belongs to
$\soc_2 (e_k A) e_w$.
But then
$f(\bar{\alpha}) f(\beta) \zeta = \lambda f(\bar{\alpha}) f(\beta) Z^{m-1}$
for some $\lambda \in K^*$.
We have then $Z \zeta = f^2(\gamma) f(\bar{\alpha}) f(\beta) \zeta = \lambda Z^m$,
and consequently $\zeta = \lambda Z^{m-1}$,
as required.
Finally, if $\theta$ is any arrow,
we have
$X_{s(\theta)}^{m-1} \theta f(\theta) g(f(\theta)) = 0$.
This proves that $A$ is isomorphic to the non-singular
higher tetrahedral algebra
$\Lambda(m,\lambda)$.
\end{proof}

We present now the lemma which was used in the above proof.

\begin{lemma}\label{lem:8.3}
Assume $Q$ has tetrahedral quiver, such that the relations (1) to  (11) (except
for (8)) hold, and also \ref{(*)} and \ref{(**)}. 
Then for  $\sigma=\alpha$ 
or $\sigma = \ba$, the syzygy $\Omega_A(\sigma A)$ is contained in the second socle of $A$.
\end{lemma}

\begin{proof}  We identify 
$\Omega_A(\alpha A)$ with the set $\{ x\in e_jA: \alpha x =0\}$. One proves from the given relations that each projective has dimension $6m$. Now, $\Omega^{-1}(\alpha A)$ is obviously 2-dimensional
and $\alpha A$ is cyclic, it follows that $\Omega(\alpha A)$ also is 2-dimensional and then it must be contained in the second socle 
of $e_jA$. Similarly, the other part holds.
\end{proof}

\section{Local presentation}\label{sec:localpresentation}

Let $A = K Q / I$ be a $2$-regular algebra of generalized
quaternion type whose quiver is not the Markov quiver.
We fixed in Section~\ref{sec:triangulation} two permutations $f$ and $g$
of arrows in $Q$,  such that $(Q,f)$ is a triangulation quiver,
$g(\alpha) = \overbar{f(\alpha)}$ for any arrow $\alpha$ of $Q$,
and with the properties:
\begin{itemize}
 \item
  For each arrow $\alpha$ of $Q$,
  $\alpha f(\alpha)$ occurs in a minimal
  relation of $I$.
 \item
  For each vertex $i$ of $Q$
  and the arrows $\alpha$ and $\bar{\alpha}$
  starting at $i$,
  the cosets $\alpha g(\alpha) + e_i J^3$ and
  $\bar{\alpha} g(\bar{\alpha}) + e_i J^3$ form a basis
  of $e_i j^2 / e_i J^3$.
\end{itemize}

The main aim of this section is to prove the following theorem
on the local presentation of $A$.

\begin{theorem}
\label{th:9.1}
Let $i$ be a vertex of $Q$ such that there is at most
one minimal relation of type C starting from $i$,
and $\alpha$, $\bar{\alpha}$ be the arrows in $Q$
starting from $i$.
Then there are choices for
$\alpha, f(\alpha), f^2(\alpha)$
and
$\bar{\alpha}, f(\bar{\alpha}), f^2(\bar{\alpha})$
such that
$\alpha f(\alpha), f(\alpha) f^2(\alpha), f^2(\alpha) \alpha$
and
$\bar{\alpha} f(\bar{\alpha}), f(\bar{\alpha}) f^2(\bar{\alpha}),
 f^2(\bar{\alpha}) \bar{\alpha}$
belong to $\soc_2(A) \setminus \soc (A)$.
\end{theorem}

We divide the proof of the above theorem
into several steps.

Let $i$ be a vertex of $Q$ such that  at most
one minimal relation of type C
starts from $i$.
It follows from
Proposition~\ref{prop:6.10}
that $e_i A$ has a basis consisting of monomials
along cycles of $g$.
We work with such a basis throughout.
We say that an element $\xi \in e_i A$ has \emph{degree} $r$
if $\xi \in e_i J^r \setminus e_i J^{r+1}$.
If so then the expansion of $\xi$ in terms of the chosen
basis is of the form
\[
  \xi = c_1 \xi_1 + c_2 \xi_2 \ \  (\operatorname{modulo} \ \ J^{r+1}),
\]
where $\xi_1, \xi_2$ are monomials of length $r$
which belong to the basis, and $c_1, c_2 \in K$, not both zero.
We say that \emph{the lowest term} of $\xi$ is $c_1 \xi_1 + c_2 \xi_2$
if both $c_1, c_2$ are non-zero,
or $c_1 \xi_1$ if $c_2 = 0$,
or $c_2 \xi_2$ if $c_1 = 0$.

\medskip

The aim is to choose $\alpha, f(\alpha)$ such that the product
$\alpha f(\alpha) \in J^r$ with $r$ as large as possible, and  the
lowest term to keep control. 
We note the following invariance property.

\begin{lemma}
\label{lem:9.2}
Assume $c_1\xi_1 = c_1(\alpha_1 \alpha_2 \dots \alpha_r) \in e_i A$
is the lowest term of $\xi$, or part of the lowest term. 
Assume that $\alpha'_s = \alpha_s + \eta_s$ for some
$\eta_s \in J^2$ and $s \in \{1,\dots,r\}$, and $\xi'$
is obtained from $\xi$ by replacing $\alpha_s$ by $\alpha'_s$.
Then the lowest terms of $\xi'$, or part of the lowest term, 
 is obtained from $c_1\xi_1$ 
by replacing $\alpha_s$ by $\alpha_s'$. 
\end{lemma}

\begin{proof}
This is clear because the terms of $\eta_s$ do not
contribute to a lowest term.
\end{proof}

\begin{definition}
\label{def:9.3}
Given a pair of arrows $(\alpha,f(\alpha))$ of $Q$,
we say that it has \emph{index} $r$ if $r$ is maximal
such that $\alpha' f(\alpha)'$ has degree $r$ for some pair
of arrows $(\alpha', f(\alpha)')$ with
$\alpha' + J^2 = \alpha + J^2$ and
$f(\alpha)' + J^2 = f(\alpha) + J^2$.
\end{definition}

We  will see soon that
$\alpha f(\alpha)$ cannot be zero in our setting. 
By the above invariance property, if no such product
is zero,  we may modify $\alpha$
and $f(\alpha)$ and assume that $(\alpha, f(\alpha))$
has index $r$, and if so we call this pair of arrows maximal.
We note that this affects only $\alpha$ and $f(\alpha)$
and no other arrows.

\begin{lemma}
\label{lem:9.4}
Let $\alpha$ be an arrow in $Q$ such that 
$\alpha \neq f(\alpha)$.
If $\alpha f(\alpha)$ is non-zero then we may assume that
\[
    \alpha f(\alpha)
     = \bar{\alpha} q_{\bar{\alpha}} + \bar{\alpha} \eta_{\bar{\alpha}} ,
\]
with both summands expanded along the $g$-cycle of $\bar{\alpha}$,
and where the lowest term $\bar{\alpha} q_{\bar{\alpha}}$
is a scalar multiple of a monomial ending in $g^{-1}(f^2(\alpha))$.
\end{lemma}

In the following, we will
 refer to the expression for $\alpha f(\alpha)$ in this
lemma as its {\it lowest form presentation}.

\begin{proof}
We may assume that $(\alpha, f(\alpha))$ is a maximal pair,
say of index $r$.

{\sc Case 1}.  Assume first that
$\alpha f(\alpha) \in J^r \setminus J^{r+1}$ and $r \geq 3$.
We may write $\alpha f(\alpha) = \alpha p + \bar{\alpha} q$
with $p$ along the $g$-cycle of $\alpha$ and $q$ along the $g$-cycle
of $\bar{\alpha}$.
By Remark~\ref{rem:6.12}
we know that $p \in J^2$.
So we may assume, after replacing $f(\alpha)$ by $f(\alpha)'  = f(\alpha) - p$,
that
\[
    \alpha f(\alpha)
     = \bar{\alpha} q
     = \bar{\alpha} q_{\bar{\alpha}} + \bar{\alpha} \eta_{\bar{\alpha}} ,
\]
where $\bar{\alpha} q_{\bar{\alpha}}$ is the lowest term,
which is a monomial along the cycle of $g$ starting with $\bar{\alpha}$.
Then either $q_{\bar{\alpha}}$ ends in $g^{-1}(f^2(\alpha))$,
or it might end in $f(\alpha)$ as these are the arrows
ending in $t(f(\alpha))$.

We claim that $q_{\ba}$ ends in $g^{-1}(f^2(\alpha))$.
Note that here $\ba q_{\ba}$ is in $J^r\setminus J^{r+1}$ and $r\geq 3$.
As well by definition, it is a scalar multiple of a monomial 
of length $r$. 
If we had $\ba q_{\ba} = \ba q_1 f(\alpha)$
then  $q_1$ would be  in the radical, and we could
 replace $\alpha$ by  $\alpha' = \alpha - \ba q_1$. Then 
$\alpha'f(\alpha) = \ba \eta_{\ba} \in J^{r+1}$ which contradicts the choice of
$(\alpha, f(\alpha))$.
So the lowest term must end in $g^{-1}(f^2(\alpha))$.

\medskip

{\sc Case 2}. Assume we have a type $C$ relation
$$\alpha f(\alpha) + c\ba g(\ba) = \alpha p + \ba \bar{q} \in J^3
$$
with $0\neq c\in K$. 
Then $g(\ba)$ and $f(\alpha)$ and also $p, \bar{q}$ end at $y = t(f(\alpha))$.
We know from Lemma~\ref{lem:4.7} 
that a relation of type $C$ cannot involve double arrows, and
it follows that $p, \bar{q} \in J^2$. 
As in the first case
we can replace $f(\alpha)$ by $f(\alpha)-p$. We get
$$\alpha f(\alpha) = \ba q_{\ba} + \ba \eta_{\ba}
$$
where $q_{\ba} = cg(\ba)$ and this is already 
in the required form since $g(\ba)= 
g^{-1}(f^2(\alpha))$.
\end{proof}

We will now show that paths of length two cannot be zero in the algebra, 
in fact they cannot lie in the socle of $A$.

\begin{lemma} \label{lem:9.5}
Assume $\alpha \neq f(\alpha)$
and $\ba \neq f(\ba)$. Then the elements
$\alpha f(\alpha)$ and $\ba f(\ba)$ are both non-zero and not in the socle.
\end{lemma}

\begin{proof}  Note first that
if $\alpha f(\alpha)$ is in the socle then we may assume
it is zero.
Namely, we
can take the appropriate version of the socle monomial and factorize it,
as $\ba f(\ba) = \ba \omega_i'$.
Then we replace $f(\ba)$ ($\neq \ba$) and obtain the identity  $\ba f(\ba)=0$.


(1) Assume for a contradiction
that both $\alpha f(\alpha)$ and $\ba f(\ba)$ are zero. 
Then
the two generators of $\Omega^2(S_i)$ can be taken as 
$$(f(\alpha), 0), \ \ (0, f(\ba)).
$$
Hence $\Omega^2(S_i)$ is a direct sum, a contradiction.
So at least one of the two elements is non-zero,  and then it does
not lie  in the socle.


(2) Suppose  one of them is zero but not the other, then 
up to labelling, and using Lemma \ref{lem:9.4}, we have identities
$$ \alpha  f(\alpha)=  0  \ \ \mbox{ and} 
\ \ \ba f(\ba) =  \alpha q_{\alpha} + \alpha \eta_{\alpha}
$$
where the  expression for
$\ba f(\ba)$ is the lowest form presentation.
Using period four, we get from this that
$$f(\ba) f^2(\ba) = 0 \ \ \mbox{ and } \ 
f(\alpha)f^2(\alpha) = 
   q_{\alpha}f^2(\ba) + \eta_{\alpha}f^2(\ba).
$$
It follows that $\alpha f(\alpha)f^2(\alpha) = \ba f(\ba)f^2(\ba)$.
Now we have
$$0 = \ba f(\ba)f^2(\ba) = \alpha q_{\alpha}f^2(\ba) + \alpha \eta_{\alpha}f^2(\ba).$$
The first summand is a non-zero scalar multiple of a monomial along
the cycle of $g$ in the basis,  and is non-zero. 
Its degree is less than the degree 
of any term of $\alpha \eta_{\alpha}f^2(\ba)$. So the terms cannot cancel,
and we have a contradiction.

Therefore we must have that both $\alpha f(\alpha)$ and
$\ba f(\ba)$ are non-zero and not in the socle.
\end{proof}

Next we consider squares of loops fixed by $f$. 

\begin{lemma}\label{lem:9.6}
Assume $\alpha$ is a loop fixed by $f$. 
Then the following statemensts hold.
\begin{enumerate}[(i)]
 \item
 We may assume $\alpha^2 \in {\rm soc}_2(e_iA)$ and we can write
$$\alpha^2 = cA_{\ba} + dB_{\ba} \ \ (c, d\in K, c\neq 0).
$$
\item $c\neq 0$, that is $\alpha^2$ is not in the socle.
\item The product $\ba f(\ba)$ is not in the socle.
\end{enumerate}
\end{lemma}


\begin{proof} The quiver has at least three vertices, therefore
the vertices $i, k, x$ must be distinct (if $k=x$ then $f(\ba)$ would be
a loop and $|Q_0|=2$). Hence  
$g$ has cycle $(\alpha \ \ba \ g(\ba) \ldots \ f^2(\ba))$
of length $>3$.

We use the basis of $e_i A$ along the cycles of $g$.
The local algebra $R = e_i A e_i$ is generated by
\[
  X = \alpha
 \qquad
  \mbox{ and }
 \qquad
  Y =  \bar{\alpha} g(\bar{\alpha}) \dots f^2(\bar{\alpha}) .
\]
In particular, $R$ has a basis consisting of monomials in $X,Y$
which alternate between $X$ and $Y$, and the socle of $R$
is spanned by $(X Y)^m = (Y X)^m$ for some $m \geq 1$.
We set $m = m_{\alpha} = m_{\bar{\alpha}}$ and
$n = |\cO(\alpha)| = |\cO(\bar{\alpha})|$.

\medskip

(I) \ Assume first that $R$ has finite representation type.

In this case, $R$ is generated by $X = \alpha$, and $Y$
must be a polynomial in $X$ and $X Y = Y X$.
Then in the above basis we must have $m = 1$, and
$R$ is $4$-dimensional with basis $\{ e_i, X, Y, X Y \}$.
It must be equal to $K[X]$, which has therefore the basis
$\{ e_i, X, X^2, X^3 \}$.
Hence we have $Y = a X^2 + b X^3$ for some
$a,b \in K$, and then 
$a \neq 0$ because $X Y \neq 0$
and $X^3$ spans the socle of $R$.
Then we obtain
\[
   X^2 = a^{-1} Y - a^{-1} b X^3 .
\]
We also note that $Y X = a X^3 + b X^4  = a X^3$,
and hence
\[
   X^2 = a^{-1} Y - a^{-2} b Y X .
\]
Here $A_{\bar{\alpha}} = Y$ and  $Y X = B_{\bar{\alpha}}$.
We can also write $\alpha^2 = \ba q_{\ba} + \ba \eta_{\ba}$ where
\[
    \alpha^2 = c A_{\bar{\alpha}} + d B_{\bar{\alpha}} ,
\]
for $c = a^{-1}$ and $d = -a^2 b$.
Since $c \in K^*$, we conclude that
$\alpha^2  \in {\rm soc}_2(e_i A) \setminus \soc(e_i A)$.
We have proved (i) and (ii) in this case.
\medskip

(II) Now assume that $R$ is of infinite representation type, then it is a tame local
symmetric algebra.

(i) \ 
If $m=1$ then $R$ is 4-dimensional and it is isomorphic to one of the algebras
$(b)$ or $(b')$ presented
in \cite[Theorem~III.1]{E4}.
If $m\geq 2$ then,
by \cite[Theorem~III.1]{E4},
the algebra $R$ must have two generators $X'$ and $Y'$
whose squares are in the second socle of $R$.
Then $X'$ (say) is a loop at $i$, and we may take
$\alpha = X' = X$.

The second socle of $R$ is independent on the choice of
generators of $R$, and hence is spanned by
\begin{align*}
  &&
  (X Y)^{m-1} X,
  &&
  (Y X)^{m-1} Y,
  &&
 (X Y)^{m} = (Y X)^{m}.
  &&
\end{align*}
Hence we have
\[
  X^2 =
    u (X Y)^{m-1} X
    + v (Y X)^{m-1} Y
    + w (Y X)^{m},
\]
for some $u,v,w \in K$.
We may assume that $u=0$. Namely the element
\[
  z = e_i - u (Y X)^{m-2} Y ,
\]
is a unit of $R$, and
\[
  X^* = X z = X -  u X (Y X)^{m-2} Y = X - u (X Y)^{m-1}.
\]
Then we obtain
\begin{align*}
  (X^*)^2
   &= v (Y X)^{m-1} Y z + w (Y X)^{m} z
   \\
   &= v (Y X^*)^{m-1} Y + w^* (Y X^*)^{m} ,
\end{align*}
for some $w^* \in K$.
We note that
$(Y X^*)^{m} = (Y X)^{m} = (X Y)^{m} = (X Y^*)^{m}$.
   Therefore, we may take $\alpha = X^*$, and
we obtain
\[ \alpha^2 = \ba q_{\ba} + \ba \eta_{\ba}
\]
where the first summand is $v(YX)^{m-1}Y$, and the
second summand is $w(YX)^m$.
This proves part (i).
\medskip

(ii) \ 
As a tool, we prove two preliminary statements
which we will apply repeatedly.
(Here $R$ may be of finite representation type.)
\medskip

(1)  Assume that $\ba f(\ba)$ lies in $\soc(A)$. 
Then $\alpha^2$  lies in ${\rm soc}(e_iA)$:

If  $\ba f(\ba)$ lies in the socle, then it
is zero since it is not a cyclic path. The generators of
$\Omega^2(S_i)$ can be taken as
$$\vf = (\alpha, -q_{\ba} - \eta_{\ba}), \ \psi = (0, f(\ba)).
$$
We have $\vf \alpha' + \psi f^2(\ba) = 0$ where $\alpha' = \alpha + \zeta$
for $\zeta \in J^2$. Taking the first component gives
$\alpha\alpha'=0$ and hence $\alpha^2\alpha'=0$
As well $\alpha^2\ba \in {\rm soc}(e_iA)e_k = 0$. This means that
$\alpha^2J=0$ and then $\alpha^2 \in {\rm soc}(e_iA)$.

\medskip

(2) \ Assume $\alpha^2 \in \soc (e_iA)$. Then $f(\ba)f^2(\ba)\in {\rm soc} A$:

With the assumption, we have
$\alpha^2 = \ba \eta_{\ba}$. Then generators of
$\Omega^2(S_i)$ can be taken as
$$\vf = (\alpha,  - \eta_{\ba}), \ \psi = (*, f(\ba))
$$
where $*$ is $0$ or $-q_{\alpha} - \eta_{\alpha}$. Then
$\vf \alpha' + \psi f^2(\ba)=0$ where $\alpha' = \alpha + \zeta$ for
$\zeta \in J^2$. This in particular gives that
$$f(\ba) f^2(\ba) = \eta_{\ba}\alpha'i \ \ \in \soc(A).$$

\smallskip

(3) If there is no loop at vertex $k$ then 
$\alpha^2\not\in {\rm soc} (e_iA)$: 

In this case, we can apply Lemma \ref{lem:9.5} for the arrows $\gamma, f(\ba)$
starting at $k$, in particular $f(\ba)f^2(\ba)$ is not in the socle. 
Now we deduce from (2) that $\alpha^2 \not\in {\rm soc}(e_iA)$. 

\medskip

Now assume there is a loop, $\gamma$ say, at vertex $k$, then necessarily
$f(\gamma) = \gamma$.  We apply everything done so far to $\gamma$ instead
of $\alpha$. Hence we may assume that
$\gamma^2 \in {\rm soc}_2(e_kA)$.

\medskip

(4) \ We claim that if there is no loop at vertex
$x = t(f(\ba))$ then $\alpha^2 \not\in \soc(e_iA)$:

If there is no loop at vertex $x$ then by (3), the lemma holds for
$\gamma$, and $\gamma^2 \not\in {\rm soc}(e_kA)$. Now we use (1) with
$\gamma^2$ and
obtain that $f(\ba)f^2(\ba) \not\in {\rm soc}(A)$. This gives by (2) that
$\alpha^2 \not\in {\rm soc}(e_iA)$, as stated. 

\medskip
(5) \ We are left with the case when there is also a loop, 
$\delta$ say, at vertex $x$. Again, it must be  fixed by $f$, 
and now the quiver has only three vertices and
a loop at each vertex. 
As before, we may assume
that $\delta^2 \in {\rm soc}_2 (e_xA)$.

Assume for a contradiction that $\alpha^2 \in \soc (e_iA)$. 
We aim to show that then $A/{\rm soc} A$ is special biserial.
With this assumption,
we have $f(\ba)f^2(\ba) \in {\rm soc}(A)$ by (2). We apply (1) with $f(\ba), f^2(\ba)$ and deduce that $\gamma^2 \in {\rm soc}(e_kA)$. Now by (2) applied
to $\gamma$ we get that $f^2(\ba)\ba$ is in the socle of $A$.
Applying (1) with $f^2(\ba), \ba$ gives that $\delta^2 \in {\rm soc}(e_xA)$.
Now applying (2) for $\delta$ gives that $\ba f(\ba)$ is in the socle
of $A$.
In total we have that $A/{\rm soc}A$ is special biserial with
2-regular quiver, which is not possible
for an algebra of generalized quaternion type
by Auslander-Reiten theory.
This completes the proof of (ii). 
Part (iii) is proved in  (1). 
\end{proof}


By Lemma \ref{lem:9.5} and Lemma \ref{lem:9.6} we know already that
a product $\alpha f(\alpha)$ is never in the socle (and hence is non-zero). 
This is part of the statement of Theorem \ref{th:9.1}.
Furthermore, by Lemma \ref{lem:9.4} and also by Lemma \ref{lem:9.6}, 
we always have a lowest term presentation of $\alpha f(\alpha)$. 
The following relates the lowest terms of $\alpha f(\alpha)$ and of
$\ba f(\ba)$.


\begin{proposition}
\label{prop:9.7}
Let $i$ be a vertex of $Q$ and  let
$\alpha$ and $\bar{\alpha}$ be the arrows starting at $i$.
Assume that
\[
    \alpha f(\alpha)
     = \bar{\alpha} q_{\bar{\alpha}} + \bar{\alpha} \eta_{\bar{\alpha}}
 \qquad
  \mbox{ and }
 \qquad
    \bar{\alpha} f(\bar{\alpha})
     = {\alpha} q_{{\alpha}} + {\alpha} \eta_{{\alpha}} ,
\]
are lowest form presentations.
Then the following statements hold.
\begin{enumerate}[(i)]
 \item
  $f(\alpha) f^2(\alpha)$
  has  lowest term
  $q_{\alpha} f^2(\bar{\alpha})$,
  starting with $g(\alpha)$.
 \item
  $f(\bar{\alpha}) f^2(\bar{\alpha})$
  has  lowest term
  $q_{\bar{\alpha}} f^2(\alpha)$,
  starting with $g(\bar{\alpha})$.
 \item
  $\alpha f(\alpha) f^2(\alpha)
   = \bar{\alpha} f(\bar{\alpha}) f^2(\bar{\alpha})$.
\end{enumerate}
Moreover, if one of $\alpha f(\alpha)$ or $\bar{\alpha} f(\bar{\alpha})$
belongs to $\soc_2(e_i A)$, then
$\alpha f(\alpha) f^2(\alpha)$
is non-zero and spans $\soc(e_i A)$.
\end{proposition}

\begin{proof}
Since $A$ is of generalized quaternion type, there is
an exact sequence in $\mod A$
\[
  0 \to
  S_i \to
  P_i \xrightarrow{d_3}
  P_x \oplus P_y \xrightarrow{d_2}
  P_j \oplus P_k \xrightarrow{d_1}
  P_i \xrightarrow{d_0}
  S_i \to
  0 ,
\]
which gives rise to a minimal projective resolution
in $\mod A$ with the properties described in
Proposition~\ref{prop:4.3}.
In particular we have
$j = t(\alpha)$, $k = t(\bar{\alpha})$,
$x = s(f^2(\ba))$ and  $y = s(f^2(\alpha))$, and
$d_1(u,v) = \alpha u + \bar{\alpha} v$
for any $(u,v) \in P_j \oplus P_k$.
Consider the elements in $P_j \oplus P_k$ of the form
\begin{align*}
  \varphi &= \big(f(\alpha), - q_{\bar{\alpha}} - \eta_{\bar{\alpha}}\big),
  &
  \psi &= \big( - q_{\alpha} - \eta_{\alpha}, f(\bar{\alpha})\big) .
\end{align*}
Then $d_1(\varphi) = 0$ and $d_1(\psi) = 0$, and hence
$\varphi$ and $\psi$ generate $\Omega_A^2 (S_i)$.
Then there is a choice of arrows $f^2(\alpha)$ and
$f^2(\bar{\alpha})$ ending at $i$ such that
\[
  \varphi f^2(\alpha) + \psi f^2(\bar{\alpha}) = 0 .
\]
Hence we obtain the equalities
\[
  f(\alpha) f^2(\alpha)
   = q_{\alpha} f^2(\bar{\alpha}) + \eta_{\alpha} f^2(\bar{\alpha})
 \qquad
  \mbox{ and }
 \qquad
    f(\bar{\alpha}) f^2(\bar{\alpha})
     = q_{\bar{\alpha}} f^2(\alpha) + \eta_{\bar{\alpha}} f^2(\alpha) .
\]
Since $q_{\alpha}$ ends in $g^{-1}(f^2(\bar{\alpha}))$,
the first term of $f(\alpha) f^2(\alpha)$ is along the $g$-cycle
of ${\alpha}$, and has the same degree as $\alpha q_{\alpha}$.
This is smaller than the degree of ${\alpha} \eta_{{\alpha}}$,
and then all terms of $\eta_{\alpha} f^2(\bar{\alpha})$
have higher degree than the degree of $q_{\alpha} f^2(\bar{\alpha})$.
This shows that $q_{\alpha} f^2(\bar{\alpha})$ is the lowest
term of $f(\alpha) f^2(\alpha)$.
Similarly we conclude that $q_{\bar{\alpha}} f^2(\alpha)$ is the lowest
term of $f(\bar{\alpha}) f^2(\bar{\alpha})$.
Part (iii) follows directly:
\[
  \alpha f(\alpha) f^2(\alpha)
   = \bar{\alpha} q_{\bar{\alpha}} f^2(\alpha)
       + \bar{\alpha} \eta_{\bar{\alpha}} f^2(\alpha)
   = \bar{\alpha} \big( q_{\bar{\alpha}} f^2(\alpha)
                                 + \eta_{\bar{\alpha}} f^2(\alpha) \big)
   = \bar{\alpha} f(\bar{\alpha}) f^2(\bar{\alpha}).
\]

Assume now that
$\alpha f(\alpha)$ (say)  
belongs to $\soc_2(e_i A)$.
Then in the lowest form presentation of $\alpha f(\alpha)$ 
as in Lemma \ref{lem:9.4}, 
the first term is a non-zero element of ${\rm soc}_2(e_iA)$
while the second term belongs to ${\rm soc}(e_iA)$. Then
$$\alpha f(\alpha)f^2(\alpha) = \ba q_{\ba}f^2(\alpha),$$ which is
a non-zero scalar multiple of the socle monomial. 
The same argument applies when $\ba f(\ba)$ belongs to
$\soc_2(e_iA)$. 
\end{proof}


With these preparations, we can now prove Theorem \ref{th:9.1}. 
The exception in the next proposition will be dealt with 
in Proposition~\ref{prop:9.10}.

\begin{proposition}
\label{prop:9.8}
Let $i$ be a vertex of $Q$, 
$\alpha$ and $\bar{\alpha}$ the arrows
starting at $i$. In case $\cO(\alpha)\neq \cO(\ba)$
we assume that there are no double arrows
starting or ending at vertex $i$. 
Assume that
$\alpha \neq f(\alpha)$ and 
$\bar{\alpha} \neq f(\bar{\alpha})$.
Then we may assume that
$\alpha f(\alpha), f(\alpha) f^2(\alpha), f^2(\alpha) \alpha$
and
$\bar{\alpha} f(\bar{\alpha}), f(\bar{\alpha}) f^2(\bar{\alpha}),
 f^2(\bar{\alpha}) \bar{\alpha}$
belong to $\soc_2(A) \setminus \soc (A)$.
\end{proposition}

\begin{proof}
Let
$j = t(\alpha)$,
$k = t(\bar{\alpha})$,
$y = t(f(\alpha))$,
and
$x = t(f(\bar{\alpha}))$.
We have from
Lemma~\ref{lem:9.4} and Lemma~\ref{lem:9.5} lowest form presentations
\[
    \alpha f(\alpha)
     = \bar{\alpha} q_{\bar{\alpha}} + \bar{\alpha} \eta_{\bar{\alpha}}
 \qquad
  \mbox{ and }
 \qquad
    \bar{\alpha} f(\bar{\alpha})
     = {\alpha} q_{{\alpha}} + {\alpha} \eta_{{\alpha}}.
\]
Further, it follows from
Proposition~\ref{prop:9.7}
(and its proof) that
\[
  f(\alpha) f^2(\alpha)
   = q_{\alpha} f^2(\bar{\alpha}) + \eta_{\alpha} f^2(\bar{\alpha})
 \qquad
  \mbox{ and }
 \qquad
    f(\bar{\alpha}) f^2(\bar{\alpha})
     = q_{\bar{\alpha}} f^2(\alpha) + \eta_{\bar{\alpha}} f^2(\alpha) ,
\]
where
$q_{\alpha} f^2(\bar{\alpha})$
is the lowest term of $f(\alpha) f^2(\alpha)$
starting with $g(\alpha)$,
$q_{\bar{\alpha}} f^2(\alpha)$
is the lowest term of $f(\bar{\alpha}) f^2(\bar{\alpha})$
starting with $g(\bar{\alpha})$,
and
$\alpha f(\alpha) f^2(\alpha)
 = \bar{\alpha} f(\bar{\alpha}) f^2(\bar{\alpha})$.
We also know from
Proposition~\ref{prop:6.10}
that $e_i A$ is spanned by monomials
along the $g$-cycles of ${\alpha}$ and $\bar{\alpha}$.

We first prove that
$\alpha f(\alpha)$,
$\bar{\alpha} f(\bar{\alpha})$,
$f(\alpha) f^2(\alpha)$,
$f(\bar{\alpha}) f^2(\bar{\alpha})$
belong to $\soc_2(A) \setminus \soc (A)$.
We consider two cases.

\smallskip

(1)
Assume
$\cO(\alpha) = \cO(\bar{\alpha})$.

\smallskip

(i) We show $\alpha f(\alpha)$ and $\ba f(\ba)$ are in the second socle.
The local algebra $R = e_i A e_i$ has generators
\[
    X= {\alpha} g({\alpha}) \dots g^{a-1}({\alpha}) ,
 \qquad
  \mbox{ and }
 \qquad
    Y= \bar{\alpha} g(\bar{\alpha}) \dots g^{b-1}(\bar{\alpha}) ,
\]
with
$g^{a-1}({\alpha}) = f^2(\alpha)$,
$g^{a}({\alpha}) = \bar{\alpha}$,
$g^{b-1}(\bar{\alpha}) = f^2(\bar{\alpha})$,
$g^{b}(\bar{\alpha}) = {\alpha}$,
and $a + b = n$.
Hence $R$ has a basis consisting of monomials in $X,Y$
which alternate between $X$ and $Y$, and the socle of $R$
is spanned by $(X Y)^m = (Y X)^m$ for some $m \geq 1$.
It follows that
$\soc(e_i A) = e_iJ^{m n}$.

Since $\bar{\alpha} q_{\bar{\alpha}} f^2(\alpha)$
is a scalar multiple of a monomial which starts with
$\bar{\alpha}$ and ends in $f^2(\alpha)$
and
$\alpha q_{\alpha} f^2(\bar{\alpha})$
is a scalar multiple of a monomial which starts with
$\alpha$ and ends in $f^2(\bar{\alpha})$,
we have
\[
    \bar{\alpha} q_{\bar{\alpha}} f^2(\alpha) = c (Y X)^t
 \qquad
  \mbox{ and }
 \qquad
    \alpha q_{\alpha} f^2(\bar{\alpha}) = d (X Y)^s ,
\]
for some $c,d \in K$ and $s,t \geq 1$.
Moreover, these terms are the lowest terms of
$\alpha f(\alpha) f^2(\alpha)$
and
$\bar{\alpha} f(\bar{\alpha}) f^2(\bar{\alpha})$,
respectively.
We know also that
$\alpha f(\alpha) f^2(\alpha)
 = \bar{\alpha} f(\bar{\alpha}) f^2(\bar{\alpha})$,
and hence
$\bar{\alpha} q_{\bar{\alpha}} f^2(\alpha)
 = \alpha q_{\alpha} f^2(\bar{\alpha})$.
Then it follows from the basis of $R$ that $s = t = m$,
and $c=d \in K^*$.
Therefore
$\alpha f(\alpha) f^2(\alpha)$
and
$\bar{\alpha} f(\bar{\alpha}) f^2(\bar{\alpha})$
are non-zero scalar multiples of the socle monomial.
It follows that the lowest terms
$\bar{\alpha} q_{\bar{\alpha}}$ and ${\alpha} q_{{\alpha}}$
of $\alpha f(\alpha)$ and $\bar{\alpha} f(\bar{\alpha})$
are in $\soc_2(e_i A) \setminus \soc(e_i A)$.
But then $\alpha f(\alpha)$ and $\bar{\alpha} f(\bar{\alpha})$
also belong to $\soc_2(e_i A) \setminus \soc(e_i A)$.

\smallskip

(ii) We prove now that
$f(\alpha) f^2(\alpha)$
belongs to $\soc_2(e_j A)$.
Since $A$ is symmetric and
$\alpha f(\alpha) f^2(\alpha)$ is non-zero in $\soc (e_i A)$,
we conclude that
$f(\alpha) f^2(\alpha) \alpha$ is a non-zero element of $\soc (e_j A)$.
Similarly, since $(X Y)^m$ spans $\soc (e_i A)$,
we obtain also that
$\soc (e_j A)$ is spanned by the element
\[
 g({\alpha}) \dots g^{a-1}({\alpha})
 (Y X)^{m-1}
 \bar{\alpha} g(\bar{\alpha}) \dots g^{b-1}(\bar{\alpha}) \alpha ,
\]
of degree $mn$.
Now $f(\alpha)f^2(\alpha) \in e_jA$ with lowest term 
$q_{\alpha}f^2(\ba)$ of degree $mn - 1$, and it is the initial
submonomial of the above socle element of this degree. 
By Lemma~\ref{lem:6.11} it lies in the second socle, unless
posibly there are double arrows ending at $j$ and $g$ has two cycles
passing through $j$. 
However we are in the case $\cO(\alpha) = \cO(\ba)$, so if $t(\alpha)=j=t(\ba)$ then $g$ passes twice through $j$, so this does not occur. 

Similarly one shows that $f(\ba)f^2(\ba)$ is in the second socle.

\smallskip

(2)
Assume that
$\cO(\alpha) \neq \cO(\bar{\alpha})$.
Let $n_{\alpha} = |\cO(\alpha)|$
and $n_{\bar{\alpha}} = |\cO(\bar{\alpha})|$.
Then the local algebra $R = e_i A e_i$ has generators
\[
    X= {\alpha} g({\alpha}) \dots g^{n_{\alpha}-1}({\alpha}) ,
 \qquad
  \mbox{ and }
 \qquad
    Y= \bar{\alpha} g(\bar{\alpha}) \dots g^{n_{\bar{\alpha}}-1}(\bar{\alpha}) ,
\]
and there are positive integers
$m_{\alpha}$ and $m_{\bar{\alpha}}$
such that each of the elements
$X^{m_{\alpha}}$ and $Y^{m_{\bar{\alpha}}}$
spans the socle of $R$.

Since
$\bar{\alpha} q_{\bar{\alpha}} f^2(\alpha)$
is a scalar multiple of a monomial which starts with
$\bar{\alpha}$ and ends in $f^2(\alpha)$ and
$\alpha q_{\alpha} f^2(\bar{\alpha})$
is a scalar multiple of a monomial which starts with
$\alpha$ and ends in $f^2(\bar{\alpha})$,
we conclude that
\[
    \bar{\alpha} q_{\bar{\alpha}} f^2(\alpha) = a Y^t
 \qquad
  \mbox{ and }
 \qquad
    \alpha q_{\alpha} f^2(\bar{\alpha}) = b X^s ,
\]
for some $a,b \in K$ and $s,t \geq 1$.
Moreover,
by Proposition~\ref{prop:9.7},
$\bar{\alpha} q_{\bar{\alpha}} f^2(\alpha)$
is the lowest term of $\alpha f(\alpha) f^2(\alpha)$,
$\alpha q_{\alpha} f^2(\bar{\alpha})$
is the lowest term of $\bar{\alpha} f(\bar{\alpha}) f^2(\alpha)$,
and
$\alpha f(\alpha) f^2(\alpha)
 = \bar{\alpha} f(\bar{\alpha}) f^2(\bar{\alpha})$.
Hence
$ \bar{\alpha} q_{\bar{\alpha}} f^2(\alpha)
 = \alpha q_{\alpha} f^2(\bar{\alpha})$.

Both elements are non-zero scalar multiples of monomials along $g$-cycles
of $\alpha$ and $\ba$. It follows from Lemma~\ref{lem:6.11} that 
$s = m_{\alpha}$ and $t = m_{\bar{\alpha}}$,
$a,b \in K^*$, and the element belongs to $\soc(R) = \soc(e_i A)$.

By our assumption, there are no double arrows ending at $i$ and
 then by Lemma~\ref{lem:6.11}
we can deduce that $\bar{\alpha} q_{\bar{\alpha}}$
and $\bar{\alpha} q_{{\alpha}}$
belong to
$\soc_2(e_i A) \setminus \soc(e_i A)$,
and consequently
$\alpha f(\alpha)$ and $\bar{\alpha} f(\bar{\alpha})$
belong to $\soc_2(e_i A) \setminus \soc(e_i A)$.

\medskip

We prove now that
$f(\alpha) f^2(\alpha) \in \soc_2(e_j A) \setminus \soc (e_j A)$.

Consider $\soc(e_jA)$, it is spanned by $f(\alpha)f^2(\alpha)\alpha$ and also by
$X_j^{m_{\alpha}}$, where
$$X_j:= (g(\alpha)g^2(\alpha)\ldots f^2(\ba)\alpha).
$$
From the lowest form presentation we have
$$f(\alpha)f^2(\alpha)\alpha = q_{\alpha}f^2(\ba)\alpha + \eta_{\alpha}f^2(\ba)
\alpha
$$
and the first term is a non-zero scalar multiple of a monomial in the basis
and the second term is expanded in terms of the basis. 
Therefore
$q_{\alpha}f^2(\alpha)\alpha = c(X_j^{m_{\alpha}})$ for $0\neq c\in K$.
By Lemma~\ref{lem:6.11}, and since no double arrows end at $j$, we deduce that
$q_{\alpha}f^2(\ba) \in \soc_2(e_jA)$ and hence  
$f(\alpha)f^2(\alpha)$ is in the second socle.

Similarly one shows that $f(\ba)f^2(\ba) \in {\rm soc}_2(e_kA)$.
The proof will be completed by applying the following lemma.
\end{proof}

\begin{lemma}\label{lem:9.9}
Assume $\alpha\neq f(\alpha)$, and assume we know that
$\alpha f(\alpha)$ and $f(\alpha)f^2(\alpha)$ are in the second
socle of $A$. Assume also that if $\cO(\alpha)\neq \cO(\ba)$
then no double arrows start or end at vertex $i$. Then
there is a choice for $f^2(\alpha)$
such that
$f^2(\alpha) \alpha \in \soc_2(e_y A)$,
and also
$f(\alpha) f^2(\alpha) \in \soc_2(e_j A)$.
\end{lemma}


\begin{proof}
We choose the generators $X_j$ and $Y_j$ of the local algebra
$e_j A e_j$ such that
$X_j$ is a monomial starting with $g(\alpha)$
and
$Y_j$ is a monomial starting with $f(\alpha)$,
depending on the two possible cases
$\cO(g(\alpha)) = \cO(f(\alpha))$
or
$\cO(g(\alpha)) \neq \cO(f(\alpha))$
(as in (1) or (2) of the proof of Proposition~\ref{prop:9.8}).
Moreover, let
$Z_j$ be the monomial starting with $g(f(\alpha))$
such that $f(\alpha) Z_j = Y_j$.
Since $f(\alpha) f^2(\alpha) \alpha$ is non-zero in $\soc (e_j A)$,
we may write
\[
   f^2(\alpha) \alpha = \lambda Z_j (X_j Y_j)^{m-1} X_j + u
   \quad
   \ \ \ \quad
   \mbox{(if $\cO(g(\alpha)) = \cO(f(\alpha))$)}
\]
or
\[
   f^2(\alpha) \alpha = \lambda Z_j Y_j^{m-1} + u
   \quad
   \mbox{(if $\cO(g(\alpha)) \neq \cO(f(\alpha))$)},
\]
for $\lambda \in K^*$
and $u \in e_y A e_j$ with $f(\alpha) u = 0$, where $m=m_{f(\alpha)}$.
We note that also $u f(\alpha)$ is in the socle of $A$.
We write
\[
  u = f^2(\alpha) u_1 + g\big(f(\alpha)\big) u_2 ,
\]
where the terms are along the basis of $e_y A$.
Note also that $u_1$ starts with $\ba$ if it is non-zero, and hence
lies in $J^2$.

\medskip

(i) \ We claim that $g(f(\alpha))u_2=0$.
We have  $0=f(\alpha)u$, and  $f(\alpha)f^2(\alpha) u_1 \in {\rm soc}_2(A)J^2 = 0$. Hence
also
$$f(\alpha)g(f(\alpha))u_2 = 0$$
Assume for a contradiction that
$g(f(\alpha))u_2\neq 0$, then it is along the cycle of $g$, and then
$f(\alpha) g(f(\alpha))u_2$ is in $e_jAe_j$ and is not zero, a contradiction.
So $g(f(\alpha))u_2=0$.

\medskip

We have now
$$f^2(\alpha)\alpha = v + f^2(\alpha) u_1$$
and $v\in {\rm soc}_2(A)$, and moreover $f^2(\alpha)u_1f(\alpha) \in {\rm soc}(A)$.
We expand the error term,
$$f^2(\alpha)u_1 = f^2(\alpha)u_1' \alpha + f^2(\alpha)u_1''g^{-1}(f(\alpha)).$$
Here one or both terms may not exist, depending on whether $\alpha, f(\alpha)$
are in the same $g$-cycle of $f^2(\alpha)$. 

\medskip

(ii) \ We claim that we can write 
$$f^2(\alpha) \alpha = v_1 + f^2(\alpha)u_1'\alpha
$$
with $v_1\in \soc_2(A)$. 
We must show that $w:= f^2(\alpha)u_1''g^{-1}(f(\alpha))$ is in the second
socle. If $w\neq 0$, then it is a linear combination of
monomials along
 $g$-cycle. Postmultiplying $f^2(\alpha)u_1$ with  
 $f(\alpha)$ gives an element
 in the socle and therefore $wf(\alpha)$ is non-zero in the socle. 
Since $w$ exists, the arrow $f(\alpha)$ belongs to $\cO(f^2(\alpha)) = \cO(\ba)$.
This means that $\cO(\ba)$ passes twice through vertex $y$. We can therefore
deduce from Lemma~\ref{lem:6.11} that
$w$ is in the second socle. 

\medskip

By construction, 
$u_1'\in J$ and therefore $e_i-u_1'$ is a unit.
We replace $f^2(\alpha)$ by $f^2(\alpha)':= f^2(\alpha)(e_i-u_1')$ and get
$$f^2(\alpha)'\alpha,  \ f(\alpha)f^2(\alpha)' \in {\rm soc}_2(A)
$$
as required.
\end{proof}


Proposition~\ref{prop:9.8} has proved Theorem~\ref{th:9.1}, excluding
a certain setting.   
The following deals with this exception.

\begin{proposition}
\label{prop:9.10}
Assume $\cO(\alpha) \neq \cO(\ba)$ and that either
$\alpha$ and $\bar{\alpha}$ are  double arrows, or 
$f^2(\alpha), f^2(\ba)$ are double arrows in   $Q$, where
$Q$ is not the Markov quiver.
Then the conclusion of Theorem \ref{th:9.1} holds.
\end{proposition}


\begin{proof}
Suppose $\alpha, \ba$ are double arrows and belong
to different $g$-cycles. Let $y=t(f(\alpha)$. Then no double
arrows start or end at vertex $y$ since otherwise $t(f(\ba))$ would
also be equal to $y$ and $Q$ would be the Markov quiver.
So we can apply Proposition~\ref{prop:9.8} with vertex $y$ and arrows
$f^2(\alpha), \overline{f^2(\alpha)}$ instead of $i, \alpha, \ba$. We get
the claim for the products along the $f$-cycle of
$\alpha$. 
Similarly, using vertex $x=t(f(\ba))$ gives the answer
for the products along the $f$-cycle of $\ba$.

\medskip

Suppose $f^2(\alpha), f^2(\ba)$ are double arrows. They are in different
cycles of $g$ as well. Then no double arrows start or end at vertex
$j$. We apply Proposition~\ref{prop:9.8} with 
vertex $j$, and arrows $f(\alpha)$ and $\overline{f(\alpha)}$ and get
the answer for the $f$-cycle of $\alpha$. 
Similarly, the result follows for the $f$-cycle of $\ba$. 
\end{proof}


It remains to prove Theorem \ref{th:9.1} in the case when
there is a loop at vertex $i$ fixed by $f$. Most of the information
has already been obtained in Lemma \ref{lem:9.6}. 

\begin{proposition}
\label{prop:9.11}
Let $\alpha$ be a loop in $Q$ with $\alpha = f(\alpha)$,
$i = s(\alpha)$, and $\bar{\alpha}$ the other arrow in $Q$
with $s(\bar{\alpha}) = i$.
Then there is a choice for $\alpha$ such that
the following statements hold.
\begin{enumerate}[(i)]
 \item
  $\alpha^2 \in \soc_2(e_i A) \setminus \soc(e_i A)$.
 \item
  There exist elements $c_{\bar{\alpha}} \in K^*$, $b_i \in K$,
  and a positive integer $m_{\bar{\alpha}}$ such that
  \[
    \alpha^2 = c_{\bar{\alpha}} A_{\bar{\alpha}} + b_i B_{\bar{\alpha}} ,
  \]
 where
  \begin{align*}
    A_{\bar{\alpha}}
      &= \big(\bar{\alpha} g(\bar{\alpha}) \dots f^2(\bar{\alpha}) \alpha
            \big)^{m_{\bar{\alpha}} - 1}
            \bar{\alpha} g(\bar{\alpha}) \dots f^2(\bar{\alpha}) ,
    \\
    B_{\bar{\alpha}}
      &= \big(\bar{\alpha} g(\bar{\alpha}) \dots f^2(\bar{\alpha}) \alpha
           \big)^{m_{\bar{\alpha}}},
 \end{align*}
  and these elements are independent on choices of
  $\bar{\alpha}, g(\bar{\alpha}), \dots, f^2(\bar{\alpha})$.
 \item
  We may assume that
  $\bar{\alpha} f(\bar{\alpha})  \in \soc_2(e_i A) \setminus \soc(e_i A)$
  and is of the form
  \[
    \bar{\alpha} f(\bar{\alpha}) = c_{\alpha} A_{\alpha} ,
  \]
  with $c_{\alpha} = c_{\bar{\alpha}}$
  and
  $A_{\alpha} = (\alpha \bar{\alpha} g(\bar{\alpha}) \dots f^2(\bar{\alpha})
          )^{m_{\bar{\alpha}}-1}
          \alpha \bar{\alpha} g(\bar{\alpha}) \dots g^{-1}(f^2(\bar{\alpha}))
          $.
 \item
   We may assume that
   $f(\bar{\alpha}) f^2(\bar{\alpha})$
   and
   $f^2(\bar{\alpha})\bar{\alpha}$
   belong to $\soc_2(A) \setminus \soc (A)$.
\end{enumerate}
\end{proposition}

\begin{proof} For
parts (i) and (ii) see Lemma \ref{lem:9.6}. 
To prove (iii),  
we take the lowest form presentation 
\[
   \bar{\alpha} f(\bar{\alpha})
     = {\alpha} q_{{\alpha}} + {\alpha} \eta_{{\alpha}} .
\]
Then it follows from
Proposition~\ref{prop:9.7}
and (ii) that
\[
  \bar{\alpha} f(\bar{\alpha}) f^2(\bar{\alpha})
   = \alpha f(\alpha) f^2(\alpha)
   = \alpha^3 ,
\]
which is non-zero in the socle.
We deduce that
$\alpha q_{\alpha}f^2(\bar{\alpha})$ is non-zero in the socle.
We have the expansion along cycles of $g$, and by Lemma~\ref{lem:6.11} we
deduce that
$\alpha q_{\alpha}$ is in the second socle (there are no double arrows ending
at vertex $i$). 
This is the lowest term of $\ba f(\ba)$ and hence $\ba f(\ba)$ is in the
second socle. The details in (iii) follow. 

\medskip

(iv)
Let $k = t(\bar{\alpha})$ and $x = t(f(\bar{\alpha}))$.
Since
$\alpha^2
 = \bar{\alpha} q_{\bar{\alpha}} + \bar{\alpha} \eta_{\bar{\alpha}}$
and
$\bar{\alpha} f(\bar{\alpha}) = {\alpha} q_{{\alpha}}$,
it follows from
Proposition~\ref{prop:9.7}
that $q_{\bar{\alpha}} f^2({\alpha}) = q_{\bar{\alpha}} \alpha$
is the lowest term of $f(\bar{\alpha}) f^2(\bar{\alpha})$
and starts with $g(\bar{\alpha})$.

Further, because $A$ is symmetric and
$\bar{\alpha} f(\bar{\alpha}) f^2(\bar{\alpha})$
is non-zero in $\soc(e_k A)$,
we conclude that
$\ba q_{\ba}\alpha$ is a non-zero monomial in the socle of
$Ae_i$. This is also the socle of $Ae_i$.
We use the dual version of Lemma~\ref{lem:6.11}, noting that no double
arrows start at $i$. We deduce that $q_{\ba}\alpha \in \soc_2(Ae_i) \subseteq
\soc_2(A)$ and hence
the element lies in $\soc_2(e_kA)$. 
By the Lemma~\ref{lem:9.9} it follows that we may assume that also
$f^2(\ba)\ba$ belongs to the second socle of $A$. This completes the
proof.
\end{proof}

We have proved Theorem \ref{th:9.1}. The following  gives some more detail
in the case when we have one type C relation starting at a vertex.


\begin{proposition}
\label{prop:9.12}
Let $i$ be a vertex of $Q$,
$\alpha,\bar{\alpha}$ the arrows
with source $i$,
and assume that
the minimal relations starting from $i$
are of the form
${\alpha} f({\alpha}) + \bar{\alpha} \gamma \in J^3$
of type C and
$\bar{\alpha} f(\bar{\alpha}) \in J^3$.
Moreover, assume that each of the vertices
$k = t(\bar{\alpha})$
and
$y = t(f(\alpha))$
is the starting vertex of at most one
minimal relation of type C.
Then the following statements hold.
\begin{enumerate}[(i)]
 \item
  We may assume that we have the commutativity relations
  of type C
  \begin{align*}
    &&
    \alpha f(\alpha) &= \bar{\alpha} \gamma, &
    f(\bar{\alpha}) f^2(\bar{\alpha}) &= \gamma f^2(\alpha), &
    f(\gamma) f^2(\gamma) &= f^2(\alpha) \gamma
    &&
  \end{align*}
  starting at the vertices $i,k,y$.
 \item
  We have the relations
  \begin{align*}
    f(\alpha) f^2(\alpha) &= p f^2(\bar{\alpha}) \in J^3,
  &
    f^2(\alpha) \alpha &= f(\gamma) v,
  \\
    \bar{\alpha} f(\bar{\alpha}) &= \alpha p \in J^3,
  &
    f^2(\bar{\alpha}) \bar{\alpha} &= q f^2(\gamma),
  \\
    \gamma f(\gamma) &= f(\bar{\alpha}) q \in J^3,
  &
    f^2(\gamma) \gamma &= v f(\alpha) \in J^3,
  \end{align*}
  where the elements on the right sides are expanded
  along cycles of $g$.
 \item
  The local algebras $e_i A e_i$, $e_k A e_k$, $e_y A e_y$
  have finite type.
 \item
  The paths of length two in \emph{(i)} and \emph{(ii)} belong to
  $\soc_2(A) \setminus \soc(A)$.
\end{enumerate}
\end{proposition}


\begin{proof}
(i) The argument in the proof of Lemma~\ref{lem:7.1} (part (i))
shows that
$$    \alpha f(\alpha)
     = \bar{\alpha} \gamma.
$$
By the imposed assumption, we have
$\bar{\alpha} f(\bar{\alpha}) \in J^3$.
Take the lowest term presentation
\[
    \bar{\alpha} f(\bar{\alpha})
     = {\alpha} q_{{\alpha}} + {\alpha} \eta_{{\alpha}} . 
\]
We have
$\cO(\alpha) = ( {\alpha} g({\alpha}) \dots g^{n_{\alpha}-1}({\alpha}))$
with $g^{n_{\alpha}-1}(\alpha) = f^2(\bar{\alpha})$,
by Lemma~\ref{lem:6.7},
and hence
$g^{-1}(f^2(\bar{\alpha})) = g^{n_{\alpha}-2}(\alpha)$.
Since ${\alpha} f({\alpha}) \in J^3$,
we conclude that
${\alpha} q_{{\alpha}}$ is of length $\geq 3$,
and then $n_{\alpha} \geq 4$.
We set $p = q_{{\alpha}} + \eta_{{\alpha}}$,
so we have
$\bar{\alpha} f(\bar{\alpha}) = {\alpha} p$.
Applying now
Proposition~\ref{prop:4.8}
to the vertex $i$, we get
\[
    f(\bar{\alpha}) f^2(\bar{\alpha}) = \gamma f^2(\alpha)
 \qquad
  \mbox{ and }
 \qquad
    f(\alpha) f^2(\alpha) = p f^2(\bar{\alpha}).
\]
Observe that
$f(\bar{\alpha}) f^2(\bar{\alpha}) = \gamma f^2(\alpha)$
is a commutativity relation of type C starting from $k$.
Then it follows from the assumption that
$\gamma f(\gamma) \in J^3$. Take the lowest form presentation
\[
    \gamma f(\gamma)
     = f(\bar{\alpha}) q_{f(\bar{\alpha})}
        + f(\bar{\alpha}) \eta_{f(\bar{\alpha})} .
\]
We set $q = q_{f(\bar{\alpha})} + \eta_{f(\bar{\alpha})}$,
so we have
$\gamma f(\gamma)  = f(\bar{\alpha}) q$.
We infer as above that
$f(\bar{\alpha}) q_{f(\bar{\alpha})}$ is of length $\geq 3$,
and hence $|\cO(f(\bar{\alpha}))| \geq 4$.
Applying now
Proposition~\ref{prop:4.8}
to the vertex $k$, we obtain
\[
    f(\gamma) f^2(\gamma) = f^2(\bar{\alpha}) \bar{\alpha}
 \qquad
  \mbox{ and }
 \qquad
    f^2(\bar{\alpha}) \bar{\alpha} = q f^2(\gamma).
\]
We note that $q$ is a monomial starting from $g(f(\bar{\alpha}))$.
Since
$f(\gamma) f^2(\gamma) = f^2(\bar{\alpha}) \bar{\alpha}$
is a commutativity relation of type C starting from $y$,
it follows from the assumption that
$f^2(\alpha)\alpha \in J^3$.
Take the lowest form presentation
\[
    f^2(\alpha) \alpha
     = f(\gamma) q_{f(\gamma)}
        + f(\gamma) \eta_{f(\gamma)}. 
\]
Then we infer as above that
$f(\gamma) q_{f(\gamma)}$ is of length $\geq 3$,
and hence $|\cO(f(\gamma))| \geq 4$.
We set $v = q_{f(\gamma)} + \eta_{f(\gamma)}$,
so we have
$f^2(\alpha) \alpha = f(\gamma) v$.
Finally, applying now
Proposition~\ref{prop:4.8}
to the vertex $y$, we obtain
the already known commutativity relation
$\alpha f(\alpha) = \bar{\alpha} \gamma$,
and the new relation
\[
    f^2(\gamma) \gamma = v f(\alpha) .
\]
This proves (i) and (ii).

\medskip

(iii)
Consider the local algebras
$R_i = e_i A e_i$,
$R_k = e_k A e_k$,
$R_y = e_y A e_y$,
and the elements of these algebras
\begin{align*}
  &&
    X_i &= \bar{\alpha} \gamma f^2(\alpha), &
    X_k &= \gamma f^2(\alpha) \bar{\alpha}, &
    X_y &= f^2(\alpha) \bar{\alpha} \gamma .
  &&
\end{align*}
By the relations established above, we have the equalities
\begin{align*}
  \alpha f(\alpha) f^2(\alpha)
   &= \bar{\alpha} \gamma f^2(\alpha)
     = \bar{\alpha} f(\bar{\alpha}) f^2(\bar{\alpha})  ,
\\
  f(\bar{\alpha}) f^2(\bar{\alpha}) \bar{\alpha}
   &= f(\bar{\alpha}) q f^2(\gamma)
     = \gamma f(\gamma) f^2(\gamma)
     = \gamma f^2(\alpha) \bar{\alpha} ,
\\
  f^2(\alpha) \alpha f(\alpha)
   &= f^2(\gamma) v f(\alpha)
     = f(\gamma) f^2(\gamma) \gamma
     = f^2(\alpha) \bar{\alpha} \gamma .
\end{align*}
We observe also that $X_i$ is the composition of all arrows
along the  $g$-orbit of ${\alpha}$,
$X_k$ is the composition of all arrows along
the  $g$-orbit of $f(\bar{\alpha})$,
and $X_y$ is the composition of all arrows along
the  $g$-orbit of $f(\gamma)$.
Therefore we obtain that
$R_i$ is generated by $X_i$,
$R_k$ is generated by $X_k$,
and
$R_y$ is generated by $X_y$.
We claim that
$X_i^2 = 0$,
$X_k^2 = 0$,
and
$X_y^2 = 0$.
By the symmetricity of $A$ and rotation argument,
it is enough to show that
$X_i^2 = 0$.
We have the equalities
\begin{align*}
  X_i \alpha
   &= \bar{\alpha} \gamma f^2(\alpha) \alpha
    = \bar{\alpha} \gamma f(\gamma) v
    = \bar{\alpha} f(\bar{\alpha}) q v
    = \alpha p q v
 ,
\\
  X_i \alpha
   &=  \alpha g(\alpha) \dots g^{n_{\alpha}-1} (\alpha)  \alpha
    = \alpha p  f^2(\bar{\alpha}) \alpha
 .
\end{align*}
Moreover, by the properties of $q$ and $v$ established above,
$\alpha p q v$ is a combination of monomials along
the $g$-orbit of ${\alpha}$ of degrees $\geq n_{\alpha} + 3$.
Therefore, $X_i \alpha = 0$.
But then $X_i^2 = X_i  \alpha f(\alpha) f^2(\alpha) = 0$.
Summing up, we proved that
$X_i$ spans $\soc(R_i) = \soc(e_i A)$,
$X_k$ spans $\soc(R_k) = \soc(e_k A)$,
and
$X_y$ spans $\soc(R_y) = \soc(e_y A)$.
In particular,
$R_i$, $R_k$, $R_y$
are of finite type.

\medskip

(iv) is a special case of Proposition~\ref{prop:9.8}.  
\end{proof}

We may visualize the situation described in the above
proposition as follows
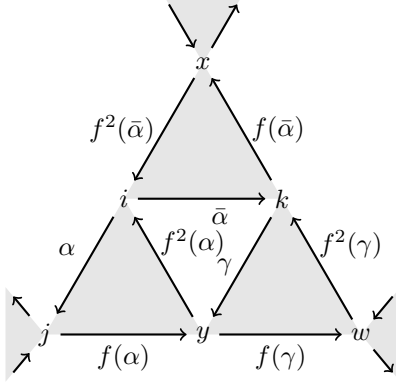
\[
\begin{tikzpicture}
[scale=1.05]
\coordinate (1) at (0,1.72);
\coordinate (2) at (0,-1.72);
\coordinate (3) at (2,-1.72);
\coordinate (4) at (-1,0);
\coordinate (5) at (1,0);
\coordinate (6) at (-2,-1.72);

\coordinate (1a) at (-0.5,2.58);
\coordinate (1b) at (0.5,2.58);

\coordinate (6a) at (-2.5,-1.12);
\coordinate (6b) at (-2.5,-2.32);

\coordinate (3a) at (2.5,-1.12);
\coordinate (3b) at (2.5,-2.32);

\fill[fill=gray!20]
    (1) -- (4) -- (5) -- cycle;
\fill[fill=gray!20]
    (2) -- (4) -- (6) -- cycle;
\fill[fill=gray!20]
    (2) -- (3) -- (5) -- cycle;

\fill[fill=gray!20]
    (1a) -- (1) -- (1b) -- cycle;
\fill[fill=gray!20]
    (6a) -- (6) -- (6b) -- cycle;
\fill[fill=gray!20]
    (3a) -- (3) -- (3b) -- cycle;

\node (1) at (0,1.72) {$x$};
\node (2) at (0,-1.72) {$y$};
\node (3) at (2,-1.72) {$w$};
\node (4) at (-1,0) {$i$};
\node (5) at (1,0) {$k$};
\node (6) at (-2,-1.72) {$j$};

\node (1a) at (-0.5,2.58) {};
\node (1b) at (0.5,2.58) {};

\node (6a) at (-2.5,-1.12) {};
\node (6b) at (-2.5,-2.32) {};

\node (3a) at (2.5,-1.12) {};
\node (3b) at (2.5,-2.32) {};

\draw[->,thick]
 (3a) edge (3)
 (3) edge (3b)
 (6b) edge (6)
 (6) edge (6a)
 (1a) edge (1)
 (1) edge (1b)
;
\draw[->,thick]
 (1) edge node [left] {$f^2(\bar{\alpha})$} (4)
 (4) edge node [below right] {$\bar{\alpha}$} (5)
 (5) edge node [right] {$f(\bar{\alpha})$} (1)
(2) edge node [below] {$f(\gamma)$} (3)
(3) edge node [above right] {$\!\!\!f^2(\gamma)$} (5)
(5) edge node [left] {$\gamma$} (2)
(2) edge node [above right] {\!\!\!$f^2(\alpha)$} (4)
(4) edge node [above left] {$\alpha$} (6)
(6) edge node [below] {$f(\alpha)$} (2)
;
\end{tikzpicture}
\]
where the shaded triangles denote $f$-orbits.
We note that the situation when there is one minimal
relation of type C starting from $i$ but two minimal
relations starting in $k$ or $y$ (or both $k$ and $y$)
is covered by
Proposition~\ref{prop:7.2}.

We end this section with the following lemma.

\begin{lemma}
\label{lem:9.13}
Let $i$ be a vertex of $Q$,
$\alpha$ and $\bar{\alpha}$ the arrows starting at $i$,
and
${\alpha} \neq f({\alpha})$. Assume 
that each product of length two along
the $f$-cycle of $\alpha$ belongs to
the second socle. Assume we have the  lowest form presentation 
$$\alpha f(\alpha) = \ba q_{\ba} + \ba \eta_{\ba} \in \soc_2(e_iA).
$$
Then $\alpha f(\alpha) = \ba q_{\ba}$. 
\end{lemma}

\begin{proof}
Observe that the lowest term $\bar{\alpha} q_{\bar{\alpha}}$
of $\alpha f(\alpha)$ belongs to
$\soc_2(e_i A) \setminus \soc(e_i A)$. Then $\ba \eta_{\ba}$ belongs
to $\soc(e_iA)$ and so 
$\bar{\alpha} \eta_{\bar{\alpha}} = \bar{\alpha} \eta_{\bar{\alpha}} e_i$.
On the other hand,
$\bar{\alpha} \eta_{\bar{\alpha}}$ belongs to $e_i A e_j$
with $j = t(f(\alpha))$.
If $i\neq j$ then 
$\bar{\alpha} \eta_{\bar{\alpha}} = 0$.
So assume now $i=j$. This means that $\ba$ is fixed by the permutation
$g$, and then the lowest form presentation is
$$\alpha f(\alpha) = c\ba^{m-1} + c'\ba^m \ \ (c, c\in K).
$$
This can be written as $c\ba^{m-1} (1- d\ba)$ and we may replace
$f(\alpha)$ by $f(\alpha)' = f(\alpha)(1-d\ba)^{-1}$. Note also that
$\ba = f^2(\ba)$ and therefore $f(\alpha)\ba$ is in the second socle.
Therefore, in any monomial $\mu$ along a cycle of $g$ in which $f(\alpha)$ occurs,
if we substitute $f(\alpha) = f(\alpha)'- f(\alpha)p$ then the error term
$f(\alpha)p$ is at least in the second socle and the relevant
summand of $\mu$ is zero. 
With this, we get the claim also in this case.
\end{proof}


\section{Proof of the Main Theorem}\label{sec:proofth}

The implication
(ii) $\Rightarrow$ (i)
follows from the general theory
because $A$ is re\-pre\-sen\-ta\-tion-infinite
(see \cite[Theorem~IV.11.19]{SY}).
The implication
(iii) $\Rightarrow$ (ii)
follows from
Theorems \ref{th:2.4}, \ref{th:2.6}, \ref{th:2.7}, \ref{th:tetr2}.
Hence it remains to prove that (i) implies (iii).

Let $A = K Q/I$ be a $2$-regular algebra of generalized quaternion type.
We may assume that $A$ is not an algebra considered in
Theorems \ref{th:5.1}, \ref{th:8.1}, \ref{th:8.2}.
In particular, $Q$ is not the Markov quiver.
Moreover, let $f$ and $g$ be the permutations of arrows in $Q$
fixed in the proof of
Theorem~\ref{th:6.1},
with the properies listed at the beginning of
Section~\ref{sec:localpresentation}.
We will prove that $A$ is isomorphic to a socle deformed
weighted triangulation algebra
$\Lambda(Q,f,m_{\bullet},c_{\bullet},b_{\bullet})$
of the triangulation quiver $(Q,f)$.

We will prove first the following theorem on the global presentation
of $A$.

\begin{theorem}
\label{th:10.1}
There is a choice of representatives of arrows of $Q$
in $A$ such that the following statements hold.
\begin{enumerate}[(i)]
 \item
  For each arrow $\alpha$
  with $f(\alpha) \neq \alpha$,
  we have
  \[
  \alpha f(\alpha)
     = \bar{\alpha} q_{\bar{\alpha}}
     \in \soc_2(A) \setminus \soc(A)
   ,
  \]
  where
  $\bar{\alpha} q_{\bar{\alpha}}$ is
  a non-zero scalar multiple of a monomial
  along the $g$-orbit of $\bar{\alpha}$,
  ending in $g^{-1}(f^2({\alpha}))$.
 \item
  For each loop $\alpha$
  with $f(\alpha) = \alpha$,
  we have
  \[
  \alpha^2
     = \bar{\alpha} q_{\bar{\alpha}} + \bar{\alpha} \eta_{\bar{\alpha}}
     \in \soc_2(A) \setminus \soc(A)
   ,
  \]
  where
  $\bar{\alpha} q_{\bar{\alpha}}$ is
  a non-zero scalar multiple of a monomial
  along the $g$-orbit of $\bar{\alpha}$,
  ending in $g^{-1}(\alpha)$,
  and
  $\bar{\alpha} \eta_{\bar{\alpha}}$ is
  a scalar multiple of a monomial
  along the $g$-orbit of $\bar{\alpha}$,
  ending in $\alpha$
  and lying in $\soc(A)$.
 \item
 For each arrow $\alpha$,
 $\alpha f(\alpha) f^2(\alpha)
  = \bar{\alpha} f(\bar{\alpha}) f^2(\bar{\alpha})$
 and
 is a non-zero element of $\soc(A)$.
\end{enumerate}
\end{theorem}

\begin{proof}
(a) \ First we show that there is a choice of arrows such that we have
globally that for each arrow $\alpha$, the product $\alpha f(\alpha)$ is in
the second socle. 
By Lemmas \ref{lem:9.5} and \ref{lem:9.6}, 
it is never in the socle and we will not comment on this any further.

(1) For each loop fixed by $f$, we choose $\alpha$ such that
$\alpha^2 \in \soc_2(A)$, using Proposition~\ref{prop:9.8}. 

(2) We fix arrows along all cycles of $f$ at which at least one relation
of type C starts, by using  Lemma~\ref{lem:7.1}    
and Proposition~\ref{prop:9.12}. 

Note that the arrows in (1) and (2) are disjoint, by Lemma~\ref{lem:4.7}. As well, the arrows
fixed in (2) are not parts of double arrows. Hence the following choices
are also disjoint from (1) and (2)

(3) We fix arrows along cycles of $f$ where one of them is part of a
double arrow, according to Lemma \ref{lem:9.13}. 

\medskip

Let $\cF$ be the set of all cycles of $f$.
We define $\cF_0$ to be the subset of $\cF$
containing the cycles fixed in (1), (2) and (3). 
If there is no cycle of $f$ in $Q$ satisfying (1), (2), (3),
then we define $\cF_0$ consisting of two $f$-cycles through
the some randomly chosen vertex of $Q$, which satsify Proposition~\ref{prop:9.8}.

Assume $\cF_0\subseteq \cF_1 \subseteq \cF$ and assume for all
arrows $\alpha$ along $f$-cycles belonging to $\cF_1$ we have
that $\alpha f(\alpha)$ is in $\soc_2(A)$. Since $Q$
is connected, 
there exist arrows $\alpha$
and $\bar{\alpha}$ starting from a vertex $i$ such
that the $f$-cycle of $\bar{\alpha}$ belongs to $\cF_1$
but the $f$-cycle of $\alpha$ does not belong to $\cF_1$.
By construction, they are not loops fixed by $f$, and the $f$-cycles of $\alpha$
and $\bar{\alpha}$ do not contain parts of  double arrows.

Applying now
Proposition~\ref{prop:9.8} (and Lemma~\ref{lem:9.9})
we chose representatives of arrows
$\alpha$, $f(\alpha)$, $f^2(\alpha)$
such that the three products along the cycle
of $f$ lie in the second socle.
Now we take $\cF_2$ as the union of $\cF_1$
and the $f$-cycle of $\alpha$.
If this is $\cF$ then we are done.
Otherwise we repeat the process.
It must stop since $\cF$ is finite.

\medskip

(b) We have now fixed all arrows satisfying (a). This fixes a basis
of $e_iA$ for all $i$ as described in Proposition~\ref{prop:6.10}. With this
we write down the element $\alpha f(\alpha)$ in the second
socle, and by using Lemma \ref{lem:9.13} we get that it has the form as stated.
\end{proof}

We work now with the presentation of $A$
established in
Theorem~\ref{th:10.1}.
For each arrow $\beta$ of $Q$,
we set
$n_{\beta} = |\cO(\beta)|$,
so $\cO(\beta)$ is the cycle
$(\beta \ g(\beta) \ \dots \ g^{n_{\beta} - 1}(\beta))$.

Let $\alpha$ be an arrows of $Q$.
Then it follows from
Theorem~\ref{th:10.1}
that there is a positive integer $m_{\bar{\alpha}}$
such that
\[
    {\alpha} f({\alpha})
     = \bar{\alpha} q_{\bar{\alpha}}
     = c_{\bar{\alpha}} A_{\bar{\alpha}}
\]
for some $c_{\bar{\alpha}} \in K^*$,
if $f(\alpha) \neq \alpha$, and
\[
    \alpha^2 = c_{\bar{\alpha}} A_{\bar{\alpha}} + b_i B_{\bar{\alpha}} ,
\]
with $c_{\bar{\alpha}} \in K^*$ and $b_i \in K$,
if $f(\alpha) = \alpha$,
where
\begin{align*}
    A_{\bar{\alpha}}
      &= \big(\bar{\alpha} g(\bar{\alpha}) \dots g^{n_{\bar{\alpha}} - 1}(\bar{\alpha})
            \big)^{m_{\bar{\alpha}} - 1}
            \bar{\alpha} g(\bar{\alpha}) \dots g^{n_{\bar{\alpha}} - 2}(\bar{\alpha}) ,
    \\
    B_{\bar{\alpha}}
      &= \big(\bar{\alpha} g(\bar{\alpha}) \dots g^{n_{\bar{\alpha}} - 1}(\bar{\alpha})
            \big)^{m_{\bar{\alpha}}}.
\end{align*}
We note that $m_{\bar{\alpha}} n_{\bar{\alpha}} \geq 3$
because $Q$ is the Gabriel quiver of $A$.

We have also the following lemma.

\begin{lemma}
\label{lem:10.2}
For each arrow $\alpha$, we have $c_{\alpha} = c_{g(\alpha)}$.
\end{lemma}

\begin{proof}
Assume first that $f(\alpha) = \alpha$.
Then $g(\alpha) = \bar{\alpha}$,
and it follows from
Proposition~\ref{prop:9.11}
that
$c_{\alpha} = c_{\bar{\alpha}} = c_{g(\alpha)}$.
Assume now that
$f(\alpha) \neq \alpha$
and
$f(\bar{\alpha}) \neq \bar{\alpha}$.
We have
\begin{align*}
  \alpha f(\alpha) f^2(\alpha)
   &= \alpha c_{g(\alpha)} A_{g(\alpha)}
    = c_{g(\alpha)} \alpha A_{g(\alpha)} ,
\\
  \bar{\alpha} f(\bar{\alpha}) f^2(\bar{\alpha})
   &= c_{\alpha} A_{\alpha} f^2(\bar{\alpha})
    = c_{\alpha} \alpha  A_{g(\alpha)}  .
\end{align*}
Since
$\alpha f(\alpha) f^2(\alpha)
  = \bar{\alpha} f(\bar{\alpha}) f^2(\bar{\alpha})$
and non-zero,
we conclude that
 $c_{\alpha} = c_{g(\alpha)}$.
\end{proof}

Therefore, we obtain a parameter function
\[
  c_{\bullet} : \cO(g) \to \bN^* .
\]

Recall also that the border $\partial(Q,f)$ of $(Q,f)$
is the set of vertices $i$ of $Q$ such that there is a loop $\alpha$
at $i$ fixed by $f$.
Hence we have also a border function
\[
  b_{\bullet} : \partial(Q,f) \to K .
\]

Summing up, we proved that $A$ is isomorphic
to the socle deformed weighted triangulation algebra
$\Lambda(Q,f,m_{\bullet},c_{\bullet},b_{\bullet})$.
We note that if $b_{i} = 0$ for any
$i \in \partial(Q,f)$, then $A$ is isomorphic to the
weighted triangulation algebra
$\Lambda(Q,f,m_{\bullet},c_{\bullet})$.
Moreover, by Theorem \ref{th:2.1},
$(Q,f)$ = $(Q(S,\vec{T}),f)$
for a directed triangulated surface $(S,\vec{T})$.
This finishes the proof of implication
(i) $\Rightarrow$ (iii)
of Main Theorem.

\section{Proof of Corollary~ \ref{cor:main2}}\label{sec:proofcor}

Let $n \geq 4$ be a natural number,
$S = K[[x,y]]$ the algebra of formal power series
in two variables over $K$,
and $\mathfrak{m} = (x,y)$ the unique
maximal ideal of $R$.
We choose irreducible power series
$f_1,\dots,f_{n+1} \in \mathfrak{m} \setminus \mathfrak{m}^2$
satisfying the following conditions:
\begin{enumerate}[(1)]
 \item
  $(f_i) \neq (f_j)$ for any $i \neq j$ in $\{1,\dots,n+1\}$;
 \item
  $(f_1,f_2) \neq \mathfrak{m}$,
  $(f_n,f_{n+1}) \neq \mathfrak{m}$,
  and
  $(f_i,f_{i+1}) = \mathfrak{m}$
  for any $i \in \{2,\dots,n-1\}$.
\end{enumerate}

For example, we may take
\begin{align*}
  f_1 &= x,&
  f_2 &= x - \lambda_2 y^2,&
  f_3 &= x - \lambda_3 y,&
  \dots,
\\
  f_{n-1} &= x - \lambda_{n-1} y,&
  f_{n} &= x^2 - \lambda_{n} y,&
  f_{n+1} &= y,
\end{align*}
with
$ \lambda_2, \lambda_3, \dots, \lambda_{n-1}, \lambda_{n}$
pairwise distinct elements of $K^*$.

Let
$R = S/(f_1 \dots f_{n+1})$
and
$T = \oplus_{i=1}^{n+1} S /(f_1 \dots f_i)$.
Then, by general theory \cite{Y},
$R$ is a one-dimensional isolated hypersurface singularity,
and let $\CM(R)$ be the category of maximal
Cohen-Macaulay modules over $R$.
It follows from Eisenbud's matrix factorization theorem
\cite{Ei}
that $\Omega_R^2 = \id$ on the stable category $\underline{\CM(R)}$
of $\CM(R)$.
Consequently, $\underline{\CM(R)}$ is a $2$-CY category
with $\tau_R = \Omega_R$.
Moreover, it follows from
\cite[Theorem~ 4.1]{BIKR}
that $T$ is a cluster tilting object in $\CM(R)$.
Consider the stable endomorphism algebra
\[
  B = \uEnd_R(T) .
\]
Then $B$ is a basic indecomposable finite-dimensional symmetric
algebra, with the Grothendieck group $K_0(B)$ of rank $n$
(see \cite[Lemma~ 2.2]{BIKR}).
Further, it follows from
\cite[Poposition~ 4.10]{BIKR}
and the condition (2) imposed on $f_1,\dots, f_{n+1}$
that the Gabriel quiver $Q_B$ of $B$ is the
$2$-regular quiver of the form
  \[
    \xymatrix{
      1 \ar@(dl,ul)[] \ar@<.5ex>[r] &
      2 \ar@<.5ex>[l] \ar@<.5ex>[r] &
      \cdots \ar@<.5ex>[l] \ar@<.5ex>[r] &
      n-1 \ar@<.5ex>[l] \ar@<.5ex>[r] &
      n \ar@<.5ex>[l] \ar@(ur,dr)[]
    }
    \ \ \, \quad
    .
  \]
Since $\Omega_R^2 = \id$ on $\underline{\CM(R)}$,
applying \cite[Theorem~ 3.2]{Du2},
we obtain that $B$ is a periodic algebra of period $4$.
In particular, we conclude that the stable Auslander-Reiten
quiver $\Gamma_B^s$ of $B$ consists of stable tubes of ranks
$1$ and $2$, because clearly $B$ is representation-infinite.

We claim now that $B$ is a wild algebra.
Suppose (for the contrary) that $B$ is tame.
Then $B$ is a $2$-regular algebra of generalized quaternion type.
Applying
Theorem~\ref{th:6.1},
we obtain that there is a permutation $f$ of arrows of $Q_B$
such that $(Q_B, f)$ is a triangulation quiver.
But it is impossible because $n \geq 4$.
Therefore, indeed $B$ is a wild algebra.
This finishes the proof of Corollary~ \ref{cor:main2}.

\section*{Acknowledgements}

The research was initiated during the visit of the second named author
at the Mathematical Institute in Oxford (March 2014) and
completed during the visit of the first named author at the Faculty of Mathematics and Computer Science
in Toru\'n (June 2017).
The main theorem of the paper was presented
during the conference
``Advances of Representation Theory of Algebras: Geometry and Homology''
(CIRM Marseille-Luminy, September 2017).


\begin{thebibliography}{99}


\bibitem{ASS}
  {I.~Assem, D.~Simson, A.~Skowro\'nski},
  {Elements of the Representation Theory of Associative Algebras 1:
  Techniques of Representation Theory},
  {London Mathematical Society Student Texts, vol. {65}},
  Cambridge University Press, Cambridge, 2006.

\bibitem{AR}
{M.~Auslander, I.~Reiten},
$D\Tr$-periodic modules and functors,
{Representation Theory of Algebras},
in: Canad. Math. Soc. Conf. Proc., vol. {18},
Amer. Math. Soc., 1996, pp. 39--50.

\bibitem{BES1}
{J.~Bia\l kowski, K.~Erdmann, A.~Skowro\'nski},
Deformed preprojective algebras of generalized Dynkin type,
{Trans. Amer. Math. Soc.}  {359}  (2007) 2625--2650.

\bibitem{BES2}
{J.~Bia\l kowski, K.~Erdmann, A.~Skowro\'nski},
Periodicity of self-injective algebras of polynomial growth,
{J. Algebra} {443} (2015) 200--269.

\bibitem{Bu}
{R.~O.~Buchweitz},
Finite representation type and periodic Hochschild (co-)homology,
in {Trends in the Representation Theory of Finite-Dimensional Algebras},
in: Contemp. Math., vol. {229},
Amer. Math. Soc., Providence, RI, 1998, pp.  81--109.

\bibitem{BIKR}
{I.~Burban, O.~Iyama, B.~Keller, I.~Reiten},
Cluster tilting for one-dimensional hypersurface singularities,
{Adv. Math.}  {217}  (2008)  2443--2484.

\bibitem{Ca}
{S.~C.~Carlson},
{Topology of Surfaces,
Knots and Manifolds,
A First Undergraduate Course},
John Wiley \& Sons, Inc., New York, 2001.

\bibitem{CB}
{W.~Crawley-Boevey},
On tame algebras and bocses,
{Proc. London Math. Soc.}  {56}  (1988)  451--483.

\bibitem{DWZ1}
 {H.~Derksen, J.~Weyman, A.~Zelevinsky},
 Quivers with potentials and their representations I.  Mutations,
 {Selecta Math. (N.S.)}  {14}  (2008) 59--119.

\bibitem{DWZ2}
 {H.~Derksen, J.~Weyman, A.~Zelevinsky},
 Quivers with potentials and their representations II:
 Applications to cluster algebras,
 {J. Amer. Math. Soc.}  {23}  (2010)  749--790.


\bibitem{DS1}
{P. Dowbor, A. Skowro\'nski},
On Galois coverings of tame algebras,
{Arch. Math. (Basel)  }{44}  (1985)  522--529.

\bibitem{DS2}
{P.~Dowbor, A.~Skowro\'nski},
On the representation type of locally bounded categories,
Tsukuba J. Math.  10  (1986) 63--72.

\bibitem{Dr}
{Y.~A.~Drozd},
Tame and wild matrix problems,
in: {Representation Theory II},
in: Lecture Notes in Math., vol. {832},
Springer-Verlag, Berlin-Heidelberg, 1980, 242--258.

\bibitem{Du1}
{A.~Dugas},
Periodic resolutions and self-injective algebras of finite type,
{J. Pure Appl. Algebra}  {214}  (2010)   990--1000.

\bibitem{Du2}
{A.~Dugas},
Periodicity of $d$-cluster-tilted algebras,
{J. Algebra}  {368}  (2012) 40--52.

\bibitem{Ei}
D. Eisenbud,
Homological algebra on a complete intersection,
with an application to group representations,
Trans. Amer. Math. Soc. 260 (1980) 35--64.

\bibitem{E1}
{K.~Erdmann},
On the number of simple modules of certain tame blocks and algebras,
{Arch. Math. (Basel)} {51} (1988) 34--38.
%
\bibitem{E2}
{K.~Erdmann},
Algebras and quaternion defect groups I,
{Math. Ann.} {281} (1988) 545--560.
%
\bibitem{E3}
{K.~Erdmann},
Algebras and quaternion defect groups II,
{Math. Ann.} {281} (1988) 561--582.
%
\bibitem{E4}
{K.~Erdmann},
{Blocks of Tame Representation Type and Related Algebras},
in: Lecture Notes in Math., vol. {1428},
Springer-Verlag, Berlin-Heidelberg, 1990.
%
\bibitem{ES1}
{K.~Erdmann, A.~Skowro\'nski},
The stable Calabi-Yau dimension of tame symmetric algebras,
{J. Math. Soc. Japan}  {58}  (2006)   97--128.
%
\bibitem{ES2}
{K.~Erdmann, A.~Skowro\'nski},
Periodic algebras,
in: {Trends in Representation Theory of Algebras and Related Topics},
in: Europ. Math. Soc. Series Congress Reports,
European Math. Soc., Z\"urich, 2008, pp. 201--251.
%
\bibitem{ES3}
{K.~Erdmann, A.~Skowro\'nski},
The periodicity conjecture for blocks of group algebras,
{Colloq. Math.} {138} (2015) 283--294.

\bibitem{ES4}
{K.~Erdmann, A.~Skowro\'nski},
Weighted surface algebras,
Preprint 2017,
{http://arxiv.org/abs/1703.02346}.
{arXiv:1703.02346}.
%
\bibitem{ES5}
{K.~Erdmann, A.~Skowro\'nski},
Higher tetrahedral algebras,
Preprint 2017,
{http://arxiv.org/abs/1706.04477}.
{arXiv:1706.04477}.
%
\bibitem{ES6}
{K.~Erdmann, A.~Skowro\'nski},
Algebras of generalized dihedral type,
Preprint 2017,
{http://arxiv.org/abs/1706.00688}.
{arXiv:1706.00688}.
%
\bibitem{ES7}
{K.~Erdmann, A.~Skowro\'nski},
From Brauer graph algebras to biserial weighted surface algebras,
Preprint 2017,
{http://arxiv.org/abs/1706.07693}
{arXiv:1706.07693}.
%
\bibitem{ESn}
{K.~Erdmann, N.~Snashall},
Preprojective algebras of Dynkin type, periodicity and the second Hochschild cohomology,
in: {Algebras and Modules II},
in: Canad. Math. Soc. Conf. Proc., vol. {24},
Amer. Math. Soc., Providence, RI, 1998, pp. 183--193.
%
\bibitem{FST}
{S.~Fomin, M.~Shapiro, D.~Thurston},
Cluster algebras and triangulated surfaces. Part I. Cluster complexes,
{Acta Math.}  {201}  (2008) 83--146.
%
\bibitem{GLS}
{C.~Geiss, D.~Labardini-Fragoso, J.~Schr\"oer},
The representation type of Jacobian algebras,
Adv. Math. {290} (2016) 364--452.
%
\bibitem{Ha1}
{D.~Happel},
Triangulated Categories in the Representation Theory of Finite-Dimensional Algebras,
{London Math. Soc. Lect. Note Ser.}, vol. {119},
Cambridge University Press, Cambridge, 1988.
%
\bibitem{HR}
{D.~Happel, C.~M.~Ringel},
The derived category of a tubular algebra,
in: {Representation Theory I},
in: Lecture Notes in Math., vol. {1177},
Springer-Verlag, Berlin-Heidelberg, 1986, pp. 156--180.
%
\bibitem{H}
{A.~Hatcher},
{Algebraic Topology},
Cambridge University Press, Cambridge, 2002.
%
\bibitem{Ho}
{T.~Holm},
Derived equivalence classification of algebras of dihedral,
semidihedral, and quaternion type,
{J. Algebra}  {211}  (1999)  159--205.
%
\bibitem{KC}
{A.~Katok, V.~Climenhaga},
{Lectures on Surfaces.
(Almost) Everything You Wanted to Know About Them},
in: Student Math. Library, vol. {46},
Amer. Math. Soc., University Park, PA, 2008.
%
\bibitem{Ki}
{L.~C.~Kinsey},
{Topology of Surfaces},
Undergraduate Texts in Math.,
Springer-Verlag, Berlin-Heidelberg, 1993.


\bibitem{LF}
{D.~Labardini-Fragoso},
Quivers with potentials associated to triangulated surfaces,
{Proc. Lond. Math. Soc.}  {98}  (2009)  797--839.
%
\bibitem{La3}
{S.~Ladkani},
Algebras of quasi-quaternion type,
preprint 2014, {http://arxiv.org/abs/1404.6834}. arXiv:1404.6834.

\bibitem{Lad2}
{S.~Ladkani},
From groups to clusters,
in: Representation Theory -- Current Trends and Perspectives, 
in: Europ. Math. Soc. Ser. Congr. Rep., 
European Math. Soc., Z\"urich, 2017, pp. 427--500.


\bibitem{L}
P.~Lampe,
Diophantine equations via cluster transformations,
J. Algebra  462  (2016)  320--337.

\bibitem{M}
A.~Markoff,
Sur les formes quadratiques binaires ind\'efinies,
Math. Ann.  15 (3-4)  (1879)  381--406.


\bibitem{NeS}
{J.~Nehring, A.~Skowro\'nski},
Polynomial growth trivial extensions of simply connected algebras,
{Fund. Math.} {132}  (1989)  117--134.
%
\bibitem{Ric2}
{J.~Rickard},
Derived equivalences as derived functors,
{J. London Math. Soc.}  {43}  (1991)   37--48.

\bibitem{R}
{C.~M.~Ringel},
{Tame Algebras and Integral Quadratic Forms},
in: Lecture Notes in Math., vol. {1099},
Springer-Verlag, Berlin-Heidelberg, 1984.


\bibitem{SS2}
  D.~Simson, A.~Skowro\'nski,
  {Elements of the Representation Theory of Associative Algebras 3:
  Representation-Infinite Tilted Algebras},
  London Mathematical Society Student Texts, vol. {72},
  Cambridge University Press, Cambridge, 2007.
  

\bibitem{Sk0}
{A.~Skowro\'nski},
Group algebras of polynomial growth,
{Manuscripta Math.}  {59}  (1987)  499--516.
\bibitem{Sk1}
{A.~Skowro\'nski},
Selfinjective algebras of polynomial growth,
{Math. Ann.}  {285}  (1989)   177--199.

\bibitem{Sk}
{A.~Skowro\'nski},
Selfinjective algebras: finite and tame type,
in: {Trends in Representation Theory of Algebras and Related Topics},
in: Contemp. Math., vol. {406},
Amer. Math. Soc., Providence, RI, 2006, pp.  169--238.


\bibitem{SY}
{A.~Skowro\'nski, K.~Yamagata},
{Frobenius Algebras I. Basic Representation Theory},
European Mathematical Society Textbooks in Mathematics,
European Math. Soc. Publ. House, Z\"urich, 2011.

\bibitem{SY2}
{A.~Skowro\'nski, K.~Yamagata},
{Frobenius Algebras II. Tilted and Hochschild Extension Algebras},
European Mathematical Society Textbooks in Mathematics,
European Math. Soc. Publ. House, Z\"urich, 2017.

\bibitem{Sw}
{R.~G.~Swan},
Periodic resolutions for finite groups,
{Ann. of Math.}  {72}  (1960) 267--291.

\bibitem{VD}
{Y.~Valdivieso-D\'iaz},
Jacobian algebras with periodic module category and exponential growth, 
J. Algebra 449 (2016) 163--174.

\bibitem{Y}
Y. Yoshino,
Cohen-Macaulay modules over Cohen-Macaulay rings,
in: London Mathematical Society Lecture Note Series, vol.146,
Cambridge University Press, Cambridge, 1990.

\end{thebibliography}
\end{document}